\documentclass[11pt]{amsart}
\usepackage{amscd,amssymb,palatino}
\usepackage[utf8]{inputenc}
\usepackage[english]{babel}
\usepackage{caption}
\usepackage[leqno]{amsmath}
\usepackage{amssymb}
\usepackage{subcaption}
\usepackage{color}
\usepackage{shadow}
\usepackage{epsfig}
\usepackage{epic}
\usepackage{eepic}
\usepackage{graphics}
\usepackage{graphicx}
\usepackage{psfrag}
\usepackage{hyperref}
\usepackage{array}
\usepackage{calc}
\usepackage{enumitem}
\usepackage[table]{xcolor}

\usepackage{tikz}
\usepackage[dvipsnames,prologue,table]{pstricks}
\usetikzlibrary{arrows,arrows.meta,decorations,patterns,positioning,automata,shadows,fit,shapes,calc,decorations.markings,backgrounds,scopes,decorations.text}

\usepackage{tipa}

\usepackage{fancyhdr}
\fancypagestyle{plain}{\fancyhf{}
\fancyfoot[C]{{\large\sf\bfseries\thepage}}}
\DeclareMathOperator{\Id}{Id}

\DeclareMathOperator{\grad}{grad}

\usepackage{calc}
\numberwithin{equation}{section}

\newcommand\1{\lower 9pt\hbox{\underbar{}}}
\numberwithin{equation}{section}
\newtheorem*{theorem*}{Main Theorem}
\newtheorem {theorem}{Theorem}[section]
\newtheorem {lemma}[theorem]           {Lemma}
\newtheorem {proposition}[theorem]     {Proposition}
\newtheorem {conjecture}[theorem]      {Conjecture}

\newtheorem{thm}{Theorem}

\newtheorem {corollary}[theorem]       {Corollary}

\theoremstyle{definition}
\newtheorem {definition}[theorem]{Definition}
\newtheorem {remark}[theorem]          {Remark}

\newtheorem {example}[theorem]         {Example}
\newtheorem {openquestion}[theorem]{Open Question}

\newcommand{\pr} {\smallskip\noindent{\bf Proof\,\,}}

\begin{document}

\title{The Arnold conjecture for singular symplectic manifolds}

\author{Joaquim Brugués}
\thanks{ All the authors are partially supported by the Spanish State Research Agency grant PID2019-103849GB-I00 of AEI / 10.13039/501100011033 and by the AGAUR project 2021 SGR 00603. }
\thanks{J. Brugués is supported by the FWO-FNRS Excellence of Science project G0H4518N "Symplectic techniques in differential geometry" with UA Antigoon Project-ID 36584}

\address{University of Antwerp, Department of Mathematics. Middelheim Campus - Building G. Middelheimlaan 1, 2020 Antwerp, Belgium}
\email{joaquim.bruguesmora@uantwerpen.be}

\author{Eva Miranda}
\thanks{  E. Miranda  is supported by the Catalan Institution for Research and Advanced Studies via an ICREA Academia Prize 2021 and by the Alexander Von Humboldt foundation via a Friedrich Wilhelm Bessel Research Award. E. Miranda is also supported by the Spanish State
Research Agency, through the Severo Ochoa and Mar\'{\i}a de Maeztu Program for Centers and Units
of Excellence in R\&D (project CEX2020-001084-M). }

\address{{Laboratory of Geometry and Dynamical Systems \& IMTech, Department of Mathematics}, Universitat Polit\`{e}cnica de Catalunya and CRM, Barcelona, Spain \\ Centre de Recerca Matemàtica-CRM}
\email{eva.miranda@upc.edu}

\author{Cédric Oms}
\thanks{C. Oms acknowledges financial support from the Margarita Salas postdoctoral contract 
financed by the European Union-NextGenerationEU and is partially supported by the ANR grant ``Cosy" (ANR-21-CE40-0002), partially supported by the ANR grant ``CoSyDy" (ANR-CE40-0014).}

\address{{Laboratory of Geometry and Dynamical Systems \& IMTech, Department of Mathematics}, Universitat Polit\`{e}cnica de Catalunya and BCAM Bilbao, Mazarredo, 14. 48009 Bilbao Basque Country - Spain}
\email{coms@bcamath.org}

\begin{abstract} 
In this article, we study the Hamiltonian dynamics on singular symplectic manifolds and prove the Arnold conjecture for a large class of $b^m$-symplectic manifolds.  Novel techniques are introduced to associate smooth symplectic forms to the original singular symplectic structure, under some mild conditions. These techniques yield the validity of the Arnold conjecture for singular symplectic manifolds across multiple scenarios.
	More precisely, we prove a lower bound on the number of 1-periodic Hamiltonian orbits for $b^{2m}$-symplectic manifolds depending only on the topology of the manifold.
	Moreover, for $b^m$-symplectic surfaces, we improve the lower bound depending on the topology of the pair $(M,Z)$.
	We then venture into the study of Floer homology to this singular realm and we conclude with a list of open questions. 
\end{abstract}
\maketitle

\section{Introduction}

Symplectic structures on manifolds with boundary \cite{nestandtsygan} led naturally to the investigation of a class of Poisson manifolds which are symplectic manifolds away from a hypersurface but degenerate along this hypersurface (see \cite{guimipi} and \cite{Gualtierili}\footnote{Other contributions are included in  \cite{mirandaplanas, guimipi2,gmps, gmps2, gmW1,gmw,  gmw2,km, kms, cavalcanti, marcutosorno1, marcutosorno2}}). In the literature, these structures are called $b^m$- or $\log$-symplectic manifolds. They also show up in the space of geodesics of the Lorentz plane \cite{khesintabachnikov} and furnish a natural phase space for regularized problems in celestial mechanics such as the restricted 3-body problem \cite{km, kms, dkm, bcontact, singularweinstein, invitation}.
Another example where Hamiltonian dynamics on $b$-symplectic manifolds is outstanding is in the context of Painlevé transcendents (\cite{anastasia}, \cite{anastasiaeva}).

The investigation of the dynamics on the odd-dimensional sibling of these singular manifolds started in \cite{singularweinstein}, \cite{escapeorbits}, and \cite{2Nescapeorbits}. In this article, we initiate the exploration of singular Hamiltonian dynamics on the even-dimensional counterpart and compare it to Hamiltonian dynamics for symplectic manifolds. In order to do so, we adopt the revolutionary approach due to Andreas Floer to prove the Arnold conjecture on compact symplectic manifolds.

\begin{conjecture}[Arnold conjecture]
Let $(M,\omega)$ be a compact symplectic manifold and let $H:\mathbb{R}\times M \to \mathbb{R}$ be a time-dependent Hamiltonian function. Suppose that the solutions of period $1$ of
	the associated Hamiltonian system are non-degenerate. Let $\mathcal{P}(H)$ denote the set of 1-periodic orbits. Then
	$$\# \mathcal{P}(H) \geq \sum_i \dim H_i (M;\mathbb{Z}_2).$$
\end{conjecture}

To tackle this long-standing conjecture, Floer defined a homology whose generators are the Hamiltonian periodic orbits and showed that this homology is isomorphic to the standard homology of the manifold $M$, thus depending only on the topology of $M$. Floer's pioneering work paved the way for a novel set of methodologies that have proven to be effective in establishing the validity of the Arnold conjecture across various contexts (we direct the reader to \cite[Section~1.1]{salamon} for a comprehensive chronology of the work on the Arnold conjecture and references therein).

In this paper, we lay the groundwork for adapting  Hamiltonian Floer theory to $b$-symplectic manifolds. This will allow us to apply the resulting theory to many natural examples endowed with a \emph{singular} symplectic structure in celestial mechanics, Lorentzian geometry and the study of Painlevé transcendents. The geometry on these manifolds can be described as open symplectic manifolds equipped with a cosymplectic structure on the open ends. In a way, our theory provides a generalization of Floer theory to non-compact symplectic manifolds but it is different from the theory of tentacular Hamiltonians in \cite{tentacular1, tentacular2}.  The cosymplectic structure on the open ends is determined by a symplectic vector field that is transverse to the boundary. To avoid compactness issues in the definition of the Floer complex, the set of Hamiltonian functions needs to be restricted.

This is reminiscent of symplectic manifolds with convex boundaries: for them, the Hamiltonian dynamics was thoroughly studied in \cite{convexsymplectic}. In order to define a Floer theory for those open symplectic manifolds, special care needs to be taken at the boundary to have a well-defined complex.

In the class of $b^m$-symplectic manifolds, we are interested in periodic Hamiltonian orbits away from the critical set.
The $b^m$-symplectic structure induces on the critical set a codimension one symplectic foliation, known as a cosymplectic structure.
When the leaves of this symplectic foliation are compact and the Hamiltonian function is smooth, the Hamiltonian vector field restricts to a smooth Hamiltonian vector field on each leaf and therefore there are infinitely many periodic orbits on the leaves.
The existence of periodic orbits away from the critical set, however, is much more subtle.
To study this question, we will argue that the appropriate Hamiltonian functions are the ones that are linear along the normal symplectic vector field and preserved by the Reeb vector field given by the cosymplectic structure. Those Hamiltonian functions are called admissible Hamiltonian functions and they are not smooth functions (as they present log-terms or higher order singularities) but they belong to the class of functions naturally associated with $b^m$-symplectic manifolds, the $b^m$-Hamiltonian functions.

In the case of $b^{2m}$-symplectic manifolds (without any restriction on the dimension), we can bound from below the number of time-$1$ periodic orbits of such Hamiltonian vector fields by the topology of the ambient manifold. This bound is achieved by proving that its associated dynamics is indeed symplectic for a \emph{smooth} geometric structure, called the desingularized symplectic structure. This desingularization process was previously introduced and studied in \cite{gmW1}. In the case where the Hamiltonian satisfies a condition of global nature (i.e. being unimodular, see Definition \ref{def:constantlinearweight}), or when the $b$-manifold satisfies the condition of being acyclic, the dynamics associated to an admissible Hamiltonian is Hamiltonian (this is the content of Proposition \ref{prop:b2mHam=Ham} and \ref{prop:cyclicb2m}). As a corollary (Corollary \ref{thm:b2msymplecticArnold}), we obtain that the number of non-degenerate periodic orbits is bounded from below by the topology of the manifold, as given by the Arnold conjecture.

\begin{thm}[Arnold conjecture for $b^{2m}$-symplectic manifolds] \label{coro:b2marnold_acyclic}\label{thmA}
		Let $(M,Z,\omega)$ be a compact $b^{2m}$-symplectic manifold, and let $H_t\in {^{b^{2m}}}C^\infty(M)$ be a time-dependent admissible Hamiltonian function. Assume that either the $b$-manifold is acyclic or that the Hamiltonian is unimodular (see Definition \ref{def:constantlinearweight}).
		Suppose that the solutions of period $1$ of the associated Hamiltonian system are non-degenerate.
		Then
		\[\# \mathcal{P}(H) \geq \sum_i \dim HM_i(M;\mathbb{Z}_2).\]
\end{thm}

	We remark that for manifolds with trivial $H^1(M)$, the condition of having a globally defining function defining $Z$ is automatically satisfied. When the defining function is not globally defined, the associated Hamiltonian flow is symplectic and we therefore obtain a weaker corollary in this case (Corollary \ref{coro:b2marnold_cyclic}).

 {For $b^{2m+1}$-symplectic structure, this strategy breaks down, as the desingularization as introduced in \cite{gmW1} associates  a  \emph{folded} symplectic structure  to a $b^{2m+1}$-structure instead of a symplectic structure.} 

{In this article, we demonstrate that the process of desingularization can be enhanced in various scenarios. Specifically, we prove that for surfaces and higher-dimensional $b^{2m+1}$-symplectic manifolds with critical sets that are \emph{trivial mapping tori}, a natural topological condition on $b^{2m+1}$-symplectic manifolds\footnote{ The critical set of a $b^m$-symplectic manifold with integral $b^m$-cohomology class is a  trivial mapping tori, as shown in \cite{gmw, gmwq2}}, it is still feasible to desingularize the singular symplectic structure into a symplectic one. Additionally, we show that the $b^m$-Hamiltonian dynamics can be understood as the smooth Hamiltonian dynamics of this new desingularized structure.} 

Building on these methods, we establish a sharper lower bound for orientable $b^m$-symplectic surfaces. This lower bound depends just on the topology of the surface and the relative position of the critical curves.
The lower bound holds regardless of the parity of $m$ and the existence of cycles in the associated graph.

Hence, we obtain the following,

\begin{thm}\label{thmB}
		Let $(\Sigma,Z,\omega)$ be a compact $b^m$-symplectic orientable surface.
		Let $H_t$ be an admissible Hamiltonian in $(\Sigma, Z, \omega)$ whose periodic orbits are all non-degenerate.
Then the number of $1$-periodic orbits of $X_H$ is bounded from below by 
	\[\# \mathcal{P}(H) \geq \sum_{v\in V} \big(2g_v + |\mathrm{deg}(v) - 2| \big),\]
	where $v$ is a connected component of $\Sigma \setminus Z$, $g_v$ denotes its genus and $\mathrm{deg}(v)$ is its degree as a vertex in the graph.
	\end{thm}

Under the admissibility condition of the Hamiltonian functions, we show that the associated Floer solutions satisfy a minimal principle:

\begin{thm}\label{thmC}
    Let $(M,Z,\omega)$ be a $b^m$-symplectic manifold and denote by $\mathcal{N}$ is a tubular neighbourhood of $Z$ in $M \setminus Z$. Let $f : \mathcal{N} \to \mathbb{R}$ be the function given by $\log|z|$ if $m = 1$ and $- \frac1{(m-1)z^{m-1}}$ if $m > 1$ and let $u: \Omega \subset \mathbb{C} \to \mathcal{N} \subset M$ be a solution of the Floer equation for an admissible Hamiltonian $H \in \mathcal{C}^{\infty}(S^1\times \mathcal{N})$.
			If $f\circ u$ attains its maximum or minimum on $\Omega$, then $f\circ u$ is constant.
\end{thm}

From this result we deduce the compactness of the space of solutions. We are thus able to introduce a well-defined  Hamiltonian Floer-type homology on $b^m$-symplectic manifolds.

Our definition of Hamiltonian Floer theory on $b^m$-symplectic manifold together with the lower bounds in Theorem \ref{thmA} and \ref{thmB},  gives rise to new open questions which we describe at the end of this article.
	
\subsection*{Organization of the paper} 

In Section \ref{sec:prel} we include a brief introduction to $b^m$-symplectic geometry and Floer homology. We also include some known results concerning symplectic manifolds with convex boundaries, which are in some sense at the opposite end of the spectrum to $b^m$-symplectic manifolds. Next, we show in Section \ref{sec:dyn} some preliminary results concerning the Hamiltonian dynamics around the critical set of $b^m$-symplectic manifolds. This section will serve as a motivation to study a particular class of Hamiltonian functions on $b^m$-symplectic manifolds. We then revisit the desingularization in Section \ref{sec:desing} and study the effect of this technique on the dynamics of Hamiltonian vector fields. In Section \ref{section:arnold} we prove in Theorem \ref{thmA} and Theorem \ref{thmB} a lower bound for the number of periodic orbits for this class of Hamiltonian functions. The Hamiltonian functions are suitable to define a Hamiltonian Floer type homology on $b^m$-symplectic manifolds, as is seen in Section \ref{sec:floer}. We conclude the article by presenting a list of open questions in Section \ref{sec:final}, which we intend to explore in future work. We include an appendix, proving that the lower bound for in Theorem \ref{thmB} is optimal.

\subsection*{Acknowledgements} 

 We are grateful to Barney Bramham,   Urs Frauenfelder, Sonja Hohloch, Charlotte Kirchhoff-Lukat, Marco Mazzuchelli, Federica Pasquotto, Leonid Polterovich, Jonathan Weitsman, Jagna Wiśniewska, and Aldo Witte for many enlightening conversations. This new version owes much to colleagues who provided invaluable insights. Special thanks to Juanjo Rue Perna for suggesting a graph theoretical approach in the proof of the appendix 2, and Pablo Nicolás for spotting a critical notational mismatch. Gratitude extends to the thorough feedback from the anonymous referee. Thank you to the organizers of \emph{Symplectic Dynamics beyond periodic orbits} and the Lorenz Center in Leiden for their hospitality in August 2022. Certainly, the workshop and conversations with several participants gave us food for thought and inspiration. We are thankful to Hansjörg Geiges and the University of Köln for their hospitality during the final days of the battlefield that led to this article. A special thanks goes to Juanjo Rué for helpful discussions concerning graph theory. The second author is indebted to the Alexander Von Humboldt Foundation for the Friedrich Wilhelm Bessel Research Award and the opportunity to enjoy a nice working atmosphere at the University of Köln.

\section{Preliminaries}\label{sec:prel}

In this section, we provide a summary of the background material required for our setting.
We begin by giving an outline of the proof of the Arnold conjecture in the case of closed and aspherical symplectic manifolds using Floer theory, following mostly the contents of \cite{audin} and \cite{salamon}.

After that we will collect the basic results on $b$- and $b^m$-symplectic manifolds that we will need.

Finally, we will reproduce shortly the construction developed in \cite{convexsymplectic}, which has been an important inspiration for our version of the Floer complex on $b^m$-symplectic manifolds.

\subsection{Basics on Floer homology}

Let $(M,\omega)$ be a compact symplectic manifold.
Let us assume that the first Chern class of $M$ vanishes on $\pi_2(M)$ and also that $\left\langle [\omega],\pi_2(M)\right\rangle=0$.

Let us introduce informally the domain on which we will work.
For a formal introduction we refer the reader to \cite[Section~6.8]{audin}.

\begin{definition}
	Let $p > 1$ and a Riemannian metric $g$ on $M$.
	The \emph{space of loops} of $M$, denoted by $\mathcal{L}^{1,p}M$, is a subset of $\mathcal{C}^0(S^1; M)$ given by the exponentials (with respect to $g$) of sections of the Banach bundle $W^{1,p}(x^{\ast}TM)$, where $x \in \mathcal{C}^{\infty}(S^1, M)$.
\end{definition}

\begin{theorem}[{\cite[Theorem~10]{schwarz}}]
	$\mathcal{L}^{1,p}M$ has a structure of smooth manifold for all $p > 1$ which does not depend on $g$.
	Moreover, $\mathcal{C}^{\infty}(S^1; M) \subset \mathcal{L}^{1,p}M \subset \mathcal{C}^0(S^1; M)$, where each inclusion is dense in the next one.
\end{theorem}

From now on, we will denote by $\mathcal{L}M$ the subspace of $\mathcal{L}^{1,2}M$ of loops that are contractible.

\begin{definition}
	Let $H_t: \mathbb{R} \times M \to \mathbb{R}$ be a time-$1$ periodic smooth Hamiltonian.
	The action functional is defined on the space of contractible loops $\mathcal{L}M$ and is given by
	\[\mathcal{A}_{H}(x)=\int_{S^1} H_t(x(t))dt - \int_{D^2} v^{\ast} \omega,\]
	where $v$ is a filling of $x$ within $M$.
\end{definition}

\begin{proposition}[{\cite[Proposition~6.3.3]{audin}}]
	A loop $x$ is a critical point of $\mathcal{A}_{H}$ if, and only if, it is a periodic solution of the Hamiltonian system $x'(t) = X_{H}(x(t))$.
\end{proposition}

\begin{definition}
	Let $x$ a critical point of $\mathcal{A}_H$.
	We say that it is \emph{non-degenerate} if
	\[\mathrm{det}\left( \mathrm{Id} - d_{x(0)} \varphi_{X_{H}}^1 \right) \neq 0 ,\]
	this means, the operator $d_{x(0)} \varphi_{X_{H}}^1$ does not have the eigenvalue $1$.
	By $\varphi_{X_{H}}^1$ here we mean the time $1$ flow of the Hamiltonian vector field $X_{H}$.

	We say that a Hamiltonian $H_t$ is \emph{regular} if all of the $1$-periodic solutions of $X_{H}$ are non-degenerate.
	We denote the set of regular Hamiltonian functions by $\mathcal{H}_{\text{reg}}$.
\end{definition}

The set of regular Hamiltonian functions is generic, in the sense that it is a countable intersection of open and sets dense in the $\mathcal{C}^{\infty}$-topology.

The Floer complex consists of the graded $\mathbb{Z}_2$-vector space $CF_*(M,H)$ generated by the critical points of $\mathcal{A}_H$ and graded by an index called the Conley-Zehnder index (see for instance \cite[Section~2.4]{salamon}).

\begin{definition}\label{def:acs}
	An \emph{almost complex structure} is an endomorphism $J \in \mathrm{End}(TM)$ such that $J^2 = - \mathrm{Id}$.
	We say that an almost complex structure is \emph{compatible} with a symplectic form $\omega$ if $\omega(JX, JY) = \omega(X,Y)$ for all $X, Y \in \mathfrak{X}(M)$ and if the bilinear map $(u,v) \mapsto \omega_p(u, J(p) v)$ is positive definite for all $p \in M$(where $u, v \in T_pM$).

	If $(M, \omega)$ is a symplectic manifold and $J$ is a compatible almost complex structure, we define the \emph{compatible Riemannian metric} on $M$ by $g_p(u,v) = \omega_p(u, J(p) v)$ for $p \in M$, $u, v \in T_pM$.
\end{definition}

\begin{definition}
	\label{definition:induced_metric}
	Let $M$ a compact manifold and $g$ a Riemannian metric.
	We say that the \emph{metric induced on $\mathcal{L}M$} is the Riemannian metric given by
	\[\langle \xi, \zeta \rangle_{x} := \int_0^1 g_{x(t)} \left( \xi(t), \zeta(t) \right) dt\]
	for all $\xi, \zeta \in T_x \mathcal{L}M$.
\end{definition}

\begin{definition}
	Let $(M, \omega, J)$ a compact symplectic manifold with a compatible almost complex structure.
	Consider the contractible loop space $\mathcal{L}M$ equipped with the metric introduced in Definition \ref{definition:induced_metric}.

	The \emph{Floer equation} is the negative gradient flow equation of the action functional $\mathcal{A}_H$ with respect to the induced metric.
	If we take $u : \mathbb{R} \to \mathcal{L}M$ smooth, the equation has the expression
	\begin{equation}\label{eq:floer}
		\frac{\partial u}{\partial s}+J(u)\frac{\partial u}{\partial t}+\grad H_t(u) = 0 .
	\end{equation}

	If $u$ is a solution to the Floer equation \ref{eq:floer}, we say that its \emph{energy} is the map $E : \mathcal{C}^{\infty}(\mathbb{R}; \mathcal{L}M) \to \mathbb{R}$ given by
	\[E(u) := \int_{\mathbb{R}} u^{\ast} d \mathcal{A}_H .\]

	We denote the set of \emph{finite energy solutions} of the Floer equation by
	\begin{equation}\label{eq:finiteenergysolutions}
		\mathcal{M} := \left\{ u \in \mathcal{C}^{\infty}(\mathbb{R}; \mathcal{L}M) \ | \ u \text{ is a solution of \ref{eq:floer} and } E(u) < + \infty \right\}.
	\end{equation}

	Moreover, the \emph{moduli space of solutions from $x$ to $y$}, where $x,y$ are $1$-periodic Hamiltonian orbits, is the set
	\begin{equation}
		\mathcal{M}(x,y) := \left\{ u \in \mathcal{M} \ \Big| \ \lim_{s \to -\infty} u(s,\cdot) = x, \ \lim_{s \to +\infty} u(s,\cdot) = y \right\} .
	\end{equation}
\end{definition}

The properties of $\mathcal{M}$ were studied by Floer using the techniques pioneered by Gromov in his analysis of pseudo-holomorphic curves.
In the particular context of the Floer equation, very similar methods are used to show the following theorem:

\begin{theorem}[{\cite[Theorems 6.5.4 and 6.5.6]{audin}}]\label{thm:Mcompact}
	The set $\mathcal{M}$ is compact in $\mathcal{C}_{\mathrm{loc}}^{\infty}(\mathbb{R} \times S^1; M)$.
	Moreover, for any $u \in \mathcal{M}$, there exist two critical points of $\mathcal{A}_H$, $x$ and $y$, such that
	\[\lim_{s\rightarrow -\infty} u(s,\cdot) = x, \ \lim_{s\rightarrow +\infty} u(s,\cdot) = y\]
	in $\mathcal{C}^{\infty}(S^1,M)$, and
	\[\lim_{s\rightarrow \pm \infty}\frac{\partial u}{\partial s}(s,\cdot) = 0. \]

	Therefore, $\mathcal{M} = \bigcup_{x,y\in \mathrm{Crit}(\mathcal{A}_H)} \mathcal{M}(x,y)$.
\end{theorem}

	In order the define a boundary map for the Floer complex one must study the sets $\mathcal{M}(x,y)$.
	It is possible to show that each $\mathcal{M}(x,y)$ is a finite dimensional smooth manifold.
To do so one studies the linear approximation of the Floer operator acting on the set of perturbations of elements of $\mathcal{M}$.
The space of such perturbations is defined as
\[\mathcal{P}^{1,p}(x,y) = \left\{P : (s,t) \mapsto \mathrm{exp}_{u(s,t)} Y(s,t) \ | \ u \in \mathcal{M}(x,y) , Y\in W^{1,p}(u^{\ast}(TM))\right\}.\]

Here $W^{1,p}(u^{\ast}(TM))$ denotes the set of maps $Y : \mathbb{R}\times S^1 \rightarrow TM$ such that $\pi \circ Y = u$ (where $\pi : TM \rightarrow M$ is the natural projection of the tangent bundle) and such that their local trivializations belong to the Sobolev space $W^{1,p}$.

Then, the Floer operator is defined as
\[\begin{array}{rccc} \mathcal{F}: & \mathcal{P}^{1,p}(x,y) & \longrightarrow & L^p(\mathbb{R}, S^1) \\ & w & \longmapsto & \frac{\partial w}{\partial s} + J(w) \frac{\partial w}{\partial t} + \mathrm{grad}_w H_t, \end{array}\]
and its linearization has the expression
\[d\mathcal{F}_u(Y) = \frac{\partial Y}{\partial s} + J(u) \frac{\partial Y}{\partial t} + \left(\mathcal{L}_Y J\right)_u \frac{\partial u}{\partial t} + \mathcal{L}_Y (\mathrm{grad}_u H) . \]

One can check (see \cite[Theorem 8.1.5]{audin} or \cite[Theorem 2.2]{salamon}) that $d\mathcal{F}_u$ is a Fredholm map for any $u \in \mathcal{M}(x,y)$ and that it has Fredholm index is $\mathrm{Ind}(d\mathcal{F}_u) = \mu_{CZ}(x) - \mu_{CZ}(y)$, the difference of Conley-Zehnder indices.
	Further, it can be seen that it is a surjective map for any non-degenerate Hamiltonian $H$ and any almost complex structure $J$ compatible with $\omega$.

From this, it is possible to determine the dimension of $\mathcal{M}(x,y)$ as follows:

\begin{theorem}\label{thm:mdimensions}
	For $p > 2$, $\mathcal{F}^{-1}(0)$ is a finite dimensional compact manifold of dimension $\mu_{CZ}(x) - \mu_{CZ}(y)$.
\end{theorem}

\begin{definition}
	Let $x$ and $y$ be two $1$-periodic Hamiltonian orbits.
	The set of \emph{non-parametrized trajectories}, denoted by $\mathcal{T}(x,y)$, is the quotient of the manifold $\mathcal{M}(x,y)$ by the action of $\mathbb{R}$.
\end{definition}

It is possible to conclude (see \cite[Chapter~9]{audin}) that $\mathcal{T}(x,y)$ is Hausdorff.
Moreover, for any pair $(x,y)$ of orbits, $\mathcal{T}(x,y)$ can be compactified into a manifold of dimension $\mu_{CZ}(x) - \mu_{CZ}(y) - 1$, which we denote by $\overline{\mathcal{T}}(x,y)$.

\begin{definition}
	Let $(x,y)$ be a pair of Hamiltonian orbits such that $\mu_{CZ}(x) = \mu_{CZ}(y) + 1$.
	Then, we denote by $n(x,y)$ the cardinality of the zero-dimensional and compact manifold $\overline{\mathcal{T}}(x,y)$, modulo 2.

	Then, for each $k \in \mathbb{N}$ the \emph{boundary map} of the Floer complex is a map $\partial_k : CF_{k+1}(M; H,J) \to CF_k(M; H,J)$, and given by
	\[\partial_k (x) = \sum_{y \in CF_k(M;H,J)} n(x,y) y.\]
\end{definition}

From the established properties it follows that the boundary map is well-defined.

\begin{theorem}[{\cite[Theorem~4]{floer2}}]
	\[\partial_k \circ \partial_{k+1} = 0.\]
\end{theorem}

Thus, the Floer complex $(CF_{\bullet}(M;H,J), \partial_{\bullet})$ is well defined and it induces a homology.
It is clear that we used both $H$ and $J$ to define this complex, so the complex (and therefore the homology) can depend on these choices.
However, as the following theorem shows, this turns out not to be the case for this homology.

\begin{theorem}[{\cite[Theorem~5]{floer2}}, {\cite[Chapter~11]{audin}}] \label{thm:floerinvariance}
	The homology induced by the Floer complex does not depend on the choice of a pair $(H,J)$.
\end{theorem}

Indeed, this complex can be identified with the Morse homology:
\begin{theorem}[{\cite[Theorem~1]{floer2}}, {\cite[Theorem~10.1.1]{audin}}] \label{thm:floerisomorse}
	The Floer homology is isomorphic to the Morse homology,
	\[HF_{\bullet}(M) \cong HM_{\bullet + n}(M),\]
	where $\mathrm{dim}(M) = 2n$.
\end{theorem}

Theorem \ref{thm:floerisomorse} proves that the Floer homology is indeed a topological invariant of an (aspherical) symplectic manifold.
Moreover, it provides explicitly the dimensions of the homology groups, in relation to the groups of the Morse homology.

The power of this result resides in the fact that the groups of the Floer complex may be impossible to compute, rendering an explicit computation of the Floer homology impossible.
However, the theorem allows the Morse inequalities to be translated to this setting without any further effort.

In particular, this proves the Arnold conjecture for the aspherical symplectic case.

\subsection{Geometrical structures on $b^m$-manifolds}

We begin by providing a quick summary of the theory of $b^m$-symplectic manifolds.
For a more detailed development on this topic we direct the reader to \cite{guimipi2,invitation,scott}.

\begin{definition} \label{def:bmanifolds}
	Let $M$ be a compact and connected manifold, and let $Z \subset M$ be an embedded submanifold of codimension 1. We call such pairs {\bf $b$-manifolds}.
	The hypersurface $Z$ is given by the regular zero level-set of a local defining function $z$. When the function $z$ is globally defined, the function will be called a \emph{global} defining function.
	When $M$ is orientable, the existence of the globally defining function implies that $Z$ is coorientable.
	We define the set of $b$-vector fields as
	\[^b\mathfrak{X}(M) := \{v \in \mathfrak{X}(M) \ | \ \left.v\right|_Z \in \mathfrak{X}(Z)\} ,\]
	that is  the set of vector fields on $M$ that are tangent to $Z$.
	This set has the structure of a projective module, so using the Serre-Swan theorem (see for instance \cite{swan1962vector}) we can define the vector bundle whose sections are $^b\mathfrak{X}(M)$, the $b$-tangent bundle $^bTM \to M$.
	
	The $b$-cotangent bundle is its dual, $^bTM^{\ast} := \left(^bTM\right)^{\ast}$. The forms on this vector bundle form a complex, denoted by $^b\Omega^{\bullet}(M)$, which contains $\Omega^{\bullet}(M)$ canonically, and moreover, we can extend the notion of differential to this complex, $d : {^b}\Omega^{\bullet}(M) \to {^b}\Omega^{\bullet + 1}(M)$.

	A {\bf $b$-symplectic manifold} is a $3$-tuple $(M,Z,\omega)$, where $(M,Z)$ is a $b$-manifold and $\omega \in {}^b\Omega^2(M)$ is a closed and non-degenerate $b$-form. We will often omit the critical set $Z$ in the notation.
\end{definition}

	More generally, we can consider singularities such as $b^m$-symplectic manifolds, as was done by \cite{scott}.
	Mimicking the $b$-case we start by introducing the analog of sections of the $b^m$-tangent bundle.

	\begin{definition}
		A \textbf{$b^m$-vector field} is a vector field $v$ on a $b$-manifold $(M,Z)$ \emph{tangent} of order $m$ at $Z$. 
	\end{definition}

	The $b^m$-vector fields  are locally generated by $\{z^m \frac{\partial}{\partial z}, \frac{\partial}{\partial x_2}, \ldots, \frac{\partial}{\partial x_n}\}$ where $z=0$ describes $Z$ and $(z,x_2 \dots, x_n)$ are coordinates on $M$. By the Serre-Swan theorem \cite{swan1962vector}, given a $b$-manifold $(M, Z)$ there exists a unique vector bundle $^{b^m} TM$ whose smooth sections are $b^m$-vector fields. Such a bundle is called the \textbf{$b^m$-tangent bundle}. Analogously, the \textbf{$b^m$-cotangent bundle} can be defined either as dual to the tangent one:
	$$
	^{b^m} T^*M = (^{b^m} TM)^*.
	$$

	We can then introduce the set of $b^m$-forms ${^{b^m}}\Omega^k(M)$ as sections of the exterior product of this bundle  $\bigwedge ^k(^{b^m} T^* M)$.
	The forms on this vector bundle form a complex, denoted by $^{b^m}\Omega^{\bullet}(M)$, which contains $\Omega^{\bullet}(M)$ canonically, and moreover, we can extend the notion of differential to this complex, $d : {^{b^m}}\Omega^{\bullet}(M) \to {^{b^m}}\Omega^{\bullet + 1}(M)$.
	We denote by ${^{b^m}}H^*(M)$ the associated $b^m$-cohomology.
	The zero-degree terms of this complex are the motivation to introduce the definition of $b^m$-functions which we denote as $^{b^m} \mathcal{C}^\infty(M)$ as $$^{b^m} \mathcal{C}^\infty(M) = \left ( \bigoplus \limits_{i = 1}^{m - 1} z^{-i} \mathcal{C}^{\infty}(Z) \right ) \oplus {^{b}}\mathcal{C}^\infty(M)$$ where $^{b}\mathcal{C}^\infty (M)$ denotes the space of $b$-functions (see \cite{gmps}, respectively \cite{gmw2} for the case of $b^m$-functions),
 $$^{b}\mathcal{C}^\infty (M) = \{a \log |z| + g, a\in \mathbb{R}, g\in \mathcal{C}^\infty (M)\}$$.

	Among $b^m$-forms, we consider  forms of degree two and define the analog of symplectic form in this set-up.
	\begin{definition}
		Let $(M^{2n}, Z)$ be a $b$-manifold.
		Let $\omega \in {^{b^m}}\Omega^2(M)$ be a closed $b^m$-form.
		We say that $\omega$ is \textbf{$b^m$-symplectic} if $\omega_p$ is of maximal rank as an element of $\Lambda^2 \left ( ^{b^m} T_p^* M \right )$ for all $p \in M$.
		We call a \textbf{$b^m$-symplectic manifold} a triple $(M, Z, \omega)$.
	\end{definition}
	When $m=1$, these forms are simply called $b$-symplectic forms.

\begin{remark}\label{rem:bPoisson}
    Alternatively, $b^m$-symplectic forms can be described by Poisson structures on the $b$-manifolds that are non-degenerate away from the critical set $Z$ and are degenerate in a \emph{controlled} way on the critical set. More precisely, a Poisson structure is a bi-vector field $\Pi$ that satisfies that the Schouten-Nijenhuis bracket with itself is zero, i.e. $[\Pi,\Pi]=0$. A $b$-symplectic structure is then equivalent to a Poisson structure such that $\Pi^n$ cuts the zero section \emph{transversally} (see \cite[Section 4]{guimipi2}). The case of $b^m$-symplectic structure relaxes this condition, see \cite{scott}. The symplectic foliation associated to the Poisson structure in these particular cases can be described using the normal local forms of the $b^m$-symplectic forms.
\end{remark}

	A $b^m$-symplectic form can be described in a tubular neighbourhood $\mathcal{N}_\varepsilon(Z)$ of the critical set $Z$ given by $\mathcal{N}_\varepsilon=z^{-1}(-\varepsilon, \varepsilon)$ as 
	\begin{equation}\label{eqn:newlaurent}
		{\omega = \sum_{j = 1}^{m}\frac{dz}{z^j} \wedge \pi^*(\alpha_{j}) +  \beta }
	\end{equation}
	\noindent where the $\alpha_j$ are closed one forms on $Z$, $\beta$  is a  closed 2-form on $\mathcal{N}_\varepsilon$, and $\pi:\mathcal{N}_\varepsilon\longrightarrow Z$ is the projection. Non-degeneracy of the form $\omega$
	implies that $\beta\vert_{Z}$ is of maximal rank and $\alpha_m|_Z$ is nowhere vanishing. The form $\alpha_m|_Z$ defines the symplectic foliation of the Poisson structure associated with $\omega$, and $\beta$ gives the symplectic form on the leaves of this foliation.
		The regular Poisson structure induced on $Z$ is known as a cosymplectic structure:

		\begin{definition}
			\label{def:cosymplectic}
			Let $Z$ a smooth manifold of dimension $2n - 1$, and let $\alpha \in \Omega^1(Z)$ and $\beta \in \Omega^2(Z)$ be closed forms.
			We say that $(Z, \alpha, \beta)$ is a \emph{cosymplectic manifold} if $\alpha \wedge \beta^{n-1}$ is a volume form.
		\end{definition}

Given a $b^m$-function $H$, the Hamiltonian vector field $X_H$ is defined by the non-degeneracy of the $b^m$-symplectic form by $\iota_{X_H} \omega=-dH$. We highlight here that the vector field $X_H$ is a $b^m$-vector field, and therefore when we are interested in the dynamics of this vector field, we view it as a smooth vector field as we did before.

	The only local invariant for $b^m$-symplectic forms is the dimension as the following theorem shows:
	\begin{theorem}[$b^m$-Darboux Theorem, {\cite{gmW1}}] \label{theo:darboux}
		Let $\omega$ be a $b^m$-symplectic form on $(M^{2n}, Z)$. Let $p \in Z$. Then we can find a local coordinate chart $(z, y_1, x_2, y_2, \ldots , x_n, y_n)$ centered at $p$ such that hypersurface $Z$ is locally defined by $z = 0$ and
		$$
		\omega = \frac{dz}{z^m}\wedge dy_1 + \sum_{i = 2}^n dx_i \wedge dy_i.
		$$
	\end{theorem}

	Natural examples of $b$-symplectic manifolds are given by symplectic manifolds with boundary (see \cite{nestandtsygan}). Other examples are given by the space of geodesics on the Lorentz plane (as described in \cite{khesintabachnikov} and discussed in \cite{bcontact}).  Let us recall this construction: Let $M$ be a smooth manifold endowed with a pseudo-Riemannian metric $g$.  Recall that the geodesic flow in $T^*M$ is the Hamiltonian vector field $X_H$ of $H(q,p)=g(p,p)/2$.
Then the set of all oriented geodesic lines $\mathfrak{L}$ can be decomposed as  ${\mathfrak{L}}={\mathfrak{L}}_+\cup   {\mathfrak{L}}_-\cup  {\mathfrak{L}}_0$  where ${\mathfrak{L}}_+, {\mathfrak{L}}_-, {\mathfrak{L}}_0$  stand for the space of oriented non-parameterized space-, time- and light-like geodesics respectively .  The set  ${\mathfrak{L}}_0$ is the common boundary of ${\mathfrak{L}}_\pm$.
Khesin and Tabachnikov proved in \cite{khesintabachnikov} that the manifolds ${\mathfrak{L}}_\pm$ carry symplectic structures which are described  from $T^*M$ by Hamiltonian reduction on the level hypersurfaces $H=\pm 1$. The manifold ${\mathfrak{L}}_0$ carries a contact structure whose symplectization is the Hamiltonian reduction of the symplectic structure in $T^*M$ (without the zero section)
on the level hypersurface $H=0$. In dimension 2 (the Lorenz plane case) this structure is just a $b$-symplectic surface (as in dimension 1 a contact structure is the same as a cosymplectic one).

	The following vector field will play an important role in this paper.
	
	\begin{definition} \label{def:normalbvectorfield}
		Let $(M,Z,\omega)$ be a $b^m$-symplectic manifold.
		By definition, there exists a natural epimorphism of vector bundles $\varphi : \left.{}^{b^m} TM\right|_Z \rightarrow TZ$ induced by the restriction of $b^m$-vector fields into $Z$.
		The kernel of this map, ${}^{b^m}N(M,Z) := \mathrm{ker}(\varphi)$, is a line bundle over $Z$ which is trivializable (see \cite[Proposition 4]{guimipi}).

		We call a $b^m$-vector field $X \in \Gamma({}^{b^m}N(M,Z)) \subset {}^{b^m}\mathfrak{X}(M,Z)$ trivializing the line bundle a {\bf normal $b^m$-vector field}.

		We will call it a {\bf normal symplectic $b^m$-vector field} and denote it as $X^{\sigma}$ if it is also symplectic with respect to $\omega$, this means, if $\mathcal{L}_{X^{\sigma}}\omega = 0$.
	\end{definition}

	$b^m$-Symplectic manifolds can be seen as \emph{open} symplectic manifolds with a certain geometric structure prescribed at the open ends.
	A similar situation occurs with \emph{convex symplectic manifolds}, which we will discuss briefly in Subsection \ref{subsec:convex}, where near the boundary the existence of a transverse Liouville vector field is required.
	In \cite{convexsymplectic} the Floer homology for convex symplectic manifolds is defined.
	From this perspective, $b$-symplectic manifolds represent the other end of the spectrum: they can be seen as symplectic manifolds with boundary equipped with a symplectic vector field (instead of Liouville) that is transverse to the boundary. The role of the symplectic vector field is played by the normal symplectic $b$-vector field (as in Definition \ref{def:normalbvectorfield}).

	\begin{example}
		Consider a $b^m$-symplectic manifold $(M,Z,\omega)$ and let $z$ be a defining function such that in local coordinates around a point in $Z$ we have the decomposition $\omega=\frac{dz}{z^m}\wedge \alpha +\beta$, $\alpha\in \Omega^1(M)$ and $\beta \in \Omega^2(M)$ as in Theorem \ref{theo:darboux}.
		The normal symplectic $b^m$-vector field can then be given by $X^{\sigma}=z^m\frac{\partial}{\partial z}$.
	\end{example}

	\begin{example} \label{ex:2btorus}
		Consider the $b^m$-symplectic torus $(\mathbb{T}^2,\{\sin\theta_1 = 0\},\omega=\frac{d\theta_1}{\sin^m \theta_1}\wedge d\theta_2)$. A normal symplectic $b^m$-vector field is given globally by $X^{\sigma}=\sin^m \theta_1 \frac{\partial}{\partial \theta_1}$. 
	\end{example}

 	The tubular neighborhood and the defining function will always be chosen to satisfy the following lemma.

	\begin{proposition}[{\cite[Theorem 2]{gmw2}}]\label{prop:decomposition}
		Let $(M,\omega)$ be a $b^m$-symplectic manifold and let $\mathcal{N}(Z)$ be tubular neighbourhood around the critical set $Z$. There exists a choice of defining function $z$ for the critical set and a projection $\pi: \mathcal{N}(Z)\to Z$ such that
		\begin{equation}\label{eq:Laurentseriesdef}
		    \omega=\sum_{i=1}^m\frac{dz}{z^i}\wedge \pi^*\alpha_i+\pi^*\beta
      	\end{equation}
		where $\alpha_i\in \Omega^1(Z)$ is closed and $\beta\in \Omega^2(Z)$ is symplectic on the foliation defined by $\alpha_m$.
	\end{proposition}

The series as in Equation (\ref{eq:Laurentseriesdef}) is called the \emph{Laurent series} associated to the $b^m$-symplectic form.

	The following lemma shows how the case that we are studying is actually quite general.
	In particular, we can construct a $b$-symplectic structure for any symplectic manifold with boundary provided we choose a normal symplectic vector field.

	\begin{lemma}[\cite{frejlichmartinezmiranda}]\label{lem:structure}
		Let $(M,\partial M)$ be a manifold with boundary $\partial M$ and let $\omega\in \Omega^2(M)$ a symplectic form on $M \setminus \partial M$ such that there exists a symplectic vector field $X^\sigma$ that points outwards or inwards of the boundary.
		Then, for each $m \in \mathbb{N}, m \geq 1$ there exists a $b^m$-symplectic structure on $(M,\partial M)$ with critical set given by $\partial M$ that coincides with the symplectic structure outside of a tubular neighborhood of the boundary $\partial M$.
	\end{lemma}

	\begin{proof}
		We start by showing that the boundary $\partial M$ can be endowed with a cosymplectic structure.

		We assume without loss of generality that the normal $b$-symplectic vector field $X^{\sigma}$ points inwards at $\partial M$.
		Let $\varphi_t : U \subset M\times \mathbb{R} \rightarrow M$ denote the flow of $X^{\sigma}$.
		As $X^{\sigma}$ is transverse to $\partial M$, there exists a tubular neighbourhood $V_{\varepsilon} := \{ (x,z) \ | \ 0\leq z < \varepsilon(x)\} \subset \partial M \times \mathbb{R}$ for some function $\varepsilon : \partial M \rightarrow \mathbb{R}$ such that
		\[\begin{array}{rccc} c : & V_{\varepsilon} & \longrightarrow & M \\ & (x,z) & \longmapsto & \varphi_z(x) \end{array}\]
			is an embedding.

			Let $\theta := c^{\ast} (\iota_{X^{\sigma}}\omega|_{\mathrm{Im}(c)})$ and $\eta := c^{\ast}\omega|_{\mathrm{Im}(c)}$.
			Both are forms in $V_{\varepsilon}$, but they can be naturally restricted to $\partial M \subset V_{\varepsilon}$, because they are invariant with respect to the coordinate $z$.
			We can show this for $\eta$ under the observation that the vector fields $\frac{\partial}{\partial z}$ and $X^{\sigma}$ are $c$-related, so $\mathcal{L}_{\frac{\partial}{\partial z}} \eta = c^{\ast}\left( \mathcal{L}_{X^{\sigma}} \omega \right) = 0$.
			Conversely,
			\[\mathcal{L}_{\frac{\partial}{\partial z}} \theta = c^{\ast}\left( \mathcal{L}_{X^{\sigma}} \iota_{X^{\sigma}} \omega \right) = c^{\ast}\left( d \iota_{X^{\sigma}}\iota_{X^{\sigma}} \omega + \iota_{X^{\sigma}} d \iota_{X^{\sigma}} \omega \right) = c^{\ast}\left( \iota_{X^{\sigma}} \mathcal{L}_{X^{\sigma}} \omega \right) = 0,\]
			where the first term vanishes because $\omega$ is skew-symmetric, and the second does because $X^{\sigma}$ is a symplectic vector field.

			Moreover, $\theta\wedge \eta^{n-1}$ is a volume form for $\partial M$, because
			\[\theta\wedge \eta^{n-1} = c^{\ast}\left( \iota_{X^{\sigma}} \omega \wedge \omega^{n-1} \right) = \frac1n c^{\ast}\left( \iota_{X^{\sigma}} \omega^n \right) = \frac1n \iota_{\frac{\partial}{\partial z}} c^{\ast}\omega^n,\]
			which is a well-defined non-degenerate form due to its $z$-invariance.

			Therefore, $(\partial M, \theta, \eta)$ is a cosymplectic manifold, and $V_{\varepsilon}$ is an embedding into $M$.
			Moreover, by our definitions it is clear that
			\[c^{\ast} \left. \omega\right|_{\mathrm{Im}(c)} = dz\wedge\theta + \eta.\]

			Let $\psi : V_{\varepsilon} \rightarrow \mathbb{R}$ be a bump function such that
			\begin{itemize}
				\item For $0\leq z < \frac{\varepsilon(x)}3$, $\psi(x,z) = 1$.
				\item For $z > \frac{2\varepsilon(x)}3$, $\psi(x,z) = 0$.
			\end{itemize}
			Consider then the $b$-form
			\[\overline{\omega} = \left(\psi \frac{dz}z + (1-\psi) dz\right)\wedge \theta + \eta .\]
		Showing that $\overline{\omega}$ is symplectic is then trivial, as $\overline{\omega}^n = \psi \frac{dz}z \wedge \theta \wedge \eta^{n-1} + (1 - \psi) dz \wedge \theta \wedge \eta^{n-1}$, which is a non-vanishing $b$-form of maximal degree.
			If we push $\overline{\omega}$ forward to $M$, it coincides with $\omega$ outside of a tubular neighborhood of $\partial M$.

			In the same way, for $m > 1$ we can use the $b^m$-form
			\[\overline{\omega} = \left( \psi \frac{dz}{z^m} + (1 - \psi)dz \right) \wedge \theta + \eta ,\]
			which is $b^m$-symplectic by the same argument.
	\end{proof}

 Associated to a cosymplectic structure as introduced in Definition \ref{def:cosymplectic}, one can define the Reeb vector field as follows.

\begin{definition}\label{def:Reebcosymplectic}
The Reeb vector field $R$ associated to a cosymplectic structure $(\alpha, \beta)$ is defined by $\iota_R \alpha=1$ and $\iota_R \beta =0$.
\end{definition}

	The notion of Reeb vector field of the cosymplectic structure of the critical set of the $b^m$-symplectic manifold, which can be seen as a conjugate of $X^{\sigma}$ with respect to $\omega$, will play an important role in our constructions.
	For $m=1$, that is for $b$-symplectic structures, there is a relationship between the Reeb vector field in $Z$ and the so called modular vector field. Let us recall the definition of modular vector fields.
	
	 Modular vector fields are well-defined for any Poisson manifold and measure how far Hamiltonian vector fields are from preserving a given volume form.

	\begin{definition} \label{def:modular}
		Let $(M,Z,\omega)$ be a $b^m$-symplectic manifold equipped with a smooth volume form $\Omega \in \Omega^{2n}(M)$.
		The {\bf modular vector field}  is the vector field defined as a derivation by
		\[f \mapsto \frac{\mathcal{L}_{X_f} \Omega}{\Omega} .\]
  \end{definition}

	The modular vector field depends on the volume form $\Omega$.
	Changing the volume form yields the same modular vector field, up to the addition of Hamiltonian vector fields (that are by definition tangent to the leaves of the symplectic foliation on $Z$ (see \cite[Propostion 25]{guimipi2}).

  As shown in \cite[Theorem 56]{guimipi2} and \cite[Theorem 19]{guimipi}, when one of the leaves of the symplectic foliation of the associated Poisson structure is compact, then the critical set is given by the mapping torus
		\begin{equation}\label{eq:mappingtorus}
  Z \cong \frac{\mathcal{L} \times [0,T]}{(x,0) \sim (\phi(x),T)}
    \end{equation}
     where $\mathcal{L}$ is a leaf of the symplectic foliation and  $\phi:\mathcal{L} \to \mathcal{L}$ is the symplectic map, given by the return map of the modular vector field, which is transverse to the symplectic foliation.

       The modular vector field is then given by $\frac{\partial}{\partial \theta}$, where $\theta$ represents the translation in the second coordinate (see \cite{gmps}).
			 We refer to the period of the modular vector field in the mapping torus as the \textbf{modular weight} of that component of $Z$.

				 In the case of $b^m$-symplectic manifolds for $m > 1$, the modular vector field vanishes on $Z$.
				 However, the Reeb vector field of the associated cosymplectic structure defines a flow transverse to the symplectic foliation.
				 When one of the symplectic leaves is compact, the critical set is given by the mapping torus as shown in Equation (\ref{eq:mappingtorus}), where $\phi$ represents the flow associated with the Reeb vector field.
				 By analogy, the period of the Reeb vector field in the mapping torus is called the \textbf{modular weight}.
				 It is noteworthy that both definitions agree when $m=1$.

    \begin{definition}\label{rem:trivialmappingdef}
        The mapping torus in Equation (\ref{eq:mappingtorus}) is \textbf{trivial} if the function $\phi$ is isotopic to the identity.
    \end{definition}

\begin{lemma}\label{lem:tubularmappingtorus}
    Let $(M,\omega)$ be a $b^m$-symplectic manifold and $\mathcal{N}(Z)$ be a tubular neighbourhood around the critical set as in Proposition \ref{prop:decomposition}. Then $\mathcal{N}(Z)$ is foliated by symplectic mapping tori.
\end{lemma} 

\begin{proof}
    By the decomposition of Equation (\ref{eq:Laurentseriesdef}), the differential forms $\widetilde{\alpha} = \sum_{i=1}^m z^{m-i} \pi^*\alpha_i$ and $\beta$ are closed differential forms on $Z$ such that $\widetilde{\alpha}\wedge \beta^{n-1}\neq 0$. Hence the restriction of $\widetilde{\alpha}$ and $\pi^*\beta$ to the hypersurface $\widetilde{Z}=z^{-1}(\varepsilon)$ for $\varepsilon$ small enough to assure that $\widetilde{Z}\subset \mathcal{N}(Z)$ are closed forms and as the condition $\widetilde{\alpha}\wedge \beta^{n-1}\neq 0$ is open, they define a symplectic mapping torus.
\end{proof}

	\subsection{Symplectic manifolds with convex boundary} \label{subsec:convex}

	\

	Floer theory in the non-compact set-up has already been studied previously in \cite{convexsymplectic}, most notably when there exists a Liouville vector field (instead of a symplectic one as described in the previous section) pointing outwards of the boundary.

	\begin{definition}\label{def:convex}
		\
		\begin{enumerate}
			\item A compact symplectic manifold $(W,\omega)$ with boundary $\partial W$ is {\bf convex} if there exists a Liouville vector field $X$ near the boundary, this means that $\mathcal{L}_X \omega=\omega$ pointing outwards.
			\item A non-compact symplectic manifold is convex if there exists an increasing sequence of compact convex submanifolds $W_i\subset W$ exhausting $W$.
		\end{enumerate}

		By convention, we assume that $X$ is always  outward-pointing on the boundary.
	\end{definition}

	This induces a contact $1$-form on the boundary, given by $\alpha=\iota_X\omega$ and the Hamiltonian dynamics on the boundary is described by the dynamics of the Reeb vector field $R_\alpha$ on $\partial M$. Important examples are given by cotangent bundles or Stein manifolds.

	Compact convex symplectic manifolds can be viewed as non-compact symplectic manifolds. Indeed, one can attach a so-called \emph{proboscis} to the boundary: using the Liouville vector field, one can glue
	$$(P_\varepsilon:=\partial W \times (-\varepsilon,\infty), d(e^r \alpha))$$
	to the convex symplectic manifold $(W,\omega)$. The resulting manifold, denoted by $(\widehat{W},\widehat{\omega})$ is called the \emph{completion} of $(W,\omega)$. The flow of the Liouville vector field generates an $\mathbb{R}^*$-action on the proboscis $P_\varepsilon$. The compatible almost complex structure in the proboscis is chosen to be invariant with respect to this action and on the boundary, the Liouville vector field is sent to the Reeb vector field.

	The non-compactness is an issue for the Floer trajectories, as they may run along the proboscis to infinity. A wise choice of the notion of admissibility for the Hamiltonian functions assures that the Floer trajectories remain in the original manifold. Namely, the Hamiltonian function $H$ is asked to satisfy on $P_\varepsilon$ that $H=h\circ f:P_\varepsilon\to \mathbb{R}$, where $f(x,r)=e^r$ (here $(x,r)\in \partial M,(-\varepsilon,\infty)$) and $h\in C^\infty(\mathbb{R})$ satisfies that $0\leq h'(\rho) < \kappa$, where
	\begin{equation}\label{eq:kappa}
		\kappa:=\inf\{c>0 \ | \ \dot{x}(t)=cR(x(t))  \text{ has a 1-periodic orbit}\}.
	\end{equation}

	It follows that the Hamiltonian equation on $\partial M$ is given by
	$$\dot{x}(t)=h'(1)R(x(t)),$$
	and thus that there are no $1$-periodic solutions on $\partial M \times [0,\infty)$.

	The authors in \cite{convexsymplectic} apply the maximum principle to the function $h$ to assure that the Floer trajectories do not escape to infinity but remain inside $M$.

	This gives rise to a Floer-type homology defined for convex symplectic manifolds. In order to compute this Floer homology, the authors proceed by proving that this homology is isomorphic to a Morse homology adapted to this non-compact situation. Morse theory for non-compact manifolds is generally not well-defined and additional assumptions on the Morse functions under consideration need to be taken. In this particular setting, an admissible Morse function is a smooth Morse function on $\widehat{M}$ whose restriction to $P_\varepsilon$ is given by
	$$F(x,r)=e^{-r},\quad x\in \partial M,r\in [-\varepsilon,\infty).$$

By choosing a suitable compatible almost complex structure $J$, the negative gradient of an admissible Morse function with respect to the metric $\omega(\cdot, J\cdot )$ fits well with the Hamiltonian vector field associated with an admissible time-independent Hamiltonian function. Indeed, by the choice of the almost complex structure, the negative gradient of an admissible Morse function in $P_\varepsilon$ is just the Liouville vector field $X$ and therefore transverse to the boundary.

	This homology gives rise to a well-defined Morse-type homology and it is classically known that this does not depend on the choice of $F$ (within the class of admissible Morse functions considered), see \cite{schwarz}.

	The authors prove that a PSS-isomorphism between the above-defined Floer homology and Morse homology holds. Thus, this turns the Floer homology on convex symplectic manifold into a homology that can be computed rather easily.

 In this article, we will proceed analogously as in \cite{convexsymplectic}: namely, $b^m$-symplectic manifolds can be viewed as manifolds with boundaries, equipped with a transverse \emph{symplectic} vector field and with an induced cosymplectic structure on the boundary as presented in Lemma \ref{lem:structure}. Around the critical set, the Hamiltonian dynamics that we will consider in the case of $b^m$-symplectic manifolds will have a component in the direction of the Reeb vector field of the cosymplectic structure as defined in Definition \ref{def:Reebcosymplectic}. However, this component, similar to Equation (\ref{eq:kappa}), will be assumed to have period small enough in order to avoid periodic orbits close to the critical set.


	\section{Dynamics around the critical hypersurface and admissible Hamiltonian functions}\label{sec:dyn}

	In this section, we focus on the definition of admissible Hamiltonian that gives rise to well-behaved dynamics close to the critical set of the $b$-symplectic (and $b^m$-symplectic) manifold.
	We will present the conditions on our functions incrementally to lay out the motivation underlying them, particularly in relation to the dynamics of the induced Hamiltonian vector fields.

	Note that if we consider $M\backslash Z$ as an open symplectic manifold with cosymplectic behavior at infinity in the sense introduced in Lemma \ref{lem:structure}, these are a subcategory of Hamiltonians on the open manifold with a prescribed behavior towards infinity.
	In this, our work here has parallels to the construction in \cite{tentacular1,tentacular2} and more recently in \cite{kirchhofflukat}, where the dynamics of the Hamiltonian needs to be subscribed near the singular locus or infinity.

	We also provide a detailed definition of the concept of admissible Hamiltonians.
	

	We begin by extending the notion of an almost complex structure compatible (see Definition \ref{def:acs}) to $b^m$-symplectic structures.

	\begin{definition}
		Let $(M,Z,\omega)$ be a $b^m$-symplectic manifold. An almost complex structure on the $b^m$-tangent bundle is an endomorphism
		$$J:{^{b^m}}TM\to {^{b^m}}TM$$
		such that $J^2=-\Id$. As $({^{b^m}}TM,\omega)$ is a symplectic vector bundle, there exist almost complex structures that are compatible with the $b^m$-symplectic structure, i.e.
		$$g_J(\cdot,\cdot):=\omega(\cdot, J\cdot)$$
		defines a $b^m$-Riemannian metric (that is a bilinear, symmetric and non-degenerate $2$-form on ${^{b^m}}TM$). We furthermore ask for the almost complex structure to be compatible with the cosymplectic structure in the tubular neighborhood around the critical set. More precisely, the almost complex structure must satisfy the following:
		\begin{enumerate}
			\item The restriction of $J$ to the symplectic leaves is a smooth almost complex structure.
			\item It leaves the distribution $\langle X^{\sigma}, R \rangle$ invariant, and in particular
				\[J (X^{\sigma}) = R .\]
			\item The almost complex structure commutes with the flow of the normal symplectic $b^m$-vector field, i.e.
				$$d\varphi_t(z,x)J(z,x)=J(\varphi_t(z),x))d\varphi_t(z,x) .$$
		\end{enumerate}
	\end{definition}

	This definition is very similar to the definition of an almost complex structure compatible with a symplectic structure around a contact-type hypersurface. The only difference here is that, instead of linking the Reeb vector field with the Liouville vector field, we have a relation between the symplectic vector field and the Reeb vector field.
	We give here an example that serves as a semi-local model around the critical set.

	\begin{example}\label{ex:acs}
		In a tubular neighbourhood $\mathcal{N}(Z)$ around the critical set $Z$, the $b^m$-symplectic form is given by
		$$\omega=\sum_{i=1}^m\frac{dz}{z^i}\wedge \pi^*\alpha_i + \pi^*\beta$$
		where $\alpha_i\in \Omega^1(Z)$ and $\beta\in \Omega^2(Z)$ are closed forms {as in Proposition \ref{prop:decomposition}}. The Reeb vector field is defined by the equations $\iota_{R}\alpha_m=1$ and $\iota_{R}\beta=0$ and we consider its extension to the tubular neighborhood. At each point around the tubular neighborhood, we have the splitting
		$${^{b^m}}TM=\langle X^\sigma\rangle \oplus \langle R\rangle \oplus \mathcal{F},$$
  where $\mathcal{F}$ is the symplectic foliation, given by the kernel of $\pi^*\alpha_m$.
  
		We define the almost complex structure by the following:
		\begin{itemize}
			\item $J(X^\sigma)=R$
			\item $J|_\mathcal{F}=J_\mathcal{F}$, where $J_\mathcal{F}$ is an almost complex structure on the foliation $\mathcal{F}$, compatible with the induced symplectic structure by $\beta$.
		\end{itemize}
		The $b^m$-metric $g_J(\cdot,\cdot)$ is given in this decomposition by
		\begin{equation}\label{eq:exactbmetric}
		g_J=\left(\frac{dz}{z^m}\right)^2+\alpha\otimes \alpha+g_\mathcal{F},
		\end{equation} 
		where $g_\mathcal{F}$ is the leafwise Riemannian metric on the foliation $F$. \end{example}
		
		Riemannian metrics on ${^b}TM$ as in Equation (\ref{eq:exactbmetric}) have appeared previously in \cite[Definition 2.12]{escapeorbits} under the name of asymptotically exact $b$-metric.
		\begin{definition}
		A metric on ${^{b^m}}TM$ is called \emph{asymptotically exact} if there exists a trivialization around the critical set $Z$
		$$(z, \pi):\mathcal{N}(Z)\to (-\varepsilon,\varepsilon)\times Z$$
		such that $g=\big(\frac{dz}{z^m}\big)^2+ \pi^*h$, where $h$ is a smooth metric on $Z$.
		\end{definition}

	This example in fact serves as a local model around the critical set.

	\begin{lemma}[\cite{byy}, Lemma 2.6]\label{lem:compatibleacs}
		Let $(M,Z,\omega)$ be a compact oriented $b$-symplectic manifold, with a defining function and tubular neighborhood $\mathcal{N}(Z)$ chosen as in Proposition \ref{prop:decomposition}.

		There exists a direct sum decomposition ${^b} TM|_{\mathcal{N}(Z)} ={^b}T\mathcal{N}(Z) = \pi^*E \oplus \pi^*\mathcal{F}$,
		where $\mathcal{F}= \ker(\alpha)$ and $E = \langle X^\sigma, R \rangle$ are subbundles of ${^b}TM|_Z$.
		Here, $X^{\sigma}$ and $R$ are defined as in Definitions \ref{def:normalbvectorfield} and \ref{def:Reebcosymplectic}, respectively.

		There also exists a compatible complex structure $J$ on the symplectic vector bundle $({^b}TM,\omega)$ such that the metric $g(\cdot,\cdot) = \omega(\cdot,J\cdot)$ takes a product form near $Z$:
		$$\left.g\right|_{\mathcal{N}(Z)}=\Big(\frac{dz}{z}\Big)^2 +\alpha \otimes \alpha+\pi^* g_\mathcal{F}$$
		where $g_\mathcal{F}$ is a compatible metric on the symplectic vector subbundle $(\mathcal{F},\beta_\mathcal{F}) \subset ({^b}TM|_Z,\omega|_Z)$.
	\end{lemma}

 We will assume from here on that the defining function and the tubular neighbourhood are chosen as in Lemma \ref{lem:compatibleacs}. 

	 By this choice, we will denote by $X^{\sigma}$ and $R$ the extensions of the normal (respectively Reeb) vector field to $\pi^{\ast} E$ in $\mathcal{N}(Z)$.

    The lemma can be generalized without any technical difficulties to $b^m$-symplectic manifolds and the proof will therefore be omitted.
	From this lemma, it readily follows that the associated $b^m$-metric to a compatible almost complex structure is an exact $b^m$-metric.

	\begin{lemma}
		The space of compatible almost complex structures is non-empty and contractible.
	\end{lemma}

	\begin{proof}
		We start by defining a compatible almost complex structure in a tubular neighborhood around the critical set as in Example \ref{ex:acs} and we then apply Exercise 2.2.19 in \cite{wendl}.
	\end{proof}

	Let $(M,Z)$ a $b$-manifold, and take $z : M \to \mathbb{R}$ a defining function.
	We will denote the set of time-dependent $b^m$-Hamiltonian functions  (respectively $b$-Hamiltonian functions) as

	\[{}^b \mathcal{C}^{\infty}(S^1\times M) := \{k(t) \log |z| + h_t \ | \ k \in \mathcal{C}^{\infty}(S^1), h \in \mathcal{C}^{\infty}(S^1\times M)\} ,\]
	\[{}^{b^m} \mathcal{C}^{\infty}(S^1\times M) := \left(\bigoplus_{i=1}^{m-1} z^{-i} \mathcal{C}^{\infty}(S^1\times \mathbb{R})\right) \oplus {}^b \mathcal{C}^{\infty}(S^1 \times M) ,\]
	where $\mathcal{C}^{\infty}(S^1\times \mathbb{R})$ refers to the set of functions taking values in time $t$ (the $S^1$-component) and the defining function $z$ (the $\mathbb{R}$-component).

	\begin{definition}\label{ref:linearalongsvf}
		Let $(M,Z,\omega)$ be a $b^m$-symplectic manifold and let $X^{\sigma}$ the normal symplectic $b^m$-vector field.
		We say that a (time-dependent) Hamiltonian $H_t \in {^{b^m}}\mathcal{C}^\infty(S^1\times M)$ is \emph{linear along $X^{\sigma}$} if $\mathcal{L}_{X^{\sigma}} H_t=k_i(t)$, where $k_i\in \mathcal{C}^{\infty}(S^1)$ is a time-dependent function defined in the neighbourhood of the critical set $Z_i$.
	\end{definition}

	\begin{remark}
		If $k \equiv 0$, this means that the Hamiltonian is a first integral of $X^{\sigma}$.
	\end{remark}

 \begin{remark}
    Note that the time dependent function $k_i(t)$ is a smooth function defined in a local neighbourhood around a connected component of the critical set $Z_i$, i.e. they can change from one connected component of $Z$ to another one.
\end{remark}

    \begin{definition}\label{def:constantlinearweight}
    If there exists a global defining function $f$ for the critical set and a unique smooth time-dependent function $k(t)$ such that in each connected component of $Z$, the Hamiloninan $H$ is given in the tubular neighbourhoods by $H_t=k(t)\log|f|+h_t$, or $H_t=k(t)\frac{1}{f^{m-1}}+h_t$ when $m>1$, where $h_t$ is a smooth time-dependant function, we say that $H_t$ is \emph{unimodular}.
\end{definition}

 \begin{example}\label{ex:btoruseasy}
     Consider the $b$-symplectic manifold $(\mathbb{T}^2,\{\theta_1=0\}\sqcup \{\theta_1=\pi\},\omega=\frac{d\theta_1}{\sin\theta_1} \wedge d\theta_2)$ and the Hamiltonian globally defined by
		$$H_t=k(t)\log \left| {\sin\theta_1} \right| +\sin\theta_2,$$
		where $k\in \mathcal{C}^{\infty}(S^1)$ is a time-dependent function. This is a unimodular $b$-Hamiltonian function.
 \end{example}

	\begin{example}\label{ex:btorus}
		Consider the $b$-symplectic manifold $(\mathbb{T}^2,\{\theta_1=0\}\sqcup \{\theta_1=\pi\},\omega=\frac{d\theta_1}{\sin\theta_1} \wedge d\theta_2)$ and the Hamiltonian
  $$H_t=k(t)\log \left| \frac{\sin\theta_1}{1 +\cos\theta_1} \right| +\sin\theta_2,$$
		where $k\in \mathcal{C}^{\infty}(S^1)$ is a time-dependent function.
		First, let us see that this is indeed a globally defined $b$-function.

        The defining function for the connected component of the critical set $Z_1=\{\theta_1=0\}$ is given by $f_1(\theta_1)=\frac{\sin\theta_1}{1+\cos\theta_1}$, which is a \emph{smooth} function around this connected component of the critical set. 
        Using the trigonometric relation $\frac{\sin\theta_1}{1+\cos{\theta_1}}=\frac{1-\cos\theta_1}{\sin\theta_1}$, we can rewrite $$H_t=k(t)\log \left| \frac{1-\cos\theta_1}{\sin\theta_1} \right| +\sin\theta_2=-k(t)\log \left| \frac{\sin\theta_1}{1-\cos\theta_1} \right| +\sin\theta_2.$$
        Around the critical set $Z_2=\{\theta_1=\pi\}$, the function $f_2(\theta_1)=\frac{\sin\theta_1}{1-\cos\theta_1}$ is a \emph{smooth} defining function for the connected component of the critical set given by $\{\theta_1=\pi\}$.

        However, this $b$-Hamiltonian function is \emph{not} a unimodular Hamiltonian function, as it can not be written as $H_t=k(t)\log|f|+h_t$ for a \emph{global} defining function $f$.
		The Hamiltonian vector field is then given by $X_H = k(t)\frac{\partial}{\partial \theta_2}+\cos \theta_2 \sin \theta_1 \frac{\partial}{\partial \theta_1}$.

		Considering the particular case in which $k$ is a constant and $k \notin \mathbb{Z}$ then the time $1$-flow of $X_H$ does not have any periodic orbits around the critical set.
	\end{example}

	We note that, in dimension 2, the Hamiltonian vector fields associated with $b$-Hamiltonian functions that are linear along $X^{\sigma}$ will always have infinitely many periodic orbits around the critical set:

	\begin{proposition}\label{prop:linaralongXnopo}
		Let $(\Sigma,Z,\omega)$ be a compact $b$-symplectic surface with $X^{\sigma}$ a normal symplectic vector field.
		Let $H$ be a $b$-Hamiltonian function such that $\mathcal{L}_{X^\sigma} H=k\in \mathbb{R}$ is constant and different from zero in a tubular neighbourhood around $Z$.
		Then, for any $\varepsilon > 0$ small enough there exist periodic orbits of period $\frac{T}{k}$ in the tubular neighborhood $\mathcal{N}_{\varepsilon}(Z)$ around the critical set with modular weight $T$. Furthermore, if $k\in T \mathbb{Z}$, then there exists $1$-periodic Hamiltonian orbit in $\mathcal{N}_\varepsilon(Z)$.
	\end{proposition}

	\begin{proof}
		By compactness, $Z$ must be a finite union of circles.
		The mapping torus as in Equation (\ref{eq:mappingtorus}) is trivial and thus a tubular neighbourhood of the critical set is diffeomorphic to $(-\varepsilon,\varepsilon)\times [0,T]\left/\raisebox{-.1em}{$(0 \sim T)$}\right.$ and the $b$-symplectic form is given in these coordinates by $\omega=\frac{dz}{z}\wedge d\theta$.
  
   In these coordinates the normal symplectic $b$-vector field is $X^{\sigma} =  z \frac{\partial}{\partial z}$.
    By definition, the $b$-Hamiltonian function is given by $H_t(z,\theta)=k_t\log|z|+h_t(\theta,z)$.
		If $H_t$ satisfies the conditions of the proposition, then $k_t=k$ and $h_t(\theta,z)=h_t(\theta)$, where $h_t(\theta)$ is a $T$-periodic function in $\theta$. Thus the Hamiltonian function is given by $H_t(z,\theta) = k\log |z| + h_t(\theta)$.

		Hence the Hamiltonian vector field is given by
        \[X_{H_t} =  k \frac{\partial}{\partial \theta} - z\frac{\partial h_t}{\partial \theta} \frac{\partial}{\partial z}.\]

		The flow of this vector field $X_{H_t}$ is given by
  		\[\begin{array}{l}
			\theta(t) = \theta_0 +kt \\
			z(t) = z_0 \exp{\left( - \displaystyle\int_0^t \frac{\partial h_s}{\partial \theta}(\theta_0 +  ks) ds \right)}
		\end{array}\]
		where $(z_0,\theta_0)$ is the initial position.
		Recall that, as the image of $\theta$ lies in $[0,T]\left/\raisebox{-.1em}{$(0 \sim T)$}\right.$, $\theta(t)$ is merely a periodic flow, and in particular $\theta(t+\frac{T}{k})= \theta(t)$.
      
        Let us define the function 	$$F(\theta_0):=\int_0^{\frac{T}{k}} \frac{\partial h_t}{\partial \theta}(\theta_0 + kt) dt.$$

		The $z$-component of the flow has a $1$-periodic orbit if and only if there is some $\theta_0 \in S^1$ such that $F(\theta_0)= 0$. We claim that such a $\theta_0$ always exists.
		Indeed, integrating the function $F$ over $S^1$ we obtain by applying Fubini's theorem that
		\[\int_0^T F(\theta) d\theta = \int_0^T \int_0^{\frac{T}{k}} \frac{\partial h_t}{\partial \theta}(\theta + kt) dt d\theta = \int_0^{\frac{T}{k}} \int_0^T \frac{\partial h_t}{\partial \theta}(\theta + kt) d\theta dt = \int_0^{\frac{T}{k}} [h_t(\theta + kt)]_{\theta=0}^{\theta=T} dt.\]

        As $h_t$ periodic in $\theta$ of period $T$, it follows that $h_t(T+kt)=h_t(kt)$ and hence $F(\theta)=0$ for some $\theta\in [0,T]$.
		Then, for any $z_0$ small enough there exists some $\theta_0$ such that $(z_0,\theta_0)$ belongs to a periodic orbit of period $\frac{T}{k}$. In particular, if $k\in T\mathbb{Z}$, then there exists orbits of period $\frac{T}{k}$ in $\mathcal{N}_\varepsilon(Z)$, and therefore also $1$-periodic orbits.
	\end{proof}

A similar statement holds for $b^m$-symplectic surfaces.

	\begin{proposition}\label{prop:bmlinaralongXnopo}
		Let $(\Sigma,Z,\omega)$ be a compact $b^m$-symplectic surface with $X^{\sigma}$ a normal symplectic vector field.
		Let $H$ be a $b^m$-Hamiltonian function such that $\mathcal{L}_{X^\sigma} H=k\in \mathbb{R}$ is constant (different from zero) in a tubular neighbourhood around $Z$.
		Then for any $\varepsilon > 0$ there exist periodic orbits around the critical set.
	\end{proposition}

 In contrast to the last proposition, in the $b^m$-symplectic case, we don't have a family of orbits with the same period. The proof will be as in Proposition \ref{prop:linaralongXnopo}, and we only repeat here the most important steps.

 \begin{proof}
 In the tubular neighborhood, the $b^m$-symplectic form is given as the following Laurent series:
      $$\omega = \left(\sum_{i=1}^m z^{m-i} a_i \right) \frac{dz}{z^m} \wedge d\theta.$$
By the condition on the Hamiltonian, it is given by
      $H_t(z,\theta) = -\frac{1}{m-1} \frac{k}{z^{m-1}} + h_t(\theta)$ if $m > 1$, where $k$ is constant.
      The Hamiltonian vector field is therefore given by
    \[X_{H_t} = \frac{1}{\sum_{i=1}^m z^{m-i} a_i} \left( k \frac{\partial}{\partial \theta} - z^m \frac{\partial h_t}{\partial \theta} \frac{\partial}{\partial z} \right).\]

As in a tubular neighborhood, the function $ \frac{1}{\sum_{i=1}^m z^{m-i} a_i}$ is non-vanishing, we consider the parametrization of the vector field $X_H$ given by
$$\widetilde{X}_{H_t}= k \frac{\partial}{\partial \theta} - z^m \frac{\partial h_t}{\partial \theta} \frac{\partial}{\partial z}.$$

    As this is a reparametrization, both vector field are orbitally equivalent, and thus periodic orbits of $\widetilde{X}_{H_t}$ correspond to periodic orbits of $X_{H_t}$.

    We now analyze the flow of the vector field $\widetilde{X}_{H_t}$ and show that this vector field has always infinitely many periodic orbits around the critical set.

    \[\begin{cases}
			 \theta(t) = \theta_0 +kt \\
			 z(t) =  \left(\frac1{z_0^{m-1}} - (m-1) \displaystyle\int_0^t \frac{\partial h_t}{\partial \theta}(\theta_0 +  ks) ds \right)^{-\frac{1}{m-1}}
		\end{cases}\]

    The flow in the $\theta$ coordinate is periodic of period $\frac{T}{k}$ as in the proof of Proposition \ref{prop:linaralongXnopo} and the same arguments apply as therein apply. This implies that for any $z_0$, there exists $\theta_0$ such that the integral curve of the reparametrization $\widetilde{X}_{{H}_t}$ with initial condition $(z_0,\theta_0)$ is periodic. The same thus holds for the Hamiltonian vector field $X_{H_t}$. However, notice that due to the reparametrization, the periods of the periodic orbits are not constant.
 \end{proof}

    \begin{remark}
        On $Z$, the flow described in Proposition \ref{prop:bmlinaralongXnopo} is as the flow in Proposition \ref{prop:linaralongXnopo}. Hence on $Z$, there is a periodic orbit of period $\frac{a_m}{k}$ on $Z$, where $a_m$ is the modular weight.
    \end{remark}

	This result is also true in higher dimensions if $H$ does not depend on time.

	\begin{lemma}
		Assume $H$ is a time-independent Hamiltonian that is a first integral of $X^{\sigma}$. Then there is a $1$-parametric family of critical points approaching the critical set.
	\end{lemma}

	\begin{proof}
		As $H$ is a first integral of $X^{\sigma}$, $H$ can be viewed as a function on $z^{-1}(\varepsilon)$, where $Z=z^{-1}(0)$. As $Z$ is compact, there exist critical points of $H$ on each hypersurface. The critical points translate to trivial periodic orbits which appear as a $1$-parametric family.
	\end{proof}

	The same result holds also in higher dimensions if $\mathcal{L}_{X^{\sigma}} H = 0$ and the geometry of $Z$ is that of a symplectic mapping torus.

	\begin{proposition} \label{prop:trivialmappingtorusinfperiodic}
		Let $(M, Z,\omega)$ be a $b$-symplectic manifold such that $Z$ is compact and $Z$ is a trivial mapping torus as in Remark \ref{rem:trivialmappingdef}.
		Then all $X^\sigma$-invariant Hamiltonian functions have periodic orbits arbitrarily close to $Z$.
	\end{proposition}

	\begin{remark}
		A sufficient condition so that $(M ,Z,\omega)$ has a trivial mapping torus at $Z$ is that the cohomology class $[\omega] \in {}^{b}H^2(M)$ is integral.
		See \cite[Section 2]{gmw} for more details.
	\end{remark}

	\begin{proof}
		Around the critical set, we have  $\omega=\frac{dz}{z}\wedge \alpha+\beta$ as in Proposition \ref{prop:decomposition}.
		The tubular neighborhood admits a codimension $2$ symplectic foliation as in Proposition \ref{lem:compatibleacs}.
		Denote the codimension $2$ foliation around the critical set by $\mathcal{F}$.
		We want to solve Hamilton's equation $\iota_{X_H}\omega=-dH$.
		The Hamiltonian vector field can be computed using the $b$-Poisson structure (see also Remark \ref{rem:bPoisson}), which is given by
		$$\Pi= z \frac{\partial}{\partial z} \wedge R+\pi_\mathcal{L},$$
		where $\pi_\mathcal{L}$ is the restriction of the Poisson structure associated to the $b$-Poisson structure to the leaf of the foliation.
		Thus the Hamiltonian vector field is given by $\Pi^\sharp(dH)$.

		As $H$ does not depend on $z$, the Hamiltonian vector field has the expression $X_H=(X_H)_{\mathcal{L}} +\frac{\partial H}{\partial \theta} z\frac{\partial}{\partial z}$ where $\frac{\partial H}{\partial \theta}$ is a function which does not depend on $z$ and $(X_H)_{\mathcal{L}}$ is a Hamiltonian vector field along the leaf.
		On one hand, by the Arnold conjecture applied to a compact symplectic leaf $\mathcal{L}$, there always exists a $1$-periodic orbit of $(X_H)_{\mathcal{L}}$ in $\mathcal{L}$.
		On the other hand, we can apply Proposition \ref{prop:linaralongXnopo} to the term $\frac{\partial H}{\partial \theta}z\frac{\partial }{\partial z}$ and thus we always find a $1$-periodic orbits on the direction of the normal symplectic vector field $X^\sigma$.
		Hence there are periodic orbits for $X_H$.
		Furthermore, they come in $1$-parametric families.
	
	\end{proof}

	In contrast to the last propositions, we define \emph{admissible Hamiltonian functions} to be linear along $X^{\sigma}$ and without periodic orbits close to the critical hypersurface. Remember that when the foliation induced on the critical set has compact leaves, the symplectic foliation is a mapping torus. Thus, by means of the classical Arnold conjecture, a smooth Hamiltonian has an infinite number of periodic orbits. The condition of admissible Hamiltonian is on the other side of the spectrum and thus underlying the importance that the associated function is not smooth. 
	In later sections, we will prove that if this condition  is met, then finite energy solutions to the Floer equation will not approach the critical hypersurface.
 
	\begin{definition}\label{def:admissibleHamiltonian}
		A Hamiltonian $H_t\in {^{b^m}}\mathcal{C}^\infty(M)$ is \emph{admissible} if in a tubular neighbourhood around a connected component of the critical set $Z_i$
		\begin{enumerate}
			\item it is linear along the normal symplectic $b^m$-vector field, $\mathcal{L}_{X^{\sigma}} H_t = k_i(t)$,
			\item it is invariant with respect to the Reeb vector field: $\mathcal{L}_{R} H_t = 0$,
			\item there are no $1$-periodic orbits of $X_H$ in a tubular neighbourhood around $Z$.
		\end{enumerate}
			We denote the set of admissible Hamiltonian functions by ${^{b^m}}\mathrm{Adm}(M,R,X^\sigma)$.
   
	\end{definition}

		\begin{remark}
			Definition \ref{def:admissibleHamiltonian} depends on the choice of the normal symplectic vector field and the tubular neighbourhood, which induces the extensions $X^{\sigma}$ and $R$ as denoted in Lemma \ref{lem:compatibleacs}. 	However, the lower bounds concerning the number of periodic Hamiltonian orbits in this paper are independent of these choices.
		\end{remark}
		
	\begin{remark} \label{remark:localFormHamiltonian}
		In the surface case, the expression in a tubular neighbourhood of the critical set of an admissible Hamiltonian function is $H_t= k(t) \log|z|$ for $m=1$ and
		$H_t=-k(t)\frac{1}{m-1}\frac{1}{z^{m-1}}$
		for $m > 1$, where $k\in C^\infty(S^1)$ is a smooth function, and $z$ is a local defining function for the critical set.

		In higher dimensions there is an additional term, represented by a local function $h_t$ such that $\mathcal{L}_{X^{\sigma}} h_t = 0$ and $\mathcal{L}_{ R } h_t = 0$, which means that $h_t$ depends only on the coordinates on the symplectic leaves.
		The local expression of an admissible Hamiltonian functions is then $H_t= k(t) \log|z| + h_t $ for $m=1$ and
		$H_t=-k(t)\frac{1}{m-1}\frac{1}{z^{m-1}} + h_t$
		for $m > 1$.

		Recall that if there exists a global defining function $z$ such that the above expression holds, then we say that the admissible Hamiltonian is \emph{unimodular}, see also Definition \ref{def:constantlinearweight}.
	\end{remark}

	\begin{example}
		Consider $(\mathbb{T}^2,\frac{d\theta_1}{\sin \theta_1}\wedge d\theta_2)$ and let $H\in {^b}C^\infty(\mathbb{T}^2)$ be a Hamiltonian such that $\mathcal{L}_{X^\sigma}H=k(t)$ and $\mathcal{L}_{ R }H=0$. Under those conditions, the Hamiltonian is given by $H_t=k(t) \log|\tan(\frac{\theta_1}{2})|$ and thus the Hamiltonian vector field is $X_{H_t}=-k(t)\frac{\partial}{\partial \theta_2}=-k(t) R$. This has $1$-periodic orbits if and only if
		\[\int_0^1 k(\theta) d\theta \in T \mathbb{Z}.\]
		where $T$ is the modular weight. 
		
		Thus the condition that there do not exist any periodic orbits in a tubular neighbourhood around $Z$ is necessary.
	\end{example}

    The preceding example can easily be adapted to the $b^2$-symplectic case:

    \begin{example}\label{ex:b2torusHam}
			Consider the $b^2$-symplectic surface given by $(\mathbb{T}^2,\omega=\frac{d\theta_1}{\sin^2 \theta_1}\wedge d\theta_2)$. An admissible Hamiltonian is given by $H_t(\theta_1,\theta_2)=-k(t)\frac{\cos\theta_1}{\sin \theta_1}$, where $k(t)$ is a globally defined function such that $0<\int_0^1 k(t)dt<1$. It follows that the associated Hamiltonian vector field is given by $X_H=k(t)\frac{\partial}{\partial \theta_2}$.
    \end{example}

	More generally, $b$-Hamiltonian functions that are linear along $X^\sigma$ and invariant with respect to the Reeb vector field do not admit any periodic orbits in a tubular neighbourhood around $Z$ if the following is satisfied:

	\begin{proposition} \label{prop:K}
		Assume that the leaves of the mapping torus associated to the $b$-symplectic structure are compact.
		Let $H_t\in {^{b}}C^\infty(M)$ such that $\mathcal{L}_{X^\sigma}H_t=k(t)$ and $\mathcal{L}_{R}H_t=0$. If
		\[\int_0^1 k(t)dt \in (0,T) ,\]
		where $T$ is the modular weight of $Z$, then there are no periodic orbits of $X_H$ in a tubular neighbourhood around $Z$. 
	\end{proposition}

	\begin{proof}
		In a  neighbourhood of a point in $Z$, the $b$-Poisson structure can be described as  $\Pi= X^\sigma \wedge R+\pi_\mathcal{L}$. 
		Under the assumptions of the proposition, the Hamiltonian vector field is given by $X_H=\Pi^\sharp(dH)=-k(t) R +(X_H)_{\mathcal{L}}$.
		The critical set is given by the mapping torus
		$$ Z = \frac{\mathcal{L} \times [0,T]}{(x,0) \sim (\phi(x),T)}$$ as in Equation \ref{eq:mappingtorus} and the Reeb vector field is $R=\frac{\partial}{\partial \theta}$, where $\theta$ is the translation in the second coordinate (see \cite{gmps}). Observe that the description of the mapping torus holds in a neighbourhood, as is explained in Lemma \ref{lem:tubularmappingtorus}.

        A necessary condition for the Hamiltonian flow to have a $1$-periodic orbit in the neighbourhood of $Z$ is that the flow of $X_H$ returns to the same leaf of the symplectic foliation. Consequently, it requires $\int_0^1 k(t)dt \in T\mathbb{Z}$.

The condition $\int_0^1 k(t)dt\in (0,T)$ imposes thus that the Hamiltonian vector field moves along the transverse direction, but not enough for it to return to the same leaf.
		Thus this condition imposes that there are no periodic orbits in a tubular neighbourhood around the critical set, as the Hamiltonian vector field is not returning the same leaf of the symplectic foliation.
	\end{proof}

        In the case where $M$ is a $b^m$-symplectic surfaces, there exists an upper bound $T_\varepsilon>0$ such that if $\int_0^1 k(t)dt \in (0,T_\varepsilon)$, then there are no periodic orbits of $X_H$ in a $\varepsilon$-neighbourhood.

    \begin{proposition}
		Assume that the leaves of the mapping torus associated to the $b^m$-symplectic structure are compact.
		In a tubular neighbourhood $\mathcal{N}_\varepsilon(Z)$, let $H_t\in {^{b^m}}C^\infty(M)$ such that $\mathcal{L}_{X^\sigma}H_t=k(t)$ and $\mathcal{L}_{R} H_t=0$. Then $\varepsilon$ can be choosen sufficiently small, such that there exists $T_\varepsilon>0$ such that if
		\[\int_0^1 k(t)dt \in (0,T_\varepsilon) ,\]
		then there are no periodic orbits of $X_H$ in $\mathcal{N}_\varepsilon(Z)$.
    \end{proposition}

    \begin{proof}
        In $\mathcal{N}_\varepsilon(Z)$, the $b^m$-symplectic form is given by $\big(\sum_{i=1}^m z^{m-i}a_i\big)\frac{dz}{z^m}\wedge d\theta$ and by assumption, the Hamiltonian is given by $H=-\frac{k(t)}{(m-1)z^{m-1}}$. Consequently, the Hamiltonian vector field is given by $X_H=\frac{k(t)}{\sum_{i=1}^m (a_iz^{m-i})} \frac{\partial}{\partial \theta}.$
        The Hamiltonian flow is given by
        $$z(t)=z_0,\quad \theta(t)=\theta_0+\frac{1}{\sum_{i=1}^m a_iz_0^{m-i}}\int_0^t k(\tau) d\tau.$$
        The necessary condition for a $1$-periodic orbit to exist is given by
        $$\int_0^1 k(\tau)d\tau \in T\left(\sum_{i=1}^m a_iz_0^{m-i}\right)\mathbb{Z},$$
        \noindent with $T$ the modular period.
        Without loss of generality, we can assume that $a_m > 0$ and that $\varepsilon>0$ is small enough so that
        $$a_m > \frac12 \liminf_{z\in(-\varepsilon,\varepsilon)} \sum_{i=1}^m a_iz^{m-i}>0.$$
				Consequently, if we take $T_\varepsilon < \frac12 T \liminf\limits_{z\in(-\varepsilon,\varepsilon)} \sum\limits_{i=1}^m a_iz^{m-i}$, then there are no $1$-periodic orbits in $\mathcal{N}_\varepsilon(Z)$.
    \end{proof}

	\begin{example} \label{ex:arnoldtorus}
		Let us consider the $2n$-torus $\mathbb{T}^{2n}$ with coordinates $(\theta_1, ... , \theta_{2n})$, with a critical set $Z = \{\sin\theta_1 = 0\}$ and with a $b$-symplectic form $\omega$ globally given by
		\[\omega = \frac{d\theta_1}{\sin\theta_1} \wedge \alpha + \beta ,\]
		for some closed $\alpha \in \Omega^1(\mathbb{T}^{2n})$ and $\beta \in \Omega^2(\mathbb{T}^{2n})$.

		Let us then take $H_t = k \log \left|\frac{\sin \theta_1}{1 + \cos\theta_1} \right|$ for $k$ constant, so that $dH_t=k\frac{d\theta_1}{\sin \theta_1}$ and thus $X_H = k R$.
		As we saw in Proposition \ref{prop:K} we can choose $k$ small enough so that $X_H$ does not have 1-periodic orbits.
		Thus, there exists a family of admissible Hamiltonian functions in $\mathbb{T}^{2n}$ with no 1-periodic orbits.
	\end{example}

	\begin{remark}
		Consider the product of $b^m$-symplectic manifolds.
		This is no longer a $b^m$-symplectic manifold but admits a symplectic structure over a Lie algebroid. Those manifolds are called $c$-symplectic manifolds in \cite{mirandascott} and are a particular case of $E$-symplectic manifolds (see also, \cite{nesttsygan2}, \cite{mirmirandanicolas}).
		More precisely, let $(M_i,Z_i,\omega_i)$ be a collection of $b$-symplectic manifolds, and consider $(M:=\Pi_{i=1}^n M_i,\overline{Z},\omega=\sum_i\pi_i^*\omega_i)$, where $\pi_i:M\to M_i$ denotes the projection and
		\[\overline{Z} = \bigcup_{i=1}^n M_1\times \cdots \times Z_i \times \cdots \times M_n .\]
		The vector field $X^\sigma=\sum_{i=1}^n X_i^\sigma$ is a symplectic vector field on $M\setminus \overline{Z}$, which is transverse to the boundary. The same constructions as before thus work in this more general set-up.
	\end{remark}

    As for Floer homology, we need to consider non-degenerate periodic orbits.
    
	\begin{definition} \label{def:regularHam}
		Consider a Hamiltonian $H_t \in {}^{b^m}\mathrm{Adm}(M,\omega)$.
		We denote by $\mathcal{P}(H)$ the set of 1-periodic orbits of the flow of $X_H$ that are completely contained in $M\setminus Z$.
		We say that $H$ is {\bf regular} if for all periodic orbits $x \in \mathcal{P}(H)$ we have that
		\[\mathrm{det}(\mathrm{Id} - d\phi_{X_H}^1(x(0))) \neq 0 .\]
		We denote the set of regular admissible Hamiltonian functions by ${}^{b^m}\mathrm{Reg}(M,\omega)$.
	\end{definition}

	The set of regular Hamiltonian functions is open and dense in ${}^{b^m}\mathrm{Adm}(M,\omega)$ in the strong Whitney $\mathcal{C}^{\infty}$-topology.


\section{Desingularization of singular symplectic structures: old and new}
\label{sec:desing}

 The desingularization process introduced in \cite{gmW1} associates a family of symplectic structures to a $b^m$-symplectic structure for even $m$.    We use this technique to relate the initial dynamics to the \emph{desingularized dynamics}.
 
 In this section, we revisit this technique and generalize it for arbitrary $m$ on surfaces and in higher dimensions when the critical set is a trivial mapping torus.  {Our novel approach allows to associate a symplectic structure in these scenarios also for odd $m$.}
 
We will also revisit the \emph{singularization} of surfaces (and manifolds with trivial mapping tori), which defines a produces a $b^m$-symplectic structure out of the symplectic one.

First, let us briefly recall the desingularization procedure of a $b^m$-symplectic manifold.

	\begin{theorem}[{\cite[Theorem 3.1]{gmW1}}] \label{th:deblog}
		Let $(M,Z,\omega)$ be a $b^{2m}$-symplectic manifold.
		Then there exists  a family of \textbf{symplectic} forms ${\omega_{\varepsilon}}$ which coincide with the $b^{2m}$-symplectic form $\omega$ outside an $\varepsilon$-neighborhood of $Z$ and for which  the family of bivector fields $(\omega_{\varepsilon})^{-1}$ converges in the $\mathcal{C}^{2m-1}$-topology to the Poisson structure $\omega^{-1}$ as $\varepsilon\to 0$ .
	\end{theorem}

	An immediate consequence of this theorem is that a $b^{2m}$-symplectic manifold also admits a symplectic structure. So, in particular, all symplectic topological obstructions apply to $b^{2m}$-symplectic manifolds. We will sketch the proof of Theorem \ref{th:deblog} following \cite{gmW1}.

As in Equation (\ref{eq:Laurentseriesdef}), a Laurent series can be associated to any closed $b^m$-form in a tubular neighbourhood $\mathcal{N}(Z)$ of $Z$ (see \cite{scott}):
	\begin{equation}
		\label{eq:Laurent}
		\omega = \frac{dz}{z^m} \wedge \left ( \sum \limits_{i = 0}^{m - 1} \pi^* (\alpha_i) z^i \right ) + \beta,
	\end{equation}
	where $\pi : \mathcal{N}(Z) \to Z$ is the projection of the tubular neighborhood onto $Z$, $\alpha_i$ is a closed smooth de Rham 1-form on $Z$, and $\beta$ is a de Rham 2-form on $M$ and $z$ is a defining function for $Z$ in $\mathcal{N}(Z)$ (the function does not need to be local at all, in fact Equation (\ref{eq:Laurent}) is valid only in $\mathcal{N}(Z)$).

Because of the formula~\ref{eq:Laurent}, the $b^{2m}$-form can be written as
	\begin{equation}
		\label{eq:b2kDecomp}
		\omega = \frac{d z}{z^{2m}} \wedge \sum \limits_{i = 0}^{2m - 1} \left ( \pi^*(\alpha_i ) z^i \right ) + \beta
	\end{equation}
	in a tubular neighbourhood of a given connected component of $Z$.

	\begin{definition}\label{def:desingularization}
		Let $(M,\omega)$, be a $b^{2m}$-symplectic manifold with critical set $Z$. Consider the decomposition given by the Equation~\ref{eq:b2kDecomp} on an $\varepsilon$-tubular neighborhood $\mathcal{N}_\varepsilon(Z)=\{|z|<\varepsilon\}$ of a connected component of $Z$.

		Let $f \in \mathcal{C}^\infty (\mathbb R)$ be an odd smooth function satisfying $f'(x) > 0$ for all $x \in [-1, 1]$ and satisfying outside this interval that 
		\begin{equation*}
			f(x) =
			\begin{cases}
				\frac{-1}{(2m-1)x^{2m -1}} - 2 \phantom{a} \mathrm{ for} \phantom{a} x < -1\\
				\frac{-1}{(2m-1)x^{2m -1}} + 2 \phantom{a} \mathrm{ for } \phantom{a} x > 1.
			\end{cases}\,
		\end{equation*}
		We define $f_\varepsilon$ as $\varepsilon^{-(2m - 1)} f(x/\varepsilon)$.
		More precisely, the expression of $f_\varepsilon$ is given as follows:
		\begin{equation}\label{def:feps}
			f_\varepsilon(x) =
			\begin{cases}
				\frac{-1}{(2m-1)x^{2m -1}} - \frac{2}{\varepsilon^{2m-1}} \phantom{a} \mathrm{ for} \phantom{a} x < -\varepsilon\\
				\frac{-1}{(2m-1)x^{2m -1}} + \frac{2}{\varepsilon^{2m-1}} \phantom{a} \mathrm{ for } \phantom{a} x > \varepsilon.
			\end{cases}\,
		\end{equation}
				Observe that $f'_\varepsilon(x)=\frac{1}{x^{2m}}$ for $|x|>\varepsilon$. In particular, $f_\varepsilon'>0$.

		The $f_\varepsilon$-\emph{desingularization} $\omega_\varepsilon$ is a smooth $2$-form that is defined on a neighbourhood $\mathcal{N}_\varepsilon(Z)$ as
		\begin{equation}
			\label{eq:b2kDesing}
			\omega_\varepsilon = d f_\varepsilon(z) \wedge \left ( \sum \limits_{i = 1}^{2m-1} z^i \pi^*(\alpha_i) \right ) + \beta.
		\end{equation}
	\end{definition}
	Observe that $\omega_\varepsilon$ can be trivially extended to the whole manifold  so that it coincides with $\omega$ outside $\mathcal{N}_\varepsilon(Z)$. Hence $\omega_\varepsilon$ is a smooth differential form on $M$.
	Furthermore, $\omega_\varepsilon$ can be shown to be symplectic, because a direct computation yields that in $\mathcal{N}_\varepsilon(Z)$, $\omega_\varepsilon^n=f_\varepsilon'(z)z^{2m}\omega$.

   For $b^{2m+1}$-symplectic manifolds, it is shown in \cite{gmW1} that one can desingularize this singular symplectic structure into a folded symplectic structure. We will show that in many cases, it is still possible to construct a smooth symplectic structure out of the $b^{2m+1}$-symplectic structure. Furthermore, given a Hamiltonian function, the initial associated Hamiltonian vector field can be reinterpreted a smooth Hamiltonian vector field under this desingularization. In order to analyze this improved desingularization,
  it will be useful to associate a graph to a $b$-manifold.
	This graph has been previously used notably in \cite{mirandaplanas} and more recently in \cite{kirchhofflukat}.

		\begin{definition}\label{def:bgraph}
			Let $(M, Z)$ be a $b$-manifold, and consider $\mathcal{N}_{\varepsilon}(Z)$ a tubular neighbourhood of $Z$ for $\varepsilon > 0$ small enough.
			The \emph{graph associated to} $(M,Z)$ is then given by the set of vertices defined by the connected components of $M \setminus \mathcal{N}_{\varepsilon}(Z)$, and by the set of edges given by the connected components of $\mathcal{N}_{\varepsilon}(Z)$.
			Two vertices $M_i, M_j \subset M \setminus \mathcal{N}_{\varepsilon}(Z)$ are said to be adjacent under this definition if there exists a connected component of $\mathcal{N}_{\varepsilon}(Z)$ which borders both of them.

			If $\varepsilon$ is small enough it is clear that the graph does not depend on the tubular neighbourhood $\mathcal{N}_{\varepsilon}(Z)$, and, in particular, it can be represented by the connected components of $M \setminus Z$ and of $Z$ for the vertices and the edges, respectively.
			In this sense, the graph is a topological invariant of the pair $(M,Z)$.
		\end{definition}

		\begin{definition}
			We say that the $b$-manifold is \emph{cyclic} (or \emph{acyclic}) if its associated graph contains (or, respectively, does not contain) a cycle.
		\end{definition}

		\begin{definition}
			Let $(\Sigma, Z)$ be a $b$-surface.
			Then we say that the \emph{weight} associated to a vertex $\Sigma_i \subset \Sigma \setminus Z$, denoted by $g_{\Sigma_i}$, is the genus of the connected component $\Sigma_i$ seen as a surface (assuming that each open part is completed with a disk).
		\end{definition}

	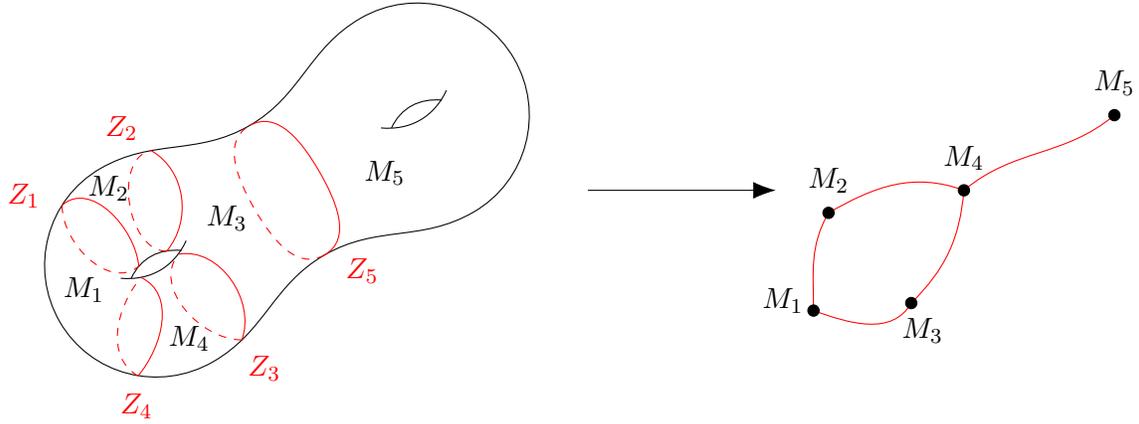
\begin{figure}[h]
		\centering
		\begin{tikzpicture}
			\draw [] (-2.48,0.29) to [out=210,in=120] (-3.02,-1.75) to [out=300,in=210] (-0.98,-2.28) to [out=30,in=210] (0.5,-0.86) to [out=30,in=210] (2.48,-0.29) to [out=30,in=300] (3.02,1.75) to [out=120,in=30] (0.98,2.29) to [out=210,in=30] (-0.5,0.86) to [out=210,in=30] (-2.48,0.29);
			\draw [] (-2.07,-1.17) to [out=65,in=175] (-1.4,-0.8);
			\draw [] (-2.21,-1.17) to [out=355,in=245] (-1.34,-0.67);
			\draw [] (1.39,0.83) to [out=65,in=175] (2.06,1.2);
			\draw [] (1.25,0.83) to [out=355,in=245] (2.12,1.33);

			\draw [red] (0.5,-0.86) to [out=30,in=300] (0.43,0.25) to [out=120,in=30] (-0.5,0.86);
			\draw [red,dashed] (0.5,-0.86) to [out=210,in=300] (-0.43,-0.25) to [out=120,in=210] (-0.5,0.86);

			\draw [red] (-1.98,-2.47) to [out=50,in=335] (-1.94,-1.16);
			\draw [red,dashed] (-1.98,-2.47) to [out=150,in=230] (-1.94,-1.16);

			\draw [red] (-0.6,-1.99) to [out=70,in=10] (-1.44,-0.85);
			\draw [red,dashed] (-0.6,-1.99) to [out=180,in=230] (-1.44,-0.85);

			\draw [red] (-2.99,-0.2) to [out=42,in=95] (-1.97,-1.01);
			\draw [red,dashed] (-2.99,-0.2) to [out=270,in=220] (-1.97,-1.01);

			\draw [red] (-1.6,-0.81) to [out=50,in=330] (-1.8,0.525);
			\draw [red,dashed] (-1.6,-0.81) to [out=180,in=190] (-1.8,0.53);

			\draw (-2.7,-1.75) node [label={$M_1$}] {};
			\draw (-2.37,-0.4) node [label={$M_2$}] {};
			\draw (-1.3,-2.4) node [label={$M_4$}] {};
			\draw (-0.8,-0.8) node [label={$M_3$}] {};
			\draw (1.3,-0.2) node [label={$M_5$}] {};

            \draw[color=red] (-3.5,-0.5) node [label={$Z_1$}] {};
            \draw[color=red] (-2.2,0.4) node [label={$Z_2$}] {};
            \draw[color=red] (-0.3,-2.8) node [label={$Z_3$}] {};
            \draw[color=red] (-2,-3.3) node [label={$Z_4$}] {};
            \draw[color=red] (1,-1.5) node [label={$Z_5$}] {};

			\draw [-{Latex[length=3mm]}] (4,0) to (6.5,0);

			\draw [red] (7,-1.6) to [out=90,in=240] (7.2,-0.3);
			\draw [red] (7,-1.6) to [out=340,in=240] (8.3,-1.5);
			\draw [red] (8.3,-1.5) to [out=45, in=265] (9,0);
			\draw [red] (7.2,-0.3) to [out=30,in=160] (9,0);
			\draw [red] (9,0) to [out=40,in=220] (11,1);

			\draw [fill] (7,-1.6) circle [radius=0.75mm]
			node [label={[left]$M_1$}] {};
			\draw [fill] (7.2,-0.3) circle [radius=0.75mm]
			node [label={$M_2$}] {};
			\draw [fill] (8.3,-1.5) circle [radius=0.75mm]
			node [label={[below,xshift=1.5mm,yshift=-2mm]$M_3$}] {};
			\draw [fill] (9,0) circle [radius=0.75mm]
			node [label={$M_4$}] {};
			\draw [fill] (11,1) circle [radius=0.75mm]
			node [label={$M_5$}] {};
		\end{tikzpicture}

		\caption{An example of the graph of a $b$-manifold (cyclic)}
		\label{figure:cyclic_graph}
	\end{figure}

The $2$-colorability of the graph of a $b$-manifold is equivalent to the existence of a global defining function for the critical set.

   \begin{lemma}\label{lem:globaldefiningfunction}\footnote{This lemma is basically a reformulation of Theorem 5.5. in \cite{mirandaplanas}.}
        Any $b$-manifold $(M,Z)$ admits a global defining function if and only if its associated graph is $2$-colorable.
    \end{lemma}
    \begin{proof}
    A globally defining function $f$ defines a sign on each connected component of $M\setminus Z$, and we can thus color the associated graph using the coloring of $\{+,-\}$.
    
    Vice versa, given a $2$-coloring $\{+,-\}$ of the associated graph, we can define a function that is positive on the connected component colored by $+$ and negative on the connected component colored by $-$, and vanishing on each connected component of $Z$ (which correspond to the vertices of the graph).
    \end{proof}
    
    Therefore, as acyclic $b$-manifolds are $2$-colorable, they admit a global defining function.  When the $b$-manifold is cyclic, the $b$-manifold admits a globally defining function if and only if the associated graph is $2$-colorable.

\subsection{Hamiltonian dynamics in the desingularization of $b^{2m}$-symplectic manifolds}

	In this subsection  we study the effect of the desingularization procedure in \cite{gmW1} of $b^{2m}$-symplectic manifolds, taking into account the dynamics defined by the Hamiltonian vector field of an admissible $b^{2m}$-Hamiltonian function.

\begin{proposition}\label{prop:b2mdesinsympl}
    Let $(M,Z,\omega)$ be a compact $b^{2m}$-symplectic manifold and let $H_t\in {^{b^{2m}}}\mathcal{C}^\infty(M)$ be an admissible Hamiltonian.
    Then the Hamiltonian vector field $X_H^\omega$ is a symplectic vector field with respect to the desingularization $\omega_\varepsilon$.
\end{proposition}

	\begin{proof}[Proof of Proposition \ref{prop:b2mdesinsympl}]
	The main idea of the proof is to desingularize the $b^{2m}$-symplectic form as in \cite{gmW1}, and to use the desingularization function to desingularize the Hamiltonian function as well: we will show that the admissible Hamiltonian can be desingularized in a tubular neighbourhood around the critical set. In this tubular neighbourhood, both vector fields agree and therefore this gives rise to symplectic dynamics.
		
		The desingularization of the $b^{2m}$-symplectic form as in Theorem \ref{th:deblog} in a tubular neighbourhood $\mathcal{N}_\varepsilon(Z)$ as in Equation \ref{def:desingularization} is given by
					$$\omega_\varepsilon = d f_\varepsilon \wedge \left ( \sum \limits_{i = 0}^{2m-1} z^i \pi^*(\alpha_i) \right ) + \beta,$$
		where $f_\varepsilon$ is defined as in Equation (\ref{def:feps}). Away from this neighbourhood, this form agrees with the initial one.
		
		By the assumptions of admissibility, the Hamiltonian is given in a tubular neighbourhood around $Z$ (as in Remark \ref{remark:localFormHamiltonian}) by
		$$H_t=-k(t)\frac{1}{2m-1}\frac{1}{z^{2m-1}}+h_t.$$
		The differential of $H_t$ is given by $dH_t=k(t)\frac{dz}{z^{2m}}+dh_t$ and thus the Hamiltonian vector field of $H_t$ with respect to the $b^{2m}$-symplectic form is given by $X_H^\omega=k(t)R+X_{h}^\beta$. Here $\beta$ is a symplectic form on the leaf of the symplectic foliation and by the condition of admissibility, the function $h_t$ is a function defined on the leaf of the symplectic foliation (see also Remark \ref{remark:localFormHamiltonian}).
		We define in the tubular neighbourhood $\mathcal{N}_\varepsilon(Z)$ the smooth Hamiltonian $\widehat{H}_t=-k(t)f_\varepsilon+h_t$.
		Similarly, the differential of $\widehat{H}$ is given by $d\widehat{H}_t=k(t)df_\varepsilon+dh_t$ and thus the Hamiltonian vector field of $\widehat{H}_\varepsilon$ with respect to the symplectic form $\omega_\varepsilon$ is given by $X_{\widehat{H}}^{\omega_\varepsilon}=k(t)R+X_{h}^\beta$.
 Thus both Hamiltonian vector fields agree, i.e. $X_H^\omega=X_{\widehat{H}}^{\omega_\varepsilon}$. As this is only a local construction, we conclude that $X_H^\omega$ by a symplectic vector field with respect to the symplectic structure $\omega_\varepsilon$.

 \end{proof}

 We will show that under additional conditions, the previous proposition can be strengthened, namely that the Hamiltonian vector field with respect to the $b^{2m}$-symplectic structure is a Hamiltonian vector field associated to the desingularization $\omega_\varepsilon$.

\begin{proposition}\label{prop:b2mHam=Ham}
		Let $(M,Z,\omega)$ be a compact $b^{2m}$-symplectic manifold with $Z$ being defined by a global defining function.
		Let $H_t\in {^{b^{2m}}}\mathcal{C}^\infty(M)$ be an admissible Hamiltonian and assume that $H_t$ is unimodular (see Definition \ref{def:constantlinearweight}).
		Then there exists a smooth Hamiltonian function $\widetilde{H}_{t}$ on $M$ such that $X^{{\omega}_\varepsilon}_{\widetilde{H}}$ coincides with $X^{\omega}_H$.
	\end{proposition}

	\begin{proof}[Proof of Proposition \ref{prop:b2mHam=Ham}]
    As in the proof of Theorem \ref{prop:b2mdesinsympl}, we will use the desingularization function to desingularize the $b^{2m}$-Hamiltonian function.
	 We will show that in the case where the critical set is defined by a global defining function and the Hamiltonian is unimodular, it can be \emph{globally} desingularized into a smooth Hamiltonian function $\widetilde{H}$. The desingularized dynamics still agrees with the initial one \emph{globally}, and we will therefore conclude that in the case where the defining function is globally defined the associated dynamics are Hamiltonian.
	
We denote by $M^+$, respectively $M^-,$ the set of connected components where the globally defined defining function $z$ is greater then $\varepsilon$, respectively smaller than $\varepsilon$, i.e. \begin{equation}\label{def:M+}
		    M^\pm=\{x\in M | \pm z >\varepsilon\}.
		\end{equation}
        Let $\mathcal{N}_\varepsilon(Z)$ be a tubular neighbourhood around the critical set as in Definition \ref{def:desingularization}.

        We desingularize the Hamiltonian function $H_t$ to a smooth Hamiltonian function $\widetilde{H}_t$ in such a way that the Hamiltonian vector field of $H_t$ (with respect to the $b^{2m}$-symplectic structure $\omega$) agrees with the Hamiltonian vector field of $\widetilde{H}_t$ (with respect to the desingularized symplectic form $\omega_\varepsilon$).

        We define the following function:

        \begin{equation}\label{eq:Hamdesiglobally}
            \widetilde{H}_t= \begin{cases}H_t-\frac{2}{\varepsilon^{2m-1}}k(t) \text{ on } M^- \\ k(t)f_\varepsilon(z)+h_t \text{ in } \mathcal{N}_\varepsilon(Z) \\ H_t+\frac{2}{\varepsilon^{2m-1}}k(t) \text{ on } M^+. \end{cases}
        \end{equation}

        By the definition of $\widetilde{H}_t$ and $f_\varepsilon$, $$\widetilde{H}_t(\varepsilon)=k(t)f_\varepsilon(\varepsilon)+h_t=-k(t)\frac{1}{(2m-1)\varepsilon^{2m-1}}+k(t)\frac{2}{\varepsilon^{2m-1}}+h_t,$$
        and $\lim\limits_{z\to \varepsilon^-} \widetilde{H}_t(z)=-k(t)\frac{1}{(2m-1)\varepsilon^{2m-1}}+h_t+\frac{2}{\varepsilon^{2m-1}}k(t)$. Therefore $\widetilde{H}_t$ defines a continuous function. Furthermore, as the derivatives agree on $\{z=\varepsilon\}$, we obtain a smooth function.

			As outside of $\mathcal{N}_\varepsilon(Z)$ the Hamiltonian $\widetilde{H}_t$ is given by $H_t$ up to a constant, and as $\omega_\varepsilon=\omega$, by construction $X_H^\omega=X_{\widetilde{H}}^{\omega_\varepsilon}$. In $\mathcal{N}_\varepsilon(Z)$, the computation is similar as was done in the first case.  This finishes the proof of Proposition \ref{prop:b2mHam=Ham}.
			
	\end{proof}

    \begin{remark}
        Note that due to Equation (\ref{eq:Hamdesiglobally}), the function $k(t)$ must be a \emph{globally} defined function, as it is added globally to ${M^\pm}$: for this reason, we ask that $H$ is unimodular, as in Definition \ref{def:constantlinearweight}.
    \end{remark}

    \begin{remark} 
        Example \ref{ex:b2torusHam} is not an example that satisfies the condition of Proposition \ref{prop:b2mHam=Ham}. Consider the $b^2$-symplectic manifold given by $(\mathbb{T}^2,\omega=\frac{d\theta_1}{\sin^2\theta_1}\wedge d\theta_2)$ and the $b^2$-Hamiltonian given by $H=k(t)\cot(\theta)$. This Hamiltonian is not unimodular. Indeed, if $H$ could be written as $k(t)\frac{1}{f}+h_t$ for a global defining function $f$, the global defining function would always a critical point away from the critical set (see also Lemma \ref{lem:globaldefiningfunction}), whereas $\cot$ is a strictly decreasing function.
    \end{remark}

    We will now see that in the case where the associated graph is acyclic, we can relax the conditions of the previous proposition: the time-dependent function $k$ does not need to be globally defined. This means that for each connected component of the critical set there is a locally defined time-dependent function $k_i$ such that the $b^{2m}$-Hamiltonian writes down as $H_t=-k_i(t)\frac{1}{(2m-1)z^{2m-1}}+h_t$.

    \begin{proposition}\label{prop:cyclicb2m}
         Let $(M,Z,\omega)$ be a compact acylic $b^{2m}$-symplectic manifold and let $H_t \in {^{b^{2m}}}C^\infty(M)$ be an admissible Hamiltonian.
         Then there exists a smooth Hamiltonian $\widetilde{H}_t\in C^\infty(M)$ such that $X^{\omega_\varepsilon}_{\widetilde{H}}=X^\omega_H$.
    \end{proposition}

    \begin{proof}
        As the associated graph is acylic, there exists a global defining function (see Lemma~\ref{lem:globaldefiningfunction}).
				By the assumptions of admissibility, the Hamiltonian is given in a tubular neighbourhood around a connected component of the critical set $Z_i$ (as in Remark \ref{remark:localFormHamiltonian}) by
		$$H_t=-k_i(t)\frac{1}{2m-1}\frac{1}{z^{2m-1}}+h_t.$$
        The differential of $H_t$ is given by $dH_t=k_i(t)\frac{dz}{z^{2m}}+dh_t$ and thus the Hamiltonian vector field of $H_t$ with respect to the $b^{2m}$-symplectic form is given by $X_H^\omega=k_i(t)R+X_{h_t}^\beta$. Here $\beta$ is a symplectic form on the leaf of the symplectic foliation and by the condition of admissibility, the function $h_t$ is a function defined on the leaf of the symplectic foliation (see also Remark \ref{remark:localFormHamiltonian}).

        We denote $M^\pm$ as in Equation \ref{def:M+}.
		 We start to define a smooth function $\widetilde{H}_t$ on a connected component of $M^+$ and we denote the corresponding vertex in the associated graph by $v_{initial}$. On this connected component, we set $\widetilde{H}_{t}=H_t$. We extend the definition of the function $\widetilde{H}_{t}$ by defining it on the connected components of $\mathcal{N}_\varepsilon(Z)$ corresponding to the edges incident to $v_{initial}$ and on the connected components of $M^-$ corresponding to the adjacent vertices of $v_{initial}$. We then reiterate this step to continue to extend $\widetilde{H}_t$ to the neighbouring connected components of $\mathcal{N}_\varepsilon(Z)$ and to the connected components of $M\setminus \mathcal{N}_\varepsilon(Z)$ corresponding to the adjacent vertices. The condition on the associated graph being acyclic implies that the graph does not close up and therefore this gives rise to a globally defined function $\widetilde{H}_t$. In order to simplify the wording, we use interchangeably the connected components of $M\setminus \mathcal{N}_\varepsilon(Z)$ (respectively connected components of $\mathcal{N}_\varepsilon(Z)$) and the vertices (respectively edges) of the associated graph $V$. For instance, when we write $v_{initial}$, we mean the connected component of $M\setminus \mathcal{N}_\varepsilon(Z)$ corresponding to $v_{initial}$ and vice versa.
		 
		 The extension of $\widetilde{H}_t$ goes as follows. Having already defined $\widetilde{H}_{t}$ on $v_{initial}$, we then define $\widetilde{H}_{t}$ on the edges incident to $v_{initial}$.
			
			In each edge $e_i$ incident to $v$, we define the Hamiltonian by $\widetilde{H}_{t} = f_\varepsilon(z)k_i(t)+h_t-k_i(t)\frac{2}{\varepsilon^{2m-1}} $.
			In this way, the function $\widetilde{H}_t$ defines a smooth function on $v_{initial}\cup \bigcup_{i} e_i$, as the function defined by parts are identical on a small neighbourhood of the intersection of $v\cap e_i^\pm$. Indeed, by the formula of $f_\varepsilon$, we obtain that for $z>\varepsilon$, 
   \begin{align*}
    \widetilde{H}_t&=f_\varepsilon(z)k_i(t)+h_t-k_i(t)\frac{2}{\varepsilon^{2m-1}}
       \\&=k_i(t)\Big(\frac{-1}{(2m-1)z^{2m -1}} + \frac{2}{\varepsilon^{2m-1}}\Big)+h_t-k_i(t)\frac{2}{\varepsilon^{2m-1}}
       \\&=-k_i(t)\frac{1}{(2m-1)z^{2m -1}}+h_t
    \end{align*}

   A similar computation yields that for $z < -\varepsilon$, $\widetilde{H}_t=-k_i(t)\frac{1}{(2m-1)z^{2m-1}}+h_t-\frac{4}{\varepsilon^{2m-1}}k_i(t)$.
   On the adjacent vertices $v_i$, we therefore define $\widetilde{H}_t=H_t-k_i(t)\frac{4}{\varepsilon^{2m-1}}$.
			As the function $f_\varepsilon$ satisfies $f_\varepsilon'>0$, the function $\widetilde{H}_{t}$ defines a smooth function on the domain of $M$ corresponding to $v_{initial}$, $e_i$ and $v_i$.

			We iterate this construction step by step over all the edges incident and vertices adjacent to the vertices of the domain of definition until the domain of definition of the function $\widetilde{H}_t$ is $M$.
			As the graph is acyclic, the function $\widetilde{H}_t$ is a well-defined smooth function on $M$.

			As in the proof of Proposition \ref{prop:b2mHam=Ham}, it can be checked that the Hamiltonian vector field of $\widetilde{H}_{t}$ associated to $(M, \widetilde{\omega}_{\varepsilon})$ equals to $X_H$, i.e. $X_H^\omega = X_{\widetilde{H}}^{\widetilde{\omega}_{\varepsilon}}$. This finishes the proof of Proposition \ref{prop:cyclicb2m}.
            
    \end{proof}

   \begin{remark}
        Once more, Example \ref{ex:b2torusHam} is not an example that satisfies the condition of Proposition \ref{prop:bHam=Ham}, as the graph is cyclic.
    \end{remark}

	\begin{remark} 
	    There are $b^{2m}$-symplectic manifolds where neither Proposition \ref{prop:b2mHam=Ham}, nor Proposition \ref{prop:cyclicb2m} directly apply, but combining the two results, we can still desingularize the $b^{2m}$-Hamiltonian function by applying ideas from both proofs.
			Consider for example the $b$-manifold as depicted in Figure \ref{figure:cyclic_graph} and a $b^{2m}$-Hamiltonian function, that has the semilocal expression $H_t=-k_1(t)\frac{1}{(2m-1)z^{2m-1}}+h_t$ in a tubular neighbourhood around $Z_1,Z_2,Z_3$ and $Z_4$, and the expression $H_t=-k_2(t)\frac{1}{(2m-1)z^{2m-1}}+h_t$ in a neighbourhood around $Z_5$.
			Here $z$ denotes a global defining function, which exists because the graph is two-colorable (see Lemma \ref{lem:globaldefiningfunction}).
			This $b^{2m}$-Hamiltonian function does not satisfy the conditions of Proposition \ref{prop:b2mHam=Ham}, and as the graph is not acyclic, Proposition \ref{prop:cyclicb2m} cannot be applied either.
			However, we can first apply the reasoning of Proposition \ref{prop:cyclicb2m} to the cyclic subgraph, and then finish the desingularization of the Hamiltonian function in the connected component denoted by $M_5$ following the proof of Proposition \ref{prop:bHam=Ham}.

		In short, the desingularization process described in Proposition \ref{prop:b2mHam=Ham} and in Proposition \ref{prop:cyclicb2m} can be applied to certain subsets of the manifold, as long as the right conditions hold in each of those subsets.
	\end{remark}

	Note that the above results only hold for $b^{2m}$-symplectic manifolds. In the case of $b^{2m+1}$-symplectic manifolds, the singular form cannot be desingularized generally due to topological constraints, as there exist manifolds that admit a $b$-symplectic structure but do not admit symplectic structures.
	See \cite[Figure 1]{cavalcanti} for an example.
	In the case of $b^{2m+1}$-symplectic manifolds the desingularization (Theorem \ref{th:deblog}) yields a so-called \emph{folded} symplectic form.
	In the next section, however, we will see that the desingularization can still be adapted in the particular case of surfaces.

\subsection{Desingularizing $b^m$-symplectic surfaces I}

As we saw in the previous section, $b^m$-symplectic manifolds together with an admissible Hamiltonian function can be desingularized as long as $m$ is an even number. In this subsection, we prove that in the case of surfaces, we can desingularize \emph{independently} on the parity of $m$.

	As will be proven, $b^m$-symplectic structures on orientable surfaces can be regularized into symplectic structures. While the case of $m$ even is already covered in Proposition \ref{prop:b2mHam=Ham} and \ref{prop:cyclicb2m}, we will show that the desingularization theorem can be adapted to the case of $m$ odd in the case of surfaces. We show that for $m$ odd and under similar conditions as in \ref{prop:b2mHam=Ham} and \ref{prop:cyclicb2m}, the Hamiltonian dynamics of an admissible $b^m$-Hamiltonian on a $b^m$-symplectic surface is given by the Hamiltonian dynamics of a smooth Hamiltonian on a symplectic manifold.

 For the sake of clarity, we distinguish the two cases in two different propositions.

	\begin{proposition}\label{prop:bHam=Ham}
		Let $(\Sigma,Z, \omega)$ a compact $b^m$-symplectic surface with $Z$ being defined by a globally defining function. Let $H_t$ be an unimodular admissible Hamiltonian (see Definition \ref{def:constantlinearweight}).
		Moreover, assume that $\Sigma$ is orientable.
		Then there exist a smooth symplectic structure $\widetilde{\omega}_\varepsilon$ and a smooth Hamiltonian $\widetilde{H}_t$ on $\Sigma$ such that $X^{\widetilde{\omega}_\varepsilon}_{\widetilde{H}}$ coincides with $X^{\omega}_H$.
	\end{proposition}

	The case when $m$ is even is already covered in Proposition \ref{prop:b2mHam=Ham}, so the novel case to consider is when  $m$ is odd. The proof of the proposition is based on an adaptation of the desingularization procedure introduced in \cite{gmW1}. However, this adaptation only works for surfaces\footnote{A different adaptation of the desingularization strategy for dynamical systems works in the integrable case in higher dimensions (see \cite[Chapter~7]{mirandaplanas2})}.

	\begin{proof}[Proof of Proposition \ref{prop:bHam=Ham}]

		The sketch of the proof is as in Proposition \ref{prop:b2mHam=Ham}: we will first construct a smooth symplectic form $\widetilde{\omega}_\varepsilon$ out of the $b^m$-symplectic one.
		We will also construct a smooth function $\widetilde{H}_t$ such that its Hamiltonian vector field with respect to $\widetilde{\omega}_\varepsilon$ coincides with the initial $b^m$-Hamiltonian vector field everywhere.
  
		As $Z$ is compact and $\Sigma$ is a surface, a connected component of $Z$ is diffeomorphic to $S^1$ and any tubular neighborhood of a connected component of the critical set is diffeomorphic to $ (-\varepsilon,\varepsilon) \times S^1$ and the $b^m$-symplectic structure in this neighborhood writes down as $\omega= \sum_{i=1}^m \frac{dz}{z^i} \wedge \pi^*(\alpha_i)$ where $\alpha_i\in \Omega^1(Z)$ are closed $1$-forms and $z$ is a defining function for $Z$. As $Z$ is one dimensional, $\alpha_i=a_i d\theta$, where $a_i$ are smooth functions on $S^1$, and we therefore rewrite 
		\begin{equation}\label{eq:Laurentseriessurface}
			\omega=\sum_{i=1}^m \big(z^{m-i}a_i \big) \frac{dz}{z^m}\wedge d\theta,
		\end{equation} 
		where $a_m(\theta) \neq 0$ for all $\theta\in S^1$ (as otherwise, $\omega$ would not be a $b^m$-symplectic form).

		By the assumptions on admissibility, the Hamiltonian is given by 
						$$H=\begin{cases}
		    k(t)\log|z|  \text{ when } m=1 \\
		    -k(t)\frac{1}{m-1}\frac{1}{z^{m-1}} \text{ when } m>1.
		\end{cases}$$
		We denote by $\Sigma^+$, respectively $\Sigma^-,$ the set of connected components where the globally defined defining function $z$ is greater than $\varepsilon$, respectively smaller than $-\varepsilon$ as in Equation (\ref{def:M+}).

		When $m=1$, let $g:(-\varepsilon,\varepsilon)\to \mathbb{R}$ be a smooth function such that $g_{\varepsilon}(x) = \log|x|$ in $\left(\frac{\varepsilon}2, \varepsilon\right)$ and $g_{\varepsilon}(x) =  -\log|x| +2\log(\frac{\varepsilon}{2})-\varepsilon$ in $\left(- \varepsilon, -\frac{\varepsilon}2\right)$. As $g_\varepsilon(\frac{\varepsilon}{2})-g_\varepsilon(-\frac{\varepsilon}{2})>0$, the function $g_\varepsilon$ can be smoothtly extended to $(-\frac{\varepsilon}{2},\frac{\varepsilon}{2})$ without having any critical points in this neighbourhood.
		
		When $m>1$, instead of the above function we consider $g_\varepsilon:(-\varepsilon,\varepsilon)\to \mathbb{R}$ a smooth function such that $g_{\varepsilon}(x) = -\frac{1}{m-1} \frac{1}{x^{m-1}}$ in $\left(\frac{\varepsilon}2, \varepsilon\right)$ and $g_{\varepsilon}(x) =  \frac{1}{m-1}  \frac{1}{|x|^{m-1}} - \frac{1}{m-1}\frac{2^{m+1}}{\varepsilon^{m-1}}$ in $\left(- \varepsilon, -\frac{\varepsilon}2\right)$. As $g_\varepsilon(\frac{\varepsilon}{2})-g_\varepsilon(-\frac{\varepsilon}{2})>0$, it can be extended to a smooth function in $(-\varepsilon,\varepsilon)$ without any critical points in the neighbourhood $(-\varepsilon,\varepsilon)$\footnote{Note that the function $g_\varepsilon$ resembles the desingularization function $f_\varepsilon$ (up to the addition of constants) in Equation \ref{eq:b2kDesing} in the case where $m$ is even.}.

		\begin{figure}
		\hspace*{-5cm}   
			\begin{minipage}{.25\textwidth}
				\begin{tikzpicture}
					\draw[->] (-3,0) -- (3,0);
					\draw[->] (0,-3) -- (0,3);
					\draw (3,0.1) node [label={$x$}] {};
					\draw [thick, cyan, samples=100, domain=0.15:3] plot ({\x},{(ln(abs(\x)))});
                    \draw[cyan] (3.5,1.5) node [label={$k\log|x|$}] {};
					\draw [thick, cyan, samples=100, domain=-3:-0.15] plot ({\x},{(ln(abs(\x)))});
				\end{tikzpicture}
			\end{minipage}
			\hspace{3cm}
			\begin{minipage}{.25\textwidth}
				\begin{tikzpicture}
					\draw[->] (-3,0) -- (3,0);
					\draw[->] (0,-3) -- (0,3);
					\draw (3,0.1) node [label={$x$}] {};
					\draw [thick, cyan, samples=100, domain=1.1:3] plot ({\x},{(ln(abs(\x)))});
					\draw [thick, dashed, cyan, samples=100, domain=0.15:1.1] plot ({\x},{(ln(abs(\x)))});
					\draw[cyan] (3.5,1.5) node [label={$k\log|x|$}] {};
					\draw [thick, green, samples=100, domain=-3:-0.15] plot ({\x},{-(ln(abs(\x)))});
					\draw[green] (-3.5,-1) node [label={$-k\log|x|$}] {};
					\draw [thick, brown, dashed, samples=100, domain=-1.1:-0.15] plot ({\x},{-(ln(abs(\x)))-2});
					\draw [thick, brown, samples=100, domain=-3:-1.1] plot ({\x},{-(ln(abs(\x)))-2});
					\draw[brown] (-3.5,-2.7) node [label={$-k\log|x|-C$}] {};
					\draw[red,thick] (1.1,0.1) -- (-1.1,-2.1);
				\end{tikzpicture}
			\end{minipage}
		\caption{An illustration of the regularization function $g_\varepsilon$ used to build a symplectic form and a smooth Hamiltonian out of the initial $b$-symplectic form and the $b$-Hamiltonian function.}
		\label{fig:f}
		\end{figure}
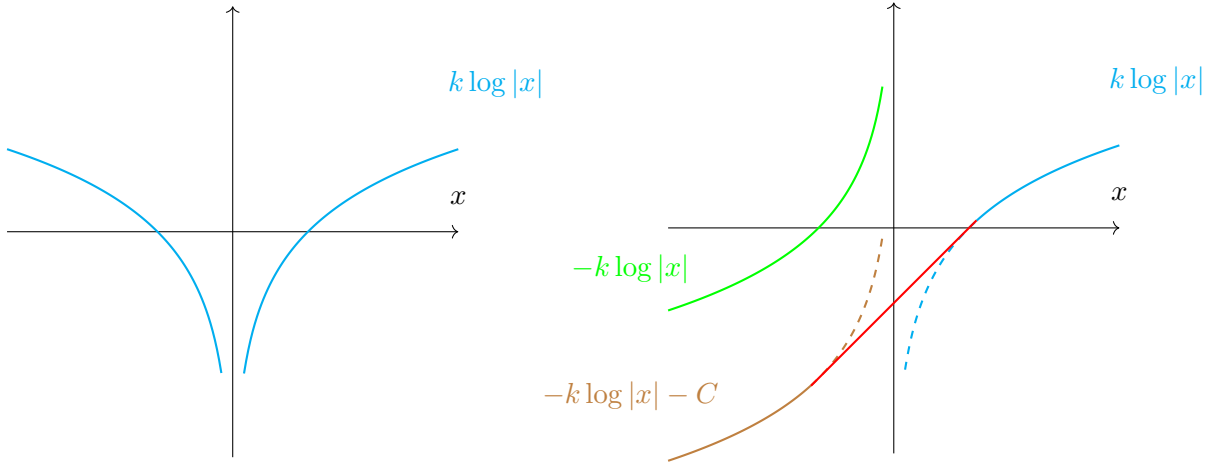

		We define a smooth  $2$-form on $\Sigma$ by
		$$ \widetilde{\omega}_{\varepsilon}= \begin{cases} \omega \text{ on } \Sigma^+ \\ (-1)^m\omega \text{ on } \Sigma^- \\ \sum_{i=1}^m \big(z^{m-i}a_i \big) dg_\varepsilon \wedge d\theta \text{ in } (-\varepsilon,\varepsilon) \times S^1. \end{cases}$$
		Note that the $2$-form $\widetilde{\omega}_\varepsilon$ is well-defined as the differential of $g_\varepsilon$ is well-defined and the form extends smoothly to $\Sigma^-$ and $\Sigma^+$.

        Similarly, we will define the Hamiltonian function by

        $$\widetilde{H}_t=\begin{cases}
            H_t \text{ in } \Sigma^+\\ k(t)g_\varepsilon(z) \text{ in } \mathcal{N}_\varepsilon(Z) \\ (-1)^m H_t-k(t)C(\varepsilon,m) \text{ in } \Sigma^-,
        \end{cases}$$

        where $C(\varepsilon,1)=-\log(\varepsilon)$ and $C(\varepsilon,m)=\frac{2^{m+1}}{\varepsilon^{m-1}}$ for $m>1$.

			We claim that $\widetilde{\omega}_{\varepsilon}$ is a symplectic $2$-form and that the Hamiltonian vector field of $\widetilde{H}_{\varepsilon}$ associated to $(\Sigma, \widetilde{\omega}_{\varepsilon})$ equals to $X_H$, i.e. $X_H^\omega = X_{\widetilde{H}}^{\widetilde{\omega}_{\varepsilon}}$.

           Indeed, as the changes do not affect $\Sigma^+$, the above statement holds in $\Sigma^+$. In $\Sigma^-$, as $d\widetilde{H}_t=(-1)^m dH_t$, and $\widetilde{\omega}=(-1)^m\omega$, it follows that both Hamiltonian vector fields agree. In the tubular neighbourhood, by definition of the Hamiltonian function, $X_H=-\frac{k(t)}{a_m}\frac{\partial}{\partial \theta}$. Similarly, by definition of $\widetilde{H}_t$, its Hamiltonian vector field with respect to $\widetilde{\omega}$ is given by $X^{\widetilde{\omega}}_{\widetilde{H}}=-\frac{k(t)}{a_m}\frac{\partial}{\partial \theta}$, and thus $X^{\widetilde{\omega}}_{\widetilde{H}}=X_H^\omega$.

           This finishes the proof.
	\end{proof}
	
 In the same way, we can do the singularization for $b^m$-symplectic surfaces whose associated graph is acylic without the condition that the Hamiltonian function is unimodular, which generalizes Proposition \ref{prop:cyclicb2m}.

 	\begin{proposition}\label{prop:acyclicbHam=Ham}
		Let $(\Sigma,Z, \omega)$ a $b^m$-symplectic surface whose associated graph is acyclic.
		Moreover, assume that $\Sigma$ is orientable.
		Then there exist a symplectic structure $\widetilde{\omega}_\varepsilon$ and a smooth Hamiltonian $\widetilde{H}_t$ on $\Sigma$ such that $X^{\widetilde{\omega}_\varepsilon}_{\widetilde{H}}$ coincides with $X^{\omega}_H$.
	\end{proposition}

 The proof is done by taking the symplectic form obtained as in Proposition \ref{prop:bHam=Ham}, and desingularizing the Hamiltonian as was done in Proposition \ref{prop:cyclicb2m}  for $b^{2m}$-Hamiltonians (but using the desingularization function of Proposition \ref{prop:bHam=Ham}).
	
	Next, we remark that the proof of Proposition \ref{prop:bHam=Ham} does not generalize to higher dimensions.
	
		\begin{remark} \label{remark:nothigherdim}
		The construction from Proposition \ref{prop:bHam=Ham} does not necessarily work in dimensions higher than $2$ for $m$ odd. We reason here for $m=1$, the case $m>1$ odd being similar.
		Nearby a connected component of $Z$ the $b$-symplectic form has the form $\omega = \frac{dz}{z} \wedge \alpha + \beta$.
		In contrast to dimension $2$, in higher dimensions there is a similar interpolation to glue $\omega$ on $M^+$ to $-\frac{dz}{z}\wedge \alpha+\beta$ on $\mathcal{N}(Z) \cap M^-$ as in the proof, however, a priori the latter does not extend to $M^-$.
	\end{remark}

While the argument does not generalize directly to higher dimensions, it applies to products of $b^m$-symplectic surfaces with symplectic surfaces (with the product $b^m$-symplectic structure).

	\begin{remark}\label{rem:product}
			The extension (as raised in Remark \ref{remark:nothigherdim}) holds for products of $b^m$-symplectic surfaces with symplectic manifolds provided that the $b^m$-symplectic form has a product structure.
			Let the $b^m$-symplectic structure on the product manifold be given by $\omega:=\omega_1+\omega_2$, where $\omega_1$ is the $b^m$-symplectic structure on the surface and $\omega_2$ is the symplectic structure on a symplectic manifold (with no restriction on the dimension).
			In that case, it is sufficient to realize that $\omega$ can be interpolated to $-\omega_1+\omega_2$ on $M^-$ using the methods just exposed.
			By the admissibility conditions, the Hamiltonian dynamics is just the product dynamics.

			The same reasoning can be applied to $c$-symplectic manifolds that arise from products of $b^m$-symplectic surfaces.
	\end{remark}

 We finish this subsection by proving a construction that is "converse" to Proposition \ref{prop:bHam=Ham}.

The following proposition shows that symplectic surfaces with a set of curves can be equipped with a $b^m$-symplectic structure. Furthermore, given a smooth Hamiltonian function with the right behaviour around the critical curves, we can define a $b$-Hamiltonian function such that both Hamiltonian vector fields define the same dynamics. 

\begin{proposition}[Singularization of a symplectic surface]\label{prop:singularization}
    Let $(\Sigma, \widetilde{\omega})$ be a symplectic surface. Let $(\gamma_i)_{i\in I}\subset \Sigma$ be a set of smooth curves given by the zero level set of a smooth global defining function $z:\Sigma\to \mathbb{R}$. Then there exists a $b^m$-symplectic structure $\omega$ on $\Sigma$ that has $\gamma_i$ as critical curves and agrees with $\widetilde{\omega}$ on $z>0$, respectively with $(-1)^m\widetilde{\omega}$ on $z<0$ outside an $\varepsilon$-neighbourhood of $Z$.
    Furthermore, if $\widetilde{H}\in C^\infty(\Sigma)$ is a smooth function such that in an $\varepsilon$-neighbourhood around $\gamma_i$, it is given by $\widetilde{H}(z,\theta)=z$, then there exists a $b$-function $H$ such that $X_{H}^\omega=X_{\widetilde{H}}^{\widetilde{\omega}}$. 
\end{proposition}

In \cite[Section 5]{cavalcanti} the singularization of symplectic structures is investigated in further generality. Nonetheless, the scrutiny of its dynamical aspects was not pursued in \cite{cavalcanti}.  This is precisely what we accomplish in the ensuing proposition:

\begin{proof}
    We denote by $\Sigma^\pm=\{z\in \Sigma| \pm z >0\}$.
    In an $\varepsilon$-neighbourhood around the curve $\gamma_i$, consider the $b^m$-function $s_\varepsilon:(-\varepsilon,0)\cup (0,\varepsilon)\times \gamma_i \to \mathbb{R}$ given by

    $$s_\varepsilon(z,\theta)= \begin{cases} \log|z| \text{ for } z\in (-\frac{\varepsilon}{2},0)\cup (0,\frac{\varepsilon}{2}) \text{ if } m=1 \\ -\frac{1}{(m-1)z^{m-1}} \text{ for } z\in (-\frac{\varepsilon}{2},0)\cup(0,\frac{\varepsilon}{2}) \text{ if } m>1 \\  z \text{ for } z>\frac{2\varepsilon}{3}\\ (-1)^m z \text{ for } z<-\frac{2\varepsilon}{3},\end{cases}$$
    For all $z$ in $(-\varepsilon,\varepsilon)\setminus \{0\}$, it is defined such that $\frac{\partial s_\varepsilon}{\partial z}(z, \theta)\neq 0$ and it glues smoothly with the  function $(-1)^m z$ in $\left(-\varepsilon,-\frac{2\varepsilon}{3}\right)$ and the identity function in $\left(\frac{2\varepsilon}{3},\varepsilon \right)$, and therefore defines a smooth globally defined function away from $\{z=0\}$. In the $\varepsilon$-neighbourhood, we define the following a $b^m$-form by $\omega:=s_\varepsilon' \widetilde{\omega}$. By the definition of the function $s_\varepsilon$, $\omega$ is a $b^m$-form that agrees for $z>\frac{2\varepsilon}{3}$ with $\widetilde{\omega}$, and for $z<-\frac{2\varepsilon}{3}$ with $(-1)^m\widetilde{\omega}$, and can thus be globally defined on $\Sigma$. By the condition that the derivative is non-vanishing, $\widetilde{\omega}$ is a $b^m$-symplectic form.
    Let $\widetilde{H}$ be a Hamiltonian as above. We define the Hamiltonian $H$ to be
    $$H=\begin{cases}\widetilde{H} \text{ when }z>\varepsilon\\ (-1)^m \widetilde{H} \text{ when } z<-\varepsilon \\  s_\varepsilon \text{ in the $\varepsilon$-neighbourhood.}\end{cases}$$
    A direct computation yields that both Hamiltonian vector fields agree in the $\varepsilon$-neighbourhood.
\end{proof}

For higher dimensions, the above reasoning breaks down in general. Note that a similar trick can be used to \emph{singularize} contact structures along convex surfaces, see \cite{bcontact}. However, as for higher dimensional symplectic manifolds, an equivalent result for Reeb dynamics is not true in general.

\subsection{Desingularizing $b^m$-symplectic surfaces II}

As shown in Proposition \ref{prop:bHam=Ham}, an acyclic $b^m$-symplectic surface can be desingularized into symplectic surfaces. This result is used in the following theorem to associate to each connected component of a $b^m$-symplectic surface $(\Sigma,\omega)$ a compact symplectic surface: indeed, given a connected component of $\Sigma\setminus Z$, we will glue $2$-disks to the critical sets bounding this connected component. The so-obtained $b^m$-symplectic manifold is by construction acyclic, and we, therefore, will be able to apply Proposition \ref{prop:bHam=Ham}.

 First, let us introduce some further notation. Consider a vertex $v$ of the associated graph of a $b$-manifold, and recall that $g_v$ denotes the genus of the corresponding connected component of $\Sigma \setminus Z$. Furthermore, we denote by $\mathrm{deg}(v)$ the degree of the vertex $v$, that is, in terms of graph theory, the number of edges incident to it. In the $b$-manifold this number corresponds to the number of connected components of $Z$ that are the boundary of the connected component of $\Sigma \setminus Z$, represented by $v$. With this notation in mind, we obtain the following theorem.

\begin{proposition}\label{thm:regularizingsurfaces}
	Let $(\Sigma,Z,\omega)$ be a compact $b^m$-symplectic orientable surface and let $H_t$ be an admissible Hamiltonian. Let $\Sigma_v$ be a connected component of $\Sigma\setminus Z$, $g_v$ the genus of $\Sigma_v$, and denote by $(\gamma_i)_{i\in I}\subset Z$ the connected components of $Z$ that bound $\Sigma_v$.
    Let $\mathrm{Int}(\Sigma_v):= \Sigma_v\setminus \mathcal{N}_\varepsilon(Z)$.
	Then there exists a closed symplectic surface $(\overline{\Sigma}_v,\omega_v)$ of genus $g_v$ such that $(\overline{\Sigma}_v \setminus \bigsqcup_{i\in I} D^2_i, \omega_v)$ is symplectomorphic to $(\mathrm{Int}(\Sigma_v),\omega)$.

	Furthermore, there exists a smooth function defined on $\overline{\Sigma}_v$ such that the flow of its Hamiltonian vector field agrees with $X_H$ on $\overline{\Sigma}_v\setminus \bigsqcup_{i\in I} D^2_i$.
\end{proposition}

As explained previously, the idea of the proof of Proposition \ref{thm:regularizingsurfaces} is to glue $2$-disks to the critical sets bounding $\Sigma_v$. We therefore start by analyzing the dynamics of the admissible Hamiltonian functions on the $2$-disk where the critical curve is given by the boundary.

\begin{remark} \label{remark:admissibleDisk}
	Given the $2$-disk $D^2$ with critical curve $Z=\partial D^2$, the $b^m$-symplectic structure writes down in a neighbourhood around $Z$ as Laurent series given by
	\[\omega = \sum_{i=1}^m a_i \frac{dr}{(1-r)^i} \wedge d\theta, \]
	where $a_i \in \mathcal{C}^\infty(S^1\times D^2)$ are smooth functions and $(r,\theta)$ are polar coordinates. Given an admissible Hamiltonian, it is given in the tubular neighborhood by
	\[H(t,r,\theta) = \left\{ \begin{array}{lc} k(t) \log|1 - r| & \text{if } m = 1 \\ -\frac{k(t)}{(m-1)(1-r)^{m-1}} & \text{if } m > 1 \end{array} \right. \]
		for some $k : S^1 \to \mathbb{R}$ such that $k(t) > 0$ and $\int_0^1 k(t) dt < \frac{2\pi}{a_m}$. Here $\frac{2\pi}{a_m}$ is the modular weight.
\end{remark}

	\begin{lemma} \label{lemma:diskgluing}

			Let $\mathcal{N}_{\varepsilon}(Z) = \{(r,\theta) \in D^2 \ | \ r > 1 - \varepsilon\}$ be an annulus around the boundary of the disk, and let $\omega$ a $b^m$-symplectic form on $\mathcal{N}_{\varepsilon}(Z)$ and $H_t$ an admissible Hamiltonian on $\mathcal{N}_{\varepsilon}(Z)$.
			Then there exist extensions of $\omega$ and $H_t$ to the whole disk such that the flow of $X_H$ has exactly one fixed point.
	\end{lemma}

	\begin{proof}
		As observed in the previous remark, $\omega$ and $H_t$ must be given in a tubular neighbourhood around $\gamma_i$ denoted by $\mathcal{N}_{\varepsilon}(\gamma_i)$ by the expressions described in Remark \ref{remark:admissibleDisk}. We trivially extend this to the whole disk.
		This means that we can choose such extensions in such a way that $X_H =-\frac{k(t)}{a_m}\frac{\partial}{\partial \theta}$ in the whole disk, which is smooth. The vector field $\frac{\partial}{\partial \theta}$ has exactly one fixed point in $(0,0) \in D^2$ and the time 1 flow of $X_H$ cannot have more fixed points because $0 < \int_0^1 k(t) dt < \frac{2\pi}{a_m}$.
	\end{proof}

    The above lemma will be used, as already mentioned, to glue $2$-disks to the connected component of $\Sigma\setminus Z$. We will now proceed to the proof of Theorem \ref{thm:regularizingsurfaces}.
	
	\begin{proof}[Proof of Proposition \ref{thm:regularizingsurfaces}]
    
		Let $v \in V$ be the vertex associated with a connected component of $\Sigma \setminus Z$ of degree $1$. Seen in isolation, the associated connected component of $\Sigma \setminus Z$, denoted by $\Sigma_v$ is diffeomorphic to a surface of genus $g_v$ punctured at one point and the attached component of $Z$ is diffeomorphic to $S^1$. We will now complete it to a $b^m$-symplectic surface, diffeomorphic to a surface of genus $g_v$ without punctures by attaching one $2$-disk to $\Sigma_v$.

        The initial $b^m$-symplectic form in the annulus around the connected component of the critical set has the expression $\sum_{i=1}^m a_i \frac{dr}{(1-r)^i} \wedge d\theta$. Using Lemma \ref{remark:admissibleDisk}, we attach to each of the critical curves $\gamma_i$ a $2$-disk to obtain a closed surface of genus $g_v$, denoted by $\overline{\Sigma}_v$, and the $b^m$-symplectic form can be extended to a globally defined $b^m$-symplectic form by extending the given one to the $2$-disks as in Lemma \ref{remark:admissibleDisk}. We also choose an extension $\overline{H}_v$ of the given Hamiltonian $H$ to this disk as in Lemma \ref{remark:admissibleDisk} in such a way that
		\begin{enumerate}
			\item near the boundary $\partial D^2$ the Hamiltonian vector field has the expression $X_{\overline{H}} =-\frac{k(t)}{a_m} \frac{\partial}{\partial \theta}$, where $k$ is the function associated to $H_t$ near the boundary $Z_i$.
			\item $X_{\overline{H}}$ vanishes exactly at one point in the interior of the disk.
		\end{enumerate}

		With these choices, we can define $(\overline{\Sigma}_v, \overline{\omega}_v, \overline{H}_v)$, where $\overline{\Sigma}_v$, as described before, is diffeomorphic to a closed surface of genus $g_v$. The $b^m$-symplectic form $\overline{\omega}_v$ and the Hamiltonian $\overline{H}_v$ are constructed by extending $(\omega, H)$ as explained earlier.
		This is a $b^m$-symplectic surface whose associated graph is given by the initial vertex $v$ of degree $1$ with one vertex of degree $1$.
		In particular, this graph is acyclic, and therefore the critical set  $\overline{\Sigma}_v$ is defined by a \emph{global} defining function.

		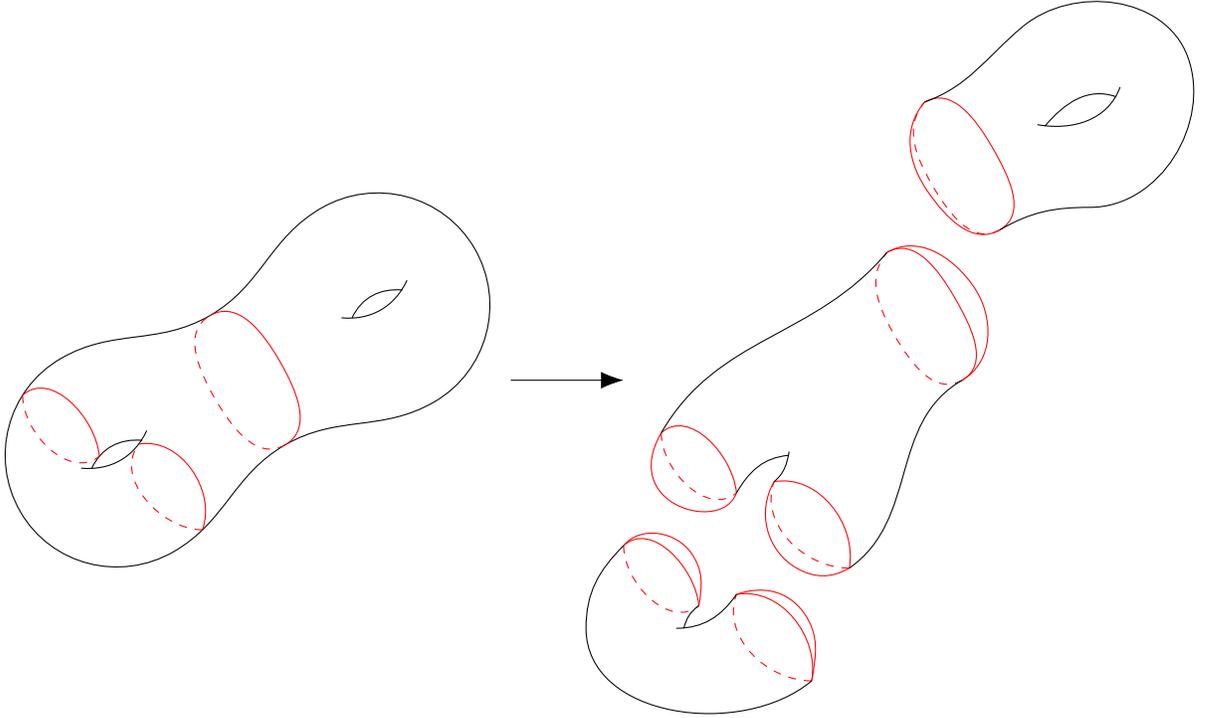
\begin{figure}[h]
			\centering
			\begin{tikzpicture}
				\draw [] (-4.98,0.29) to [out=210,in=120] (-5.52,-1.75) to [out=300,in=210] (-3.48,-2.28) to [out=30,in=210] (-2,-0.86) to [out=30,in=210] (-0.02,-0.29) to [out=30,in=300] (0.52,1.75) to [out=120,in=30] (-1.52,2.29) to [out=210,in=30] (-3,0.86) to [out=210,in=30] (-4.98,0.29);
				\draw [] (-4.57,-1.17) to [out=65,in=175] (-3.9,-0.8);
				\draw [] (-4.71,-1.17) to [out=355,in=245] (-3.84,-0.67);
				\draw [] (-1.11,0.83) to [out=65,in=175] (-0.44,1.2);
				\draw [] (-1.25,0.83) to [out=355,in=245] (-0.38,1.33);

				\draw [red] (-2,-0.86) to [out=30,in=300] (-2.07,0.25) to [out=120,in=30] (-3,0.86);
				\draw [red,dashed] (-2,-0.86) to [out=210,in=300] (-2.93,-0.25) to [out=120,in=210] (-3,0.86);

				\draw [red] (-3.1,-1.99) to [out=70,in=10] (-3.94,-0.85);
				\draw [red,dashed] (-3.1,-1.99) to [out=180,in=230] (-3.94,-0.85);

				\draw [red] (-5.49,-0.2) to [out=42,in=95] (-4.47,-1.01);
				\draw [red,dashed] (-5.49,-0.2) to [out=270,in=220] (-4.47,-1.01);

				\draw [-{Latex[length=3mm]}] (1,0) to (2.5,0);

				\draw [red] (2.5,-2.2) to [out=42,in=95] (3.5,-3);
				\draw [red,dashed] (2.5,-2.2) to [out=270,in=220] (3.5,-3);
				\draw [red] (2.5,-2.2) to [out=50,in=125] (3.4,-2.3) to [out=305,in=80] (3.5,-3);

				\draw [red] (5,-4) to [out=80,in=10] (4,-2.85);
				\draw [red,dashed] (5,-4) to [out=180,in=250] (4,-2.85);
				\draw [red] (5,-4) to [out=80,in=305] (4.9,-3.1) to [out=125,in=20] (4,-2.85);

				\draw [] (2.5,-2.2) to [out=225,in=90] (2,-3.3) to [out=270,in=220] (5,-4);
				\draw [] (4,-2.85) to [out=235,in=0] (3.2,-3.3);
				\draw [] (3.5,-3) to [out=215,in=75] (3.3,-3.3);

				\draw [red] (3,-0.7) to [out=42,in=95] (4,-1.5);
				\draw [red,dashed] (3,-0.7) to [out=270,in=220] (4,-1.5);
				\draw [red] (3,-0.7) to [out=240,in=130] (3,-1.5) to [out=310,in=250] (4,-1.5);

				\draw [red] (5.5,-2.5) to [out=80,in=10] (4.5,-1.35);
				\draw [red,dashed] (5.5,-2.5) to [out=180,in=250] (4.5,-1.35);
				\draw [red] (4.5,-1.35) to [out=240,in=150] (4.8,-2.5) to [out=330,in=210] (5.5,-2.5);

				\draw [] (4,-1.5) to [out=60,in=185] (4.69,-1);
				\draw [] (4.5,-1.35) to [out=40,in=260] (4.7,-0.95);

				\draw [red] (7,0) to [out=30,in=300] (6.9,1.1) to [out=120,in=30] (6,1.7);
				\draw [red,dashed] (7,0) to [out=210,in=300] (6.1,0.6) to [out=120,in=230] (6,1.7);
				\draw [red] (7,0) to [out=30,in=305] (7.15,1.25) to [out=125,in=30] (6,1.7);

				\draw [] (3,-0.7) to [out=60,in=230] (6,1.7);
				\draw [] (5.5,-2.5) to [out=35,in=210] (7,0);

				\draw [red] (7.5,2) to [out=30,in=300] (7.4,3.1) to [out=120,in=30] (6.5,3.7);
				\draw [red,dashed] (7.5,2) to [out=210,in=300] (6.6,2.6) to [out=120,in=230] (6.5,3.7);
				\draw [red] (7.5,2) to [out=210,in=305] (6.55,2.5) to [out=125,in=230] (6.5,3.7);

				\draw (6.5,3.7) to [out=20,in=220] (7.8,4.7) to [out=40,in=110] (10,4.3) to [out=290,in=0] (8.7,2.3) to [out=180,in=30] (7.5,2);
				\draw (8,3.4) to [out=350,in=250] (9.1,3.9);
				\draw (8.1,3.38) to [out=50,in=160] (9.05,3.77);
			\end{tikzpicture}

			\caption{Completion of a $b^m$-symplectic surface with disks at the punctures}
			\label{figure:fillings}
		\end{figure}
  
		{The case when $v \in V$ has a degree greater than $1$ is analogous.
		In this case, the associated connected component of $\Sigma \setminus Z$ is diffeomorphic to a surface of genus $g_v$ punctured at $\mathrm{deg}(v)$ points and the components of $Z$ adjacent to $\Sigma_v$ are diffeomorphic to $S^1 \sqcup \dots \sqcup S^1$.
		We can use Lemma \ref{lemma:diskgluing} to complete it to a closed surface of genus $g_v$ by attaching a $2$-disk to each connected component of the critical set.
		The $b^m$-symplectic form and the $b^m$-Hamiltonian are extended to the completed surface as before.}
        
        As a result, we obtain a $b^m$-symplectic surface whose associated graph is given by the initial vertex $v$ with one vertex of degree $1$ for each disk attached. In particular, this graph is acyclic and the critical set is therefore defined by a global defining function.

        We can therefore apply Proposition \ref{prop:bHam=Ham} to this new surface, so the Hamiltonian vector field and the $b^m$-symplectic structure can be regularized to a smooth Hamiltonian vector field on a smooth symplectic surface. The resulting flow has $\mathrm{deg}(v)$ trivial periodic orbits on the attached disks.
  \end{proof}

\subsection{Desingularizing $b^m$-symplectic manifolds with trivial mapping tori}

In this subsection, we further generalize the previous results. We will show that the techniques used for the proof of $b^m$-symplectic surfaces can be used in higher dimensions whenever the induced structure on the critical set is that of a trivial mapping torus.
Here, trivial mapping torus means that in Equation (\ref{eq:mappingtorus}) the map $\phi$ is the identity.

Given a $b$-manifold, we denote by $M_v$ a connected component of $M\setminus Z$. Similarly, we denote by $Z_v$ the union of connected components of $Z$ that bound $M_v$.

In the following proposition, we show that under the assumption that the mapping tori associated with the  $b^m$-symplectic structure are trivial, we can cut along the hypersurfaces $Z_i$ and glue solid mapping tori to the hypersurfaces. We thus obtain for each connected component of $M\setminus Z$ a closed manifold. We will first prove that these closed manifolds can be equipped with a $b^m$-symplectic structure (Proposition \ref{prop:completionbsymplecticmanifold}) and then that this $b^m$-symplectic structure, together with an admissible Hamiltonian function, can be desingularized into a smooth symplectic structure (Corollary \ref{cor:regularisingbsymphd}).

\begin{proposition}\label{prop:completionbsymplecticmanifold}
    Let $(M,Z,\omega)$ be a $b^m$-symplectic manifold with trivial mapping tori. A connected component of $M\setminus Z$, denoted by $M_v$, can be compactified to a closed $b^m$-symplectic manifold $(\overline{M}_v,Z_v,\overline{\omega})$.
    The manifold $\overline{M}_v$ is diffeomorphic to the closed manifold obtained by gluing to $v$ a solid mapping torus for each connected component $e_v$ of $\partial M_v$.
    Furthermore, there exists a $b^m$-symplectic form $\overline{\omega}$ on $\overline{M}_v$ that  coincides with $\omega$ on $M_v$. Given an admissible Hamiltonian function $H_t \in {^{b^m}}\mathcal{C}^\infty(M)$, it gives rise to an admissible Hamiltonian function $\overline{H}_t \in  {^{b^m}}\mathcal{C}^\infty(\overline{M}_v)$ on the $b^m$-symplectic manifold $(\overline{M}_v,Z_v,\overline{\omega})$, such that Hamiltonian vector field $X_{H_t}$ coincides with $X_{\overline{H}_t}$ on $M_v$.
\end{proposition}

Note that the associated graph to the $b^m$-symplectic manifold $(\overline{M}_v,Z_v,\overline{\omega})$ is acyclic by construction.

\begin{proof}
    As the critical sets are diffeomorphic to trivial mapping tori by assumption, a connected component of $Z_v$ is diffeomorphic to $S^1 \times \mathcal{L}$, where $\mathcal{L}$ is a leaf of the symplectic foliation on the connected component of $Z_v$. We consider the solid mapping torus given by $D^2\times \mathcal{L}$, which can be glued to the connected component of $e_v$ along the boundary of the disk. Furthermore, the $b^m$-symplectic form in a neighbourhood of $Z_v$ is given by $\omega=\frac{dz}{z^m}\wedge d\theta+\beta$, where $\beta$ is a symplectic form on $\mathcal{L}$ and $d\theta$ is a volume form on $S^1$. On the solid mapping torus, we consider the $b^m$-symplectic form given by ${^{b^m}}\omega+\beta$, where ${^{b^m}}\omega$ is the $b^m$-symplectic form on $D^2$ as given in Lemma \ref{lemma:diskgluing}.
    
    By the admissibility of the Hamiltonian function $H_t$, in a neighbourhood around the critical set, it is given by $H_t = k(t)\log|z|+h_t$ when $m=1$ (or $H_t = k(t)\frac{1}{z^{m-1}}+h_t$ when $m>1$), where $h_t\in \mathcal{C}^\infty(\mathcal{L})$ is a smooth function on the leaf of the symplectic foliation. It can be extended to the solid mapping torus as a sum of the Hamiltonian as in Lemma \ref{lemma:diskgluing} and $h_t$ on the symplectic leaf $\mathcal{L}$.
\end{proof}

Note that around the mapping torus, the $b^m$-symplectic structure obtained is the product of a $b^m$-symplectic surface and a symplectic manifold $(\mathcal{L},\beta)$. Therefore, the regularization of the $b^m$-symplectic structure into a symplectic structure works, as observed in Remark \ref{rem:product}.

\begin{corollary}\label{cor:regularisingbsymphd}
    The $b^m$-symplectic manifold $(\overline{M}_v,Z_v,\overline{\omega})$ can be regularized to a symplectic manifold $(\overline{M}_v,\widetilde{\omega})$. Furthermore, given an admissible $b^m$-Hamiltonian function $H_t\in {^{b^m}}\mathcal{C}^\infty(M)$ on the initial $b^m$-symplectic manifold $(M,Z,\omega)$, the Hamiltonian flow of the extension $\overline{H}_t$ on $\overline{M}_v$ is Hamiltonian with respect to a smooth Hamiltonian function $\widetilde{H}_t$ and the symplectic structure $\widetilde{\omega}$.
\end{corollary}

\begin{proof}
The $b^m$-symplectic manifold $(\overline{M}_v,Z_v,\overline{\omega})$ can be regularized as in Proposition \ref{prop:bHam=Ham} because the solid mapping tori that are attached to $M_v$ are trivial, and the $b^m$-symplectic structure is a product of a $b^m$-symplectic surface and the symplectic manifold $(\mathcal{L},\beta)$.
As for the Hamiltonian flow, the Hamiltonian function $\overline{H}_t$ can be desingularized as in the proof of Proposition \ref{prop:bHam=Ham} once again because the mapping torus is assumed to be trivial. The function is globally defined because the $b^m$-symplectic manifold $(\overline{M}_v,Z_v,\overline{\omega})$ is acyclic. It then follows from the construction that the $b^m$-Hamiltonian vector field $X_{\overline{H}}$ is given by the Hamiltonian vector field of $\widetilde{H}_t$ associated with the symplectic structure $\widetilde{\omega}$.
\end{proof}

We conclude this section by noting that the singularization as in Proposition \ref{prop:singularization} works under the condition that the given hypersurfaces as trivial mapping tori:

\begin{proposition}[Singularization of a symplectic manifolds admitting trivial mapping tori]\label{prop:singularizationhigherdimensions}
    Let $(M, \widetilde{\omega})$ be a symplectic manifold. Let $\bigsqcup Z_i$ a set of smooth hypersurfaces given by the zero-level set of a smooth defining function $z:M\to \mathbb{R}$ and each of the $Z_i$ is a trivial mapping torus. Then there exists a $b^m$-symplectic structure $\omega$ on $M$ that has $Z_i$ as critical curves and agrees with $\widetilde{\omega}$ on $z>0$, respectively with $(-1)^m\widetilde{\omega}$ on $z<0$ outside an $\varepsilon$-neighbourhood of $Z$.
    Furthermore, if $\widetilde{H}\in C^\infty(M)$ is a smooth function such that in an $\varepsilon$-neighbourhood around $\gamma_i$, it is given by $\widetilde{H}(z,\theta)=z$, then there exists a $b$-function $H$ such that $X_{H}^\omega=X_{\widetilde{H}}^{\widetilde{\omega}}$. 
\end{proposition}

The above proposition was previously proved in \cite{cavalcanti} without taking into account the dynamical aspects.

\begin{proof}
    Under the condition that $Z_i$ is a trivial mapping tori, there exists a tubular neighbourhood around $Z_i$ such that $\omega=\alpha\wedge \frac{dz}{z^m} +\beta$, where $\alpha,\beta\in \Omega^\bullet(Z)$ are closed forms of degree $1$, respectively $2$. The $b^m$-symplectic form in the tubular neighbourhood around $Z$ is then defined by $\omega:=ds_\varepsilon\wedge \alpha + \beta$, where the function $s_\varepsilon$ is the $b^m$-function as defined in Proposition \ref{prop:singularization}. The proof then follows as in Proposition \ref{prop:singularization}.
\end{proof}

	\section{The Arnold conjecture through desingularization}\label{section:arnold}

	This section contains a proof of the Arnold conjecture for some $b^m$-symplectic manifolds.
	More precisely, we give a lower bound for the number of periodic orbits of an admissible Hamiltonian in terms of the topology of $M$ and the relative position of $Z$ in $M$. The strategy will be to use the desingularization techniques developed in the previous section.
	First, we provide a proof for any $b^{2m}$-symplectic manifold, with no restriction on the dimension.
	We then prove an improved lower bound on the number of periodic orbits for $b^{m}$-symplectic surfaces $(\Sigma, Z, \omega)$ for general $m$.
	We generalize this lower bound for $b^m$-symplectic manifolds of higher dimensions under the condition that $Z$ is a union of trivial mapping tori.

 \subsection{The Arnold conjecture  for $b^{2m}$-symplectic manifolds}

We present two associated corollaries, obtained as a consequence of Proposition \ref{prop:b2mHam=Ham}.

For $b^{2m}$-symplectic manifolds (that are acyclic or under the condition that the Hamiltonian function is unimodular), we recover the same lower bound as in the Arnold conjecture for smooth symplectic manifolds:

	\begin{corollary}[Arnold conjecture $b^{2m}$-symplectic manifolds] \label{coro:b2marnold_acyclic}\label{thm:b2msymplecticArnold}
		Let $(M,Z,\omega)$ be a compact $b^{2m}$-symplectic manifold whose critical set is given by a global defining function $f$. Let $H_t\in {^{b^{2m}}}C^\infty(M)$ be a time-dependent admissible Hamiltonian function.
		Assume that 
  \begin{itemize}
        \item either the associated graph of the $b$-manifold is acyclic, or that
        \item $H_t$ is unimodular.
    \end{itemize}
		Suppose that the solutions of period $1$ of the associated Hamiltonian system are non-degenerate.
		Then
		\[\# \mathcal{P}(H) \geq \sum_i \dim HM_i(M;\mathbb{Z}_2).\]
	\end{corollary}

\begin{proof}
    By Proposition \ref{prop:b2mHam=Ham}, respectively Proposition \ref{prop:cyclicb2m}, we can desingularize the $b^{2m}$-symplectic manifold into a symplectic manifold. Furthermore, the Hamiltonian flow of the admissible $b^{2m}$-Hamiltonian is given by the Hamiltonian flow of a smooth Hamiltonian function. We therefore conclude by the Arnold conjecture on symplectic manifolds that the number of $1$-periodic orbits is bounded from  below by
    $$\mathcal{P}(H) \geq \sum_i \dim HM_i(M;\mathbb{Z}_2).$$
\end{proof}

	As we can semi-locally desingularize the Hamiltonian function as:

	\begin{corollary}[Arnold conjecture for $b^{2m}$-symplectic manifolds (being defined by a local defining function)] \label{coro:b2marnold_cyclic}
		Let $(M,Z,\omega)$ be a closed $b^{2m}$-symplectic manifold of dimension $2n$, with $Z$ being defined by a \emph{local} defining function. Further assume that $M$ aspherical, i.e. $\pi_2(M)=0$.
    Let $H$ be an admissible $b^{2m}$-Hamiltonian and denote by $\phi$ the flow of $X_H$.
		Suppose that $\phi$ is a $b^{2m}$-symplectomorphism on $(M,Z,\omega)$ which is isotopic to the identity through $b^{2m}$-symplectomorphisms.

		If all the fixed points of $\phi$ are non-degenerate, then the number of fixed points of $\phi$ is at least the sum of the Betti numbers of the Novikov homology over $\mathbb{Z}_2$ associated with the Calabi invariant of $\phi$.
	\end{corollary}

\begin{proof}
    The proof is once more a corollary of Proposition \ref{prop:b2mdesinsympl}. The resulting desingularization yields a symplectic manifold with a symplectic flow. The lower bound then follows from the main theorem of \cite{vanono}.
\end{proof}	

	Note that this lower bound is zero in the case of $\mathbb{T}^2$, and therefore does not contradict Example \ref{ex:arnoldtorus}.

 \begin{remark}
     The main theorem in \cite{vanono} is stated for negative monotone symplectic manifolds. One could try and generalize this definition for $b^m$-symplectic manifolds and prove that this definition is stable under desingularization.
		 This would yield a stronger version of Corollary \ref{coro:b2marnold_cyclic}.
 \end{remark}
 
	\subsection{The Arnold conjecture for $b^m$-symplectic surfaces}

In this section, we improve the previous results of the Arnold conjecture for $b^m$-symplectic manifolds. Indeed, the above results only hold for $m$ even, as the desingularization theorem (Theorem \ref{th:deblog}) does not yield symplectic structures for $m$ odd.

	However, in the case of surfaces, Proposition \ref{prop:bHam=Ham} clears the way to prove a better lower bound on the number of $1$-periodic orbits of the Hamiltonian flow. Indeed, the flow of an admissible $b^m$-Hamiltonian on a $b^m$-symplectic surface is on each connected component of $\Sigma \setminus Z$  Hamiltonian (for a smooth symplectic form). We therefore can use the known result about the dynamics on symplectic manifolds. In particular, we can state the same corollaries as for general $b^{2m}$-symplectic manifolds (Corollary \ref{coro:b2marnold_acyclic} and \ref{coro:b2marnold_cyclic}), however, we will improve the lower bound in the following proposition.

	\begin{theorem} \label{prop:lowerbound_surfaces}
		Let $(\Sigma,Z,\omega)$ be a compact $b^m$-symplectic orientable surface. For a vertex $v$ corresponding to a connected component of $\Sigma\setminus Z$, let $g_v$ be its genus and $\mathrm{deg}(v)$ its degree respectively.
		Let $H_t$ be an admissible Hamiltonian in $(\Sigma, Z, \omega)$ whose periodic orbits are all non-degenerate.
Then the number of $1$-periodic orbits of $X_H$ is bounded from below by 
	\[\# \mathcal{P}(H) \geq \sum_{v\in V} \big(2g_v + |\mathrm{deg}(v) - 2| \big).\]
	\end{theorem}

	\begin{proof}
  		We will prove that for each connected component of $\Sigma \setminus Z$ of degree $1$, there are at least $2g_v+1$ periodic orbits, whereas, for the connected components of degree $2$ or more, there are at least $2g_v + \mathrm{deg}(v)-2$ periodic orbits.

		Let $v \in V$ be the vertex associated with a connected component of $\Sigma \setminus Z$ of degree $1$. We apply Theorem \ref{thm:regularizingsurfaces} to $v$ and obtain a closed symplectic surface $(\overline{\Sigma}_v, \overline{\omega}_v)$ and a smooth Hamiltonian function $ \overline{H}_v$ such that its Hamiltonian vector field agrees with the initial one on the connected component of $\Sigma \setminus Z$.
        Thus the smooth Hamiltonian vector field must have at least $2g_v+2$ periodic orbits by the Arnold conjecture for smooth symplectic manifolds. As there is exactly one $1$-periodic orbit in the glued-in disk, we conclude that for the vertex of degree $1$, there are at least $2g_v+1$ periodic orbits in the connected component of $M\setminus Z$ corresponding to $v$.

        We now treat the case when $v \in V$ is the vertex associated to a connected component of $\Sigma \setminus Z$ of degree strictly higher than $1$. In this case, Proposition \ref{thm:regularizingsurfaces} yields as before a closed symplectic surface $(\overline{\Sigma}_v, \overline{\omega}_v)$ and a smooth Hamiltonian function $ \overline{H}_v$ such that its Hamiltonian vector field agrees with the initial one on the connected component of $\Sigma \setminus Z$. On the $\deg(v)$ glued-in disks in $\overline{\Sigma}_v$,  the resulting flow has $\mathrm{deg}(v)$ trivial periodic orbits on the attached disks. These trivial periodic orbits have by construction Conley-Zehnder index $\pm 1$ as they are local maxima and minima of the smooth Hamiltonian.

  We will now use the Morse inequalities for the Floer complex\footnote{The authors would like to thank Marco Mazzucchelli for this observation.} to show that the number of critical points for a Hamiltonian flow with at least $\mathrm{deg}(v)$ periodic orbits of Conley-Zehnder index $\pm 1$ is bounded from below by $2\mathrm{deg}(v)+2g_v-2$.

	Denote by $c_i = \{x \in \mathcal{P}(H) \ | \ \mu_{CZ}(x) = i\}$, the dimension of the $i$-th Floer complex group, and by $b_i$ the dimension of the $i$-th Floer homology group.
	The Euler characteristic is given by the alternated sum of the $b_i$, that is
	\begin{equation} \label{eq:morse_ineq}
		2-2g_v=b_0-b_1+b_2=c_{-1}-c_0+c_1.
	\end{equation}

    As there at least $\mathrm{deg}(v)$ $1$-periodic orbits of Conley-Zehnder index $\pm 1$, that is $c_{-1}+c_1\geq \mathrm{deg}(v)$, we thus obtain that $c_1\geq 2g_v-2+\mathrm{deg}(v)$.
    Using again that $c_{-1} + c_1 \geq \mathrm{deg}(v)$, we thus obtain that the number of periodic orbits is bounded from below by
	\[c_{-1} + c_0 + c_1 \geq 2\mathrm{deg}(v) + 2g_v - 2 .\]

	Knowing that there are $\mathrm{deg}(v)$ $1$-periodic orbits contained in the disks that were glued to the connected component of $M\setminus Z$, we deduce that the number of $1$-periodic orbits in $v$ is bounded from below by $\mathrm{deg}(v) + 2g_v - 2$.
	\end{proof}

	Let us note that the lower bound obtained in Theorem \ref{prop:lowerbound_surfaces} is an improvement of the lower bound obtained in Corollary \ref{coro:b2marnold_acyclic} in the acyclic case.

As is proved in the Appendix, the lower bound in Theorem \ref{prop:lowerbound_surfaces} is sharp: given any $b^m$-symplectic surface, there exists an admissible $b^m$-function such that it has exactly the number of $1$-periodic orbits as given by the lower bound in Theorem \ref{prop:lowerbound_surfaces}.

{As discussed in Remark \ref{rem:product}, the extension method of Proposition \ref{prop:bHam=Ham} holds for products of $b^m$-symplectic surfaces and symplectic manifolds. The resulting dynamics is the product dynamics. Under the same hypothesis, we, therefore, obtain the following corollary (where we adopt the same notation as in Theorem \ref{prop:lowerbound_surfaces})}.

\begin{corollary}[Arnold conjecture for $b^m$-symplectic product manifolds]
	Let $(M=\Sigma\times W,Z=Z_{\Sigma}\times W,\omega=\omega_1 + \omega_2)$ be the product of a compact orientable $b^m$-symplectic surface $(\Sigma,Z_{\Sigma},\omega_1)$ and a compact symplectic manifold $(W,\omega_2)$. Let $H_t$ be an admissible Hamiltonian in $(M,Z,\omega)$ whose $1$-periodic orbits are all non-degenerate. Then the number of $1$-periodic orbits of $X_H$ is bounded from below by
	$$\# \mathcal{P}(H) \geq \left(\sum_i \dim HM^i (W;\mathbb{Z}_2) \right) \cdot \left( \sum_{v \in V} \big(2g_v + |\mathrm{deg}(v) - 2| \big)\right).$$
\end{corollary}

\subsection{Arnold conjecture for $b^m$-symplectic manifolds with trivial mapping tori}

Using the construction of cutting along trivial mapping tori and gluing in solid mapping tori, we obtain the following theorem.

\begin{theorem} \label{prop:lowerboundtrivialmapping}
		Let $(M,Z=\bigsqcup Z_i,\omega)$ be a compact $b^m$-symplectic manifold of dimension $2n$ with trivial mapping tori $Z_i$.
		Let $H_t$ be an admissible Hamiltonian in $(M, Z, \omega)$ whose periodic orbits are all non-degenerate.
Then the number of $1$-periodic orbits of $X_H$ is bounded from below by 
		\[\sum_{v \in V} \max\biggl\{\sum_{i=0}^{2n} \dim HM^i(\overline{M}_v;\mathbb{Z}_2)- \sum_{e_j}\mathcal{P}(X_{h}|_{Z_j}), 0\biggr\} .\]
\end{theorem}

Here $h$ denotes the local function given by the expression of an admissible Hamiltonian as in Remark \ref{remark:localFormHamiltonian}.

As before, in this proof we denote by $M_v$ the connected components of $M\setminus Z$ that corresponds to the vertex $v$ and by $\overline{v}$ the completion to a closed manifold as in Corollary \ref{cor:regularisingbsymphd}. By $Z_v$ we denote the connected components of $Z$ that in the associated graph correspond to edges, incident to $v$.

We remark that in contrast to Theorem \ref{prop:lowerbound_surfaces}, this is not a topological invariant because of the second term that counts the number of periodic orbits of $X_{h}$ in $Z_i$.

\begin{proof}
    As the mapping torus is trivial, we apply Corollary \ref{cor:regularisingbsymphd} to each connected component of $M\setminus Z$. We thus obtain a symplectic manifold $(\overline{M}_v,\widetilde{\omega})$ and a smooth Hamiltonian function whose Hamiltonian flow coincides on the given connected component of $M\setminus Z$. By the Arnold conjecture applied to this symplectic manifold, we obtain that there are at least
    $$\sum_{i=0}^{2n} \dim HM^i (\overline{M}_v,\mathbb{Z}_2)$$
    periodic orbits on $(\overline{M}_v,\widetilde{\omega})$.
    However, by the condition of admissibility, there are no periodic orbits around the hypersurfaces $Z_v$. Furthermore, by the construction of the solid mapping tori, the Hamiltonian vector field flow is given by 
$$ X_{H}=k(t)\frac{\partial}{\partial \theta}+X_{h},$$
where $X_{h}$ denotes the Hamiltonian flow of $h_t\in \mathcal{C}^\infty(L)$ with respect to the symplectic structure $\beta$ on $L$.
    This is the product of a rotation on $D^2$ with the flow of the Hamiltonian function $h_t$ on $L$, and by the assumption on admissibility, periodic orbits are given thus by periodic orbits of $k(t)\frac{\partial}{\partial \theta}$ in $D^2$ and $X_{H}$ in $L$. As the flow of $k(t)\frac{\partial}{\partial \theta}$ is periodic in the center of the disk, the only periodic orbits of $X_{H}$ are given by periodic orbits of $X_{h}$ in $L$.
\end{proof}

	\section{A Floer complex on $b^m$-symplectic manifolds}\label{sec:floer}

In the preceding section a proof of the Arnold conjecture could be attained in the surface case for $b^m$-symplectic manifold.
Floer developed a whole toolbox of Floer theory in order to address the Arnold conjecture (see \cite{floer1} for the surface case).
This motivates us to introduce the Floer complex for higher dimensional $b^m$-symplectic manifolds.

	\subsection{The Floer equation}

To control the behavior of the $J$-holomorphic curves approaching the critical set $Z$, we need the following proposition.

	\begin{proposition} \label{prop:max_principle}
		Let $(M,Z,\omega)$ a $b^m$-symplectic manifold, and let $\mathcal{N} = \mathcal{N}(Z) \cong ((-\varepsilon,0)\cup(0, \varepsilon)) \times Z$ a tubular neighbourhood of the singular hypersurface (not including $Z$) with a normal symplectic vector field $X^{\sigma}$.
		Let $\Omega \subset \mathbb{C}$ with coordinates $\eta = s + it$, and take $H$ an admissible Hamiltonian in $\mathcal{C}^{\infty}(\Omega \times \mathcal{N})$, so that locally it has the form
		\[H(s,t,z,x) = \left\{ \begin{array}{lc} k(s,t) \log |z| + h(s,t,x) & \text{ if } m = 1 \\ - k(s,t) \frac{1}{(m-1)z^{m-1}} + h(s,t,x) & \text{ if } m > 1 . \end{array} \right.\]
		Take also $J \in \Gamma(\Omega \times \mathcal{N}, T^{\ast}M\otimes TM)$ a compatible almost complex structure adapted to $\omega$.

		Let $u : \Omega \rightarrow \mathcal{N} \subset M$ be a solution to the Floer equation,
		\begin{equation} \label{eq:floer_equation}
			\frac{\partial u}{\partial s} + J(\eta,u(\eta)) \left( \frac{\partial u}{\partial t} - X_{H}(u(\eta)) \right) = 0,
		\end{equation}

		and take $f : \mathcal{N} \rightarrow \mathbb{R}$ given by $f(z,x) = \log |z|$ if $m = 1$ and $f(z,x) = - \frac{1}{(m-1)z^{m-1}}$ if $m > 1$.

		Then
		\[\Delta (f\circ u) = - \frac{\partial k}{\partial s} .\]
	\end{proposition}

	\

	\begin{proof}
		Let $d^c (v) := dv \circ i = \frac{\partial v}{\partial t} ds - \frac{\partial v}{\partial s} dt$.
		Then
		\[- dd^c(v) = (\Delta v) ds\wedge dt .\]

		Computing,
		\begin{align}
			- d^c(f\circ u) =&\frac{\partial  }{\partial s} (f\circ u) dt - \frac{\partial  }{\partial t} (f\circ u) ds = \left(df(u) \frac{\partial u}{\partial s} \right) dt - \left(df(u) \frac{\partial u}{\partial t} \right) ds  \nonumber \\
			=&\left(df(u) \left( \frac{\partial u}{\partial s} + J(\eta,u) \frac{\partial u}{\partial t} \right)\right) dt - \left(df(u) \left(J(\eta,u) \frac{\partial u}{\partial t} \right)\right) dt +  \nonumber \\
			&+ \left(df(u) \left(J(\eta,u) \frac{\partial u}{\partial s} - \frac{\partial u}{\partial t} \right)\right) ds - \left(df(u) \left(J(\eta,u) \frac{\partial u}{\partial s} \right)\right) ds  \nonumber \\
			=&- \omega(\nabla f(u), X_H(u)) dt + \omega\left(\nabla f(u), \frac{\partial u}{\partial t} \right) dt - \nonumber \\
			&- \omega\left(\nabla f(u) , J(\eta,u) X_H(u) \right) ds + \omega\left(\nabla f(u), \frac{\partial u}{\partial s} \right) ds \label{eq:finalfloer}
		\end{align}

    Recall that in these coordinates the normal symplectic $b^m$-vector field is given by $X^\sigma=z^m\frac{\partial}{\partial z}$. On the other hand, the metric $\omega(\cdot,J\cdot)$ satisfies Equation (\ref{eq:exactbmetric}), from which we can deduce that $\nabla f=X^\sigma$.

		Now, we apply the fact that with our choice of $J$ we have $\nabla f = X^{\sigma}$, so, for the first term
		\[\omega( X^{\sigma}, X_H(u) ) dt = \left( \mathcal{L}_{X^{\sigma}} H \right) dt = k(s,t) dt.\]
		For the second and fourth terms
		\[\omega \left(X^{\sigma} , \frac{\partial u}{\partial t} \right) dt + \omega\left( X^{\sigma} , \frac{\partial u}{\partial s} \right) ds =  u^{\ast} i_{X^{\sigma}} \omega .\]
		Finally, for the third term
		\begin{gather*}
			\omega (X^{\sigma} , J(\eta,u) X_H(u) ) ds = \omega(X_H(u) , J(\eta,u) X^{\sigma}(u)) ds = \\ \omega(X_H(u), R) ds = \left(\mathcal{L}_R H \right) ds = 0 .
		\end{gather*}

		Collecting everything, we deduce that Equation (\ref{eq:finalfloer}) yields that
		\[ - d^c(f\circ u) = - k(s,t) dt - u^{\ast} i_{X^{\sigma}} \omega .\]

		If we apply the differential, it is clear that $d (k(s,t) dt) = \frac{\partial k}{\partial s} ds \wedge dt $, and
		\[d \left(u^{\ast} i_{X^{\sigma}} \omega\right) = u^{\ast} (d i_{X^{\sigma}} \omega) = u^{\ast} (\mathcal{L}_{X^{\sigma}} \omega) = 0 .\]

		Therefore,
		\[(\Delta (f\circ u)) ds\wedge dt = - dd^c (f\circ u) = - \frac{\partial k}{\partial s} ds \wedge dt. \]
	\end{proof}
	
	As a corollary of Proposition \ref{prop:max_principle}, we obtain the following result.

	\begin{theorem}[Minimum Principle] \label{theorem:min_principle}
		Let $u \in \mathcal{C}^{\infty}(\Omega, \mathcal{N})$ satisfying one of the following conditions:
		\begin{enumerate}
			\item $u$ is a solution to the Floer equation \ref{eq:floer_equation} for an admissible Hamiltonian $H \in \mathcal{C}^{\infty}(S^1 \times \mathcal{N})$. 	If $f \circ u$ attains its maximum or minimum on $\Omega$, then $f \circ u$ is constant.
			\item $u$ is a solution to the Floer equation \ref{eq:floer_equation} for a parameter-dependent admissible Hamiltonian $H \in \mathcal{C}^{\infty}(\Omega \times \mathcal{N})$ such that $ \frac{\partial k}{\partial s}(s,t) \geq 0 \ \forall (s,t) \in \Omega$. If $f \circ u$ attains its minimum on $\Omega$, then $f\circ u$ is constant.
		\end{enumerate}

	\end{theorem}

	\begin{proof}
		In the first case, this is straightforward since, according to Proposition \ref{prop:max_principle}, it follows that $f \circ u$ is harmonic and therefore satisfies the Maximum (respectively Minimum) Principle.

		In the second case, it is a consequence of the minimum principle, because $\Delta (f \circ u) = - \frac{\partial k}{\partial s} \leq 0 $ by the choice of the admissible Hamiltonian.
	\end{proof}

	\subsection{The Floer complex}
	
 For the remainder of this section, we assume that $(M,\omega)$ is aspherical, this means, that $[\omega]$ vanishes on $\pi_2(M)$.
	Recall from Definition \ref{def:regularHam} that we denote by $\mathcal{P}(H)$ the set of 1-periodic orbits of $X_H$ for an admissible and regular Hamiltonian $H : M \times S^1 \rightarrow \mathbb{R}$.
	Since $M$ is compact and the image of all elements of $\mathcal{P}(H)$ lies outside of an open neighborhood of $Z$, $\mathcal{P}(H)$ is finite.

	\begin{definition}
		Let $H_t \in {}^{b^m}\mathrm{Reg}(M,\omega)$.
		We define ${}^{b^m}CF(M,\omega,H)$ as the $\mathbb{Z}_2$-vector space generated over $\mathcal{P}(H)$, this means, the set of formal sums of the type
		\[v = \sum_{x \in \mathcal{P}(H)} v_x x , \ v_x \in \mathbb{Z}_2 .\]
	\end{definition}

	Under the assumption that the first Chern class $c_1 = c_1(\omega) \in H^2(M,\mathbb{Z})$ of the bundle $(TM, J)$ vanishes on $\pi_2(M)$, then the Conley-Zehnder index $\mu_{CZ}$ of $x \in \mathcal{P}(H)$ is well-defined (see for instance \cite{salamon}).
	This index can be normalized in such a way that for any critical points of a $\mathcal{C}^2$-small enough $H$ it is satisfied that
	\[\mu_{CZ}(x) = 2n - \mathrm{ind}(x),\]
 where $\mathrm{ind}(x)$ is the Morse index of $H$ at $x$.

	We can use the Conley-Zehnder index to turn ${}^{b^m}CF(M,H,\omega)$ into a graded vector space.

	For $x,y \in \mathcal{P}(H)$ we denote by $\mathcal{M}(x,y)$ the moduli space of Floer trajectories from $x$ to $y$, this means,
	\[\left\{ \begin{array}{l} \frac{\partial u}{\partial s} + J(u) \frac{\partial u}{\partial t} + \mathrm{grad}_u H = 0 \\ \displaystyle\lim_{s \rightarrow - \infty} u(s,t) = x(t) , \ \lim_{s \rightarrow +\infty} u(s,t) = y(t) . \end{array} \right.\]

		As $(M,\omega)$ is aspherical, we cannot have bubbles of pseudo-holomorphic spheres (see, for instance, Section 6.6 on \cite{audin}).
		Moreover, as a consequence of Theorem \ref{theorem:min_principle} we cannot have solutions of the Floer equation approaching $Z$ in any way.
		Thus, we conclude that $\mathcal{M}(x,y)$ is compact for any pair $x,y \in \mathcal{P}(M)$ and that it is a manifold of dimension $\mu_{CZ}(x) - \mu_{CZ}(y)$.
		As in standard Floer theory, we conclude from here that whenever $\mu_{CZ}(x) - \mu_{CZ}(y) = 1$, the quotient $\raisebox{.1em}{$\mathcal{M}(x,y)$}\left/\raisebox{-.1em}{$\mathbb{R}$}\right.$ by the action along the variable $s$ is a finite set see \cite[Chapter~9]{audin}).

		Let
		\[n(x,y) := \# \left\{\raisebox{.1em}{$\mathcal{M}(x,y)$}\left/\raisebox{-.1em}{$\mathbb{R}$}\right. \right\} \mod 2.\]

		Then, for each index $k$ we can define the boundary operator of the Floer complex,
		\[\begin{array}{rccc} \partial_k : & {}^{b^m}CF_k(M,\omega, H, J) & \longrightarrow & {}^{b^m}CF_{k-1}(M,\omega, H, J) \end{array}\]

			as defined in the generators of ${}^{b^m}CF_k(M,\omega, H, J)$ by
			\[\partial_k x := \sum_{\substack{y \in \mathcal{P}(H) \\ \mu_{CZ}(y) = k -1}} n(x,y) y .\]

			\begin{definition} \label{definition:floer_homology}
				The Floer homology is the one given by
				\[{^{b^m}}HF_{k}(M,\omega,H) := \frac{\mathrm{ker} \partial_k}{\mathrm{im} \partial_{k+1}} .\]
			\end{definition}

			\begin{remark}
				The homology constructed so far, as the notation implies, depends on the choice of $H$, $J$, and $\omega$, besides depending on $M$, $Z$, and the relative topology between them.
				To be precise, the family of admissible Hamiltonian functions depends on $X^{\sigma}$ and $R$, but we do not include this dependence in the notation as they are accounted for in the choice of a Hamiltonian $H$.

				Showing invariance with respect to these choices is beyond the scope of this paper, and is the subject of a paper in preparation by the authors.
			\end{remark}

			\begin{remark}
				The construction of this complex (and homology) is related to our methods in Section \ref{sec:desing} due to the conditions on admissible Hamiltonian functions, in particular to the fact that the dynamics of $X_H$ is, in a sense, split on the connected components of $M \backslash Z$.
				Moreover, due to Theorem \ref{theorem:min_principle} we can deduce that solutions to the Floer equation with finite energy do not cross the singular hypersurface $Z$, this means, our Floer complex splits on the connected components of $M\backslash Z$,
				\[{}^{b^m}CF_{\bullet} (M,\omega,H,J) = \bigoplus_{M_i \in M\backslash Z} {}^{b^m}CF_{\bullet}(M_i,\omega,H,J) .\]
			\end{remark}

			\section{Conclusions and Open Questions}\label{sec:final}

			In this article, we were able to give a complete proof of the Arnold conjecture for $b$-symplectic surfaces.
			
			Given the $b^3$-symplectic structure of the restricted three-body problem investigated in \cite{kms}, the first open question would be the following:
			\begin{openquestion}
				Is it possible to find a compact surface on the restricted three-body problem where the Arnold conjecture yields  new applications in Celestial mechanics?
			\end{openquestion}

		As shown in the previous section, by the definition of admissible Hamiltonian functions, there  are only finitely many $1$-periodic orbits around the critical set, and by the minimum principle (Theorem \ref{theorem:min_principle}), there is a Floer complex associated to such a Hamiltonian. We furthermore proved in Theorem \ref{prop:lowerbound_surfaces} that for  orientable surfaces, the number of $1$-periodic orbits of the Hamiltonian vector field is bounded from below by the topology of $M$ and the relative position of $Z$ and that in the case of $b^{2m}$-symplectic manifolds, the dynamics is described by the usual Hamiltonian dynamics and therefore the same lower bound as in the standard Arnold conjecture holds.

The second open question that we raise is whether there is a \emph{finer way} to compute the Floer homology associated with an admissible Hamiltonian (both for $b$-symplectic and for $b^{2m}$-symplectic manifolds). Because of the Mazzeo-Melrose formula for $b$-cohomology (see Theorem 27 in \cite{guimipi}) we might speculate with other decomposition formulas for the Floer complex that take $Z$ into account.

\begin{openquestion}
	Is it possible to compute ${}^{b^m}HF_k(M,\omega,H)$ in terms of the topology of $M$ and $Z$?
\end{openquestion}

Indeed, we suspect that in the $b^{2m}$-symplectic case, the lower bound obtained in Theorem \ref{thm:b2msymplecticArnold} is unsatisfactory. As for the proof of the surface case (Theorem \ref{prop:lowerbound_surfaces}), the fact that the dynamics is tangent to $Z$ should give a higher lower bound. In that sense, an optimal lower bound should take into account not only the topology of $M$, but also the one of $Z$ and its relative position in $M$. Similarly, we expect that a similar result to the one of Theorem \ref{prop:lowerbound_surfaces} holds in higher dimensions.  Our method of proof in the case of surfaces is a hands-on construction in dimension 2.

\begin{openquestion}
Does Theorem \ref{prop:lowerbound_surfaces} hold in higher dimensions?
\end{openquestion}

In Section \ref{sec:dyn}, we justify the notion of admissible Hamiltonian by the fact that there are only finitely many periodic orbits around the critical set. It would be interesting to construct a Floer-type homology for more general Hamiltonian functions. At this moment, this looks terra incognita to us.

\begin{openquestion}\label{open:floer_homology}
Can we associate a Floer homology with more general $b^m$-Hamiltonian functions?
\end{openquestion}

More precisely, this question raises whether one could possibly enlarge the set of admissible Hamiltonian functions considered in this paper.

Another open question that remains is studying smooth Hamiltonian functions on $b^m$-symplectic manifolds.
Indeed, for smooth Hamiltonian functions, the associated Hamiltonian vector field to a $b^m$-symplectic structure is tangent to the leaves of the foliation. Even compact $b$-symplectic manifolds can have non-compact leaves, as shown in the next example.

\begin{example}
	Consider the $4$-torus $M=\mathbb T^4$ with coordinates $\theta_1,\theta_2,\theta_3, \theta_4$. Inside this torus, consider $Z$ the union of two $3$-tori given by the equation $\sin \theta_4=0$.
	The critical set $Z$ can be endowed with a cosymplectic structure which is also a regular Poisson structure:
	Denote by $\mathcal{F}$ the codimension
	one foliation with leaves given by
	$$\theta_3=a\theta_1+b \theta_2+k,\quad k\in\mathbb R,$$
	where $a,b, 1\in\mathbb R$ are fixed and independent
	over $\mathbb Q$. Each leaf is diffeomorphic to $\mathbb R^2$
	\cite{mamaev}. The one-form
	$\alpha=\frac{a}{a^2+b^2+1}\,d\theta_1+\frac{b}{a^2+b^2+1}\,d\theta_2-\frac{1}{a^2+b^2+1}\,d\theta_3$
	defines a symplectic foliation with symplectic form induced by the form
	$$\beta=d\theta_1\wedge d\theta_2+b\, d\theta_1\wedge d\theta_3-a\,d\theta_2\wedge d\theta_3$$
	The  two-form
	$\omega= \frac{d\theta_4}{\sin\theta_4} \wedge \alpha+ \beta$ is a $b$-symplectic form on $M$.

	Now consider the vector field
	$$X=\frac{a}{a^2+b^2+1}\frac{\partial}{\partial\theta_1}+\frac{b}{a^2+b^2+1}\frac{\partial}{\partial\theta_2}-\frac{1}{a^2+b^2+1}\frac{\partial}{\partial\theta_3}.$$
	Due to the rational independence of $a$ and $b$ this vector field does not have any periodic orbits. To check this, it is enough to consider the projection on the $2$-torus with coordinates $\theta_1$ and $\theta_2$ where the vector field projects to a vector field with dense orbits.
	This vector field is the Reeb vector field associated to the cosymplectic structure on $Z$.
\end{example}

It would be interesting to see whether, even if the leaves are non-compact, there is a lower bound (depending on the topology of the leaf) on the number of $1$-periodic Hamiltonian orbits associated to a \emph{smooth} Hamiltonians on each leaf of $Z$.

\begin{openquestion} \label{open:leaf_bound}
Given a $b^m$-symplectic manifold with a smooth Hamiltonian function $H$, is the number of non-degenerate periodic time-$1$ Hamiltonian orbits bounded on each leaf bounded from below by the topology of $Z$?
\end{openquestion}

Given the recent advances in \cite{pinowitte}, by Theorem 1.3 therein, the $b^m$-symplectic structure can be regularized (in the sense of \cite{pinowitte}) to a symplectic foliation. The last open question can be rephrased as to whether the foliated version of the Arnold conjecture holds. To put it simply, the analogous statement for odd dimensions is true, as proven in \cite{foliatedWeinstein}. The foliated Arnold conjecture was first discussed at the \emph{"Workshop on topological aspects of symplectic foliations"} in 2017 at the Université de Lyon.


\appendix

\section{The lower bound in Theorem \ref{prop:lowerbound_surfaces} is optimal}

In this section, we will prove that the lower bound in Theorem \ref{prop:lowerbound_surfaces} is optimal. More precisely, we will prove the following proposition.

\begin{proposition}\label{prop:optimalbound}
	Let $(\Sigma,\omega)$ be a compact orientable $b^m$-symplectic surface with critical set $Z = \bigsqcup \gamma_i$, where each $\gamma_i$ is diffeomorphic to a circle.
		Then there exists an admissible $b$-function $F: \left( \Sigma, \bigsqcup \gamma_i \right) \to \mathbb{R} \cup \{\pm \infty \}$ such that $X_F$ has exactly
		\[\sum_{v\in V} \big(2g_v+|\mathrm{deg}(v)-2|\big)\]
		$1$-periodic orbits, where $V$ is the set of vertices associated to the graph of the $b$-manifold.
\end{proposition}

The proposition follows from Proposition \ref{prop:singularization} combined with fairly basic Morse theory on surfaces.

\begin{proof}

   First, we sketch the idea of the proof. 
    
    \begin{enumerate}
        \item Given a $b^m$-symplectic surface, we will consider each connected component of $\Sigma\setminus Z$, denoted by $\Sigma_v$, and construct a Morse function on it that has exactly $2g_v+|\mathrm{deg}(v)-2|$ critical points (here $v$ denotes the associated vertex in the graph). 
        \item To do so, we first construct a smooth Morse function on $\overline{\Sigma}_v$ (as in Proposition \ref{thm:regularizingsurfaces}) having $2 + 2g_v + 2(\mathrm{deg}(v) - 2)$ critical points. Of those critical points, $\mathrm{deg}(v)$ critical points (of index $2$ or $0$, i.e. maxima or minima) will be chosen to lie in the disks $D_v$, given by the connected components of $\overline{\Sigma}_v\setminus \Sigma_v$.
        \item In the next step, we singularize the smooth Morse functions to obtain $b^m$-functions and the aim is to glue all of those $b^m$-functions (defined on $\Sigma_v$ for each $v$) together to a global $b^m$-function defined on $\Sigma$.
    \end{enumerate}
    For the $b^m$-function to be globally defined, the signs of the singularization need to match. In Step ($2$), it is therefore important to choose whether to place a minimum or a maximum in a connected component of $\overline{\Sigma}_v\setminus \Sigma_v$, denoted by $D_v$. Hence, before providing a proof following the sketch above we start by fixing supplementary data of the graph.
		This data will determine whether the Morse function constructed in Step (2) will have a minimum or a maximum in a given disk that is glued to ${\Sigma}_v$. The parity of $m$ is important: a $b^{2m}$-function changes sign around a connected component of the neighbourhood, whereas a $b^{2m+1}$-function has the same sign in a tubular neighbourhood. We therefore distinguish two cases. Let $\Sigma_1$ and $\Sigma_2$ be two connected components of $\Sigma\setminus Z$, that correspond to adjacent vertices in the graph, and let $Z_{12}$ be the connected component of $Z$ that corresponds to the edge that joins these vertices. 
    \begin{itemize}
        \item In the case of a $b^{2m+1}$-symplectic surface, we will have to construct the Morse functions in Step (2) such that it has either a maximum, or a minimum, on both disks that compactify $\Sigma_1$ and $\Sigma_2$ along $Z_{12}$. We denote by $D_i$ the disk compactifying $\Sigma_i$ along $Z_{12}$, for $i=1,2$.
        \item In the case of a $b^{2m}$-symplectic surface, we will have to construct the Morse functions such that it has a maximum on one disk being glued to $\Sigma_1$ along $Z_{12}$, and a minimum on the one that is being glued to $\Sigma_2$ along $Z_{12}$ as in the proof of Proposition \ref{thm:regularizingsurfaces} ({or} vice-versa, a minimum \emph{and} maximum).
    \end{itemize}
    To make these choices, we fix additional data of the graph:
    \begin{itemize}
        \item in the case of a $b^{2m+1}$-symplectic surface, we fix a $2$-coloring of the edges, that we denote for simplicity by $\{+,-\}$. The color $+$  (respectively $-$) means that the Morse function will have a maximum  (respectively minimum) in the connected component of $D_1$ and $D_2$ that is glued to $Z_{12}$.
        
        \item in the case of $b^{2m}$-symplectic surface, we choose an orientation of the edges of the graph. The orientation indicates on which of the compactifying disks we put a minimum, respectively maximum. We choose the convention to assign a $+$-sign to a connected component of $\mathcal{N}_\varepsilon(Z_i)\setminus Z_i$ that corresponds to the terminal vertex associated to the edge $Z_i$, respectively a $-$-sign to the initial vertez associated to the edge $Z_i$ (see \cite[p.28]{diestel} for the definition of terminal and initial vertex).
    \end{itemize}
    To satisfy that the function that we construct has the least possible number of critical points, the coloring (in the $b^{2m+1}$-symplectic case) or the orientation (in the $b^{2m}$-symplectic case) has to be chosen as follows: the graph needs to impose that in Step (2), for each $\Sigma_v$ having degree greater than 1 in the associated graph, there is at least one connected component of $\overline{\Sigma}_v\setminus \Sigma_v$ in which the Morse function has a maximum, and one in which it has a minimum. In fact, if all the compactifying disks of a $\Sigma_v$ had only maxima (or similarily, all of them had only minima), then in Step (2) of the construction, the Morse function would not satisfy the minimal bound.

 In the case of $b^{2m+1}$-symplectic surfaces, this translates to the fact that the $2$-coloring of the edges needs to be chosen in such a way that every vertex of degree 2 or more has incident edges of both colors.
	By an elementary result in graph theory, a graph admits such an edge coloring if and only if it contains no cycle with an odd number of vertices.
	This is always true for the graph associated with a $b^{2m+1}$-symplectic surface because, in particular, these graphs always admit a vertex two-coloring (see for instance \cite{mirandaplanas}).
	Thus, as the graph admits a vertex two-coloring, it does not contain a cycle with an odd number of vertices.
	For simplicity, let us denote the edge 2-coloring by the labels $\{-, +\}$.

    In the case of $b^{2m}$-symplectic surfaces, we equip the edges with an orientation.
		More precisely, we want to orient the edges such that for each vertex $v$ of degree greater than 1 there exists at least one edge incident to $v$, whose initial vertex is $v$, and at least one edge incident to $v$ whose terminal vertex is $v$.
		We call such an orientation a \emph{good} orientation, and a graph always admits a good orientation, as is proved in Lemma \ref{lem:goodorientation}. We thus fix a good orientation of the graph.

    In this way, both for $b^{2m}$- and $b^{2m+1}$-symplectic surfaces, we obtain an associated sign to $\mathcal{N}_\varepsilon(Z_i)\setminus Z_i$ for each connected component $Z_i$ of $Z$.

	For each $\Sigma_v$ we construct a closed symplectic manifold $\left( \overline{\Sigma}_v, {\overline{\omega}}_v \right)$ as in Proposition \ref{thm:regularizingsurfaces}.
	Then $\overline{\Sigma}_v$ contains $\mathrm{deg}(v)$ connected components of $Z$, still labeled with either the sign $+$ or~$-$. By construction, each of these connected components bounds a $2$-disk.

	As $\overline{\Sigma}_v$ is orientable, it admits a perfect Morse function, this means, a Morse function that has the minimal number possible of critical points, in our case $2 + 2g_v$.
	Such a Morse function can be manipulated in Morse coordinates by adding non-degenerate critical points in such a way that we end up with a Morse function $F_v^1 : \overline{\Sigma}_v \to \mathbb{R}$ such that each disk delimited by a critical curve $\gamma_i$ contains exactly one critical point of $F_v^1$ in its interior: a maximum if $\gamma_i$ has the $+$ label, or a minimum if $\gamma_i$ has the $-$ label.
	Moreover, we can assume that in the tubular neighbourhood of any critical curve it is possible to choose cylindrical coordinates $(z,\theta)$ such that $F_v^1$ has the local expression $F_v^1(z, \theta) = z$ in the neighbourhood.

	The constructed Morse function has then $2 + 2g_v$ critical points if $\mathrm{deg}(v) = 1$ or $2 + 2g_v + 2(\mathrm{deg}(v) - 2)$ critical points otherwise.
	This is because the perfect Morse function has the maximum and the minimum in two different bounding disks, and $2g_v$ saddle points.
	For each of the remaining $\mathrm{deg}(v) - 2$ disks, we add one minimum or maximum and a saddle point.
	In the first case, exactly $1+2g_v$ critical points are contained away from the disks delimited by the critical curve, and in the second case $2g_v+\mathrm{deg}(v)-2$.

    We will now use the function $F^1_v$ to construct suitable Hamiltonian dynamics on $\Sigma$. If needed, we divide $F^1_v$ by a constant large enough so that $X^{\overline{\omega}_v}_{F^1_v}$ has no $1$-periodic orbits besides its critical points.

	We now apply Proposition \ref{prop:singularization} and obtain a $b^m$-function $F_v^2$ on the $b^m$-symplectic surface $(\Sigma_v, \omega_v)$ such that $X_{F_v^2}^{\omega_v}$ has exactly the same 1-periodic orbits as $X_{F_v^1}^{\overline{\omega}_v}$.
	In particular, there exists a tubular neighbourhood around each $\gamma_i$ in such a way that $F_v^2$ has the local expression 

	\begin{equation}
		F_v^2(z, \theta) = \begin{cases} \pm \log |z| & \text{if } m = 1 \\ \pm \frac1{z^{m-1}} & \text{if } m > 1 , \end{cases}
			\label{eq:localexpressionA}
	\end{equation}
	with the sign coinciding with the label associated to $\mathcal{N}_\varepsilon(Z_i)\setminus Z_i$ as explained above.

	Let $D_i$ denote the disk delimited by the curve $\gamma_i$, and let us restrict to $\Sigma_v$.
	The number of $1$-periodic orbits of the vector field $X_{F_v^2}^{\omega}$ is $1 + 2g_v$ if $\mathrm{deg}(v) = 1$ and $2g_v + \mathrm{deg}(v) - 2$ if $\mathrm{deg}(v) > 1$.

	To conclude, we construct the $b^m$-function $F^3 \in {}^{b^m}\mathcal{C}^{\infty}(\Sigma)$ by gluing together the components $F_v^2$.
	To do this we need to take into account the parity of $m$. The choice of the coloring of the graph (in the $b^{2m+1}$-symplectic case) or the orientation of the graph (in the $b^{2m}$-symplectic case) assure that we can define $F^3$ by its restriction on each of the connected components of $\Sigma \setminus Z$ so that it coincides with $F^2_v$, and use the expression from Equation (\ref{eq:localexpressionA}) in the tubular neighbourhood of each component of $Z$.

\end{proof}

		\begin{figure}[h]
			\centering
			\begin{tikzpicture}
				\draw (-10, 0) to [out=90,in=180] (-9, 2) to [out=0,in=90] (-8, 0) to [out=270,in=0] (-9, -2) to [out=180,in=270] (-10, 0);
				\draw (-9.05, 1.3) to [out=310, in=50] (-9.05, 0.25);
				\draw (-9, 1.25) to [out=240, in=130] (-9, 0.30);
				\draw (-9.05, -0.25) to [out=310, in=50] (-9.05, -1.3);
				\draw (-9, -0.30) to [out=240, in=130] (-9, -1.25);

				\draw[color=purple] (-9.92, 0.95) to [out=310, in=230] (-9.13, 0.95);
				\draw[color=purple,dashed] (-9.92, 0.95) to [out=50, in=130] (-9.13, 0.95);
				\draw[color=purple] (-8.88, 0.95) to [out=310, in=230] (-8.08, 0.95);
				\draw[color=purple,dashed] (-8.88, 0.95) to [out=50, in=130] (-8.08, 0.95);

				\draw[color=purple] (-9.78, -1.45) to [out=335, in=205] (-8.22, -1.45);
				\draw[color=purple,dashed] (-9.77, -1.45) to [out=15, in=165] (-8.22, -1.45);

				\draw [-{Latex[length=2mm]}] (-7.5,0) to (-6.5,0);

				\draw (-5.78, -1.45) to [out=105, in=270] (-6, 0) to [out=90, in=260] (-5.92, 0.95);
				\draw (-4.22, -1.45) to [out=75, in=270] (-4, 0) to [out=90, in=280] (-4.08, 0.95);
				\draw (-5.13, 0.95) to [out=275, in=130] (-5, 0.30);
				\draw (-4.88, 0.95) to [out=265, in=50] (-5.05, 0.25);
				\draw (-5.05, -0.25) to [out=310, in=50] (-5.05, -1.3);
				\draw (-5, -0.30) to [out=240, in=130] (-5, -1.25);

				\draw [color=blue] (-5.92, 0.95) to [out=310, in=230] (-5.13, 0.95);
				\draw [color=blue,dashed] (-5.92, 0.95) to [out=35, in=145] (-5.13, 0.95);
				\draw [color=blue] (-5.92, 0.95) to [out=85, in=95] (-5.13, 0.95);

				\draw [color=red] (-4.88, 0.95) to [out=310, in=230] (-4.08, 0.95);
				\draw [color=red,dashed] (-4.88, 0.95) to [out=35, in=145] (-4.08, 0.95);
				\draw [color=red] (-4.88, 0.95) to [out=85, in=95] (-4.08, 0.95);
				\draw[color=red] (-5.78, -1.45) to [out=335, in=205] (-4.22, -1.45);
				\draw[color=purple,dashed] (-5.77, -1.45) to [out=15, in=165] (-4.22, -1.45);
				\draw[color=red] (-5.78, -1.45) to [out=305, in=235] (-4.22, -1.45);

				\draw [color=blue] (-6.1, 0.8) node [label={$-$}] {};
				\draw [color=red] (-3.9, 0.8) node [label={$+$}] {};
				\draw [color=red] (-3.9, -2) node [label={$+$}] {};

				\draw (-5.92, 2) to [out=80, in=180] (-5, 3) to [out=0, in=100] (-4.08, 2);
				\draw (-5.18, 2) to [out=85, in=230] (-5, 2.4);
				\draw (-5.05, 2.45) to [out=310, in=95] (-4.82, 2);

				\draw [color=blue] (-5.92, 2) to [out=340, in=200] (-5.18,2);
				\draw [color=blue, dashed] (-5.92, 2) to [out=20, in=160] (-5.18,2);
				\draw [color=blue] (-5.92, 2) to [out=290, in=250] (-5.18,2);

				\draw [color=red] (-4.82, 2) to [out=340, in=200] (-4.08,2);
				\draw [color=red, dashed] (-4.82, 2) to [out=20, in=160] (-4.08,2);
				\draw [color=red] (-4.82, 2) to [out=290, in=250] (-4.08,2);

				\draw [color=blue] (-6.1, 1.7) node [label={$-$}] {};
				\draw [color=red] (-3.9, 1.7) node [label={$+$}] {};

				\draw (-5.8, -2.5) to [out=300, in=240] (-4.2, -2.5);
				\draw [color=red] (-5.8, -2.5) to [out=340, in=200] (-4.2, -2.5);
				\draw [color=red] (-5.8, -2.5) to [out=50, in=130] (-4.2, -2.5);
				\draw [color=red, dashed] (-5.8, -2.5) to [out=10, in=170] (-4.2, -2.5);

				\draw [color=red] (-3.9, -2.8) node [label={$+$}] {};

				\draw [-{Latex[length=2mm]}] (-3.5,0) to (-2.5,0);

				\draw (-2,-1) to (-2, 1);
				\draw (-1.2,-1) to (-1.2, 1);
				\draw [color=red] (-2, 1) to [out=330, in=210] (-1.2, 1);
				\draw [color=red] (-2, 1) to [out=85, in=95] (-1.2, 1);
				\draw [color=red, dashed] (-2, 1) to [out=25, in=155] (-1.2, 1);
				\draw [color=blue, dashed] (-2, -1) to [out=30, in=150] (-1.2, -1);
				\draw [color=blue] (-2, -1) to [out=265, in=275] (-1.2, -1);
				\draw [color=blue] (-2, -1) to [out=335, in=205] (-1.2, -1);

				\draw [color=red] (-1, 0.9) node [label={$+$}] {};
				\draw [color=blue] (-1, -1.7) node [label={$-$}] {};

				\draw (-0.5, 1.5) to [out=270, in=120] (0, 0) to [out=300, in=100] (0.5, -1.7);
				\draw (0.5, 1.5) to [out=280, in=180] (1.3, 0.75) to [out=0, in=280] (2.1, 1.5);
				\draw (3.1, 1.5) to [out=270, in=60] (2.6, 0) to [out=240, in=80] (2.1, -1.7);
				\draw [color=red] (-0.5, 1.5) to [out=330, in=210] (0.5, 1.5);
				\draw [color=red] (-0.5, 1.5) to [out=85, in=95] (0.5, 1.5);
				\draw [color=red, dashed] (-0.5, 1.5) to [out=25, in=155] (0.5, 1.5);
				\draw [color=red] (2.1, 1.5) to [out=330, in=210] (3.1, 1.5);
				\draw [color=red] (2.1, 1.5) to [out=85, in=95] (3.1, 1.5);
				\draw [color=red, dashed] (2.1, 1.5) to [out=25, in=155] (3.1, 1.5);
				\draw [color=blue, dashed] (0.5, -1.7) to [out=30, in=150] (2.1, -1.7);
				\draw [color=blue] (0.5, -1.7) to [out=265, in=275] (2.1, -1.7);
				\draw [color=blue] (0.5, -1.7) to [out=335, in=205] (2.1, -1.7);

				\draw [color=red] (0.6, 1.55) node [label={$+$}] {};
				\draw [color=red] (3.2, 1.55) node [label={$+$}] {};
				\draw [color=blue] (2.3, -2.5) node [label={$-$}] {};
    			\draw (1.35, 0.25) to [out=310, in=50] (1.35, -0.8);
				\draw (1.35, 0.2) to [out=240, in=130] (1.37, -0.75);

				\draw [color=red] (3.5, 0) to [out=90, in=180] (4.5, 1) to [out=0, in=90] (5.5,0);
				\draw [color=red] (3.5, 0) to [out=315, in=225] (5.5,0);
				\draw [color=red, dashed] (3.5, 0) to [out=30, in=150] (5.5,0);
				\draw (3.5, 0) to [out=270, in=180] (4.5, -1) to [out=0, in=270] (5.5,0);

				\draw [color=red] (5.5, 0.5) node [label={$+$}] {};
			\end{tikzpicture}

			\caption{Cutting and filling a $b^{2m}$-symplectic surface with signs at the disks. Red disks contain maxima and blue disks contain minima}
			\label{figure:perfectmorse}
		\end{figure}
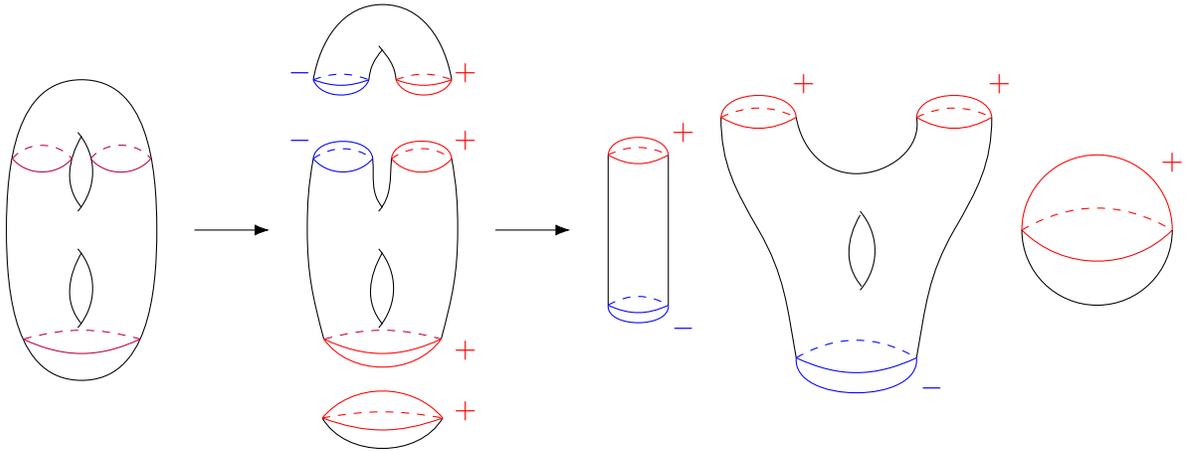

		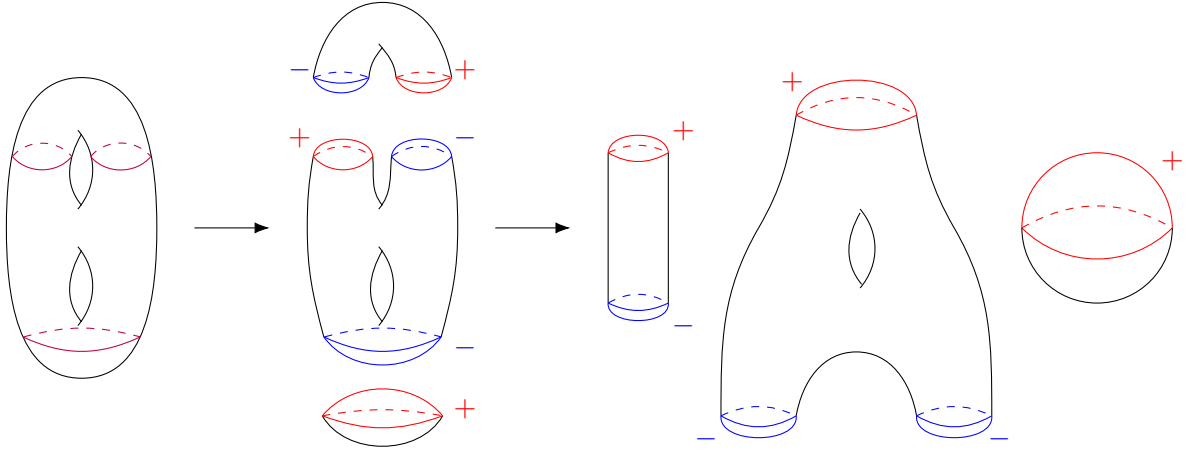
\begin{figure}[h]
			\centering
			\begin{tikzpicture}
				\draw (-10, 0) to [out=90,in=180] (-9, 2) to [out=0,in=90] (-8, 0) to [out=270,in=0] (-9, -2) to [out=180,in=270] (-10, 0);
				\draw (-9.05, 1.3) to [out=310, in=50] (-9.05, 0.25);
				\draw (-9, 1.25) to [out=240, in=130] (-9, 0.30);
				\draw (-9.05, -0.25) to [out=310, in=50] (-9.05, -1.3);
				\draw (-9, -0.30) to [out=240, in=130] (-9, -1.25);

				\draw[color=purple] (-9.92, 0.95) to [out=310, in=230] (-9.13, 0.95);
				\draw[color=purple,dashed] (-9.92, 0.95) to [out=50, in=130] (-9.13, 0.95);
				\draw[color=purple] (-8.88, 0.95) to [out=310, in=230] (-8.08, 0.95);
				\draw[color=purple,dashed] (-8.88, 0.95) to [out=50, in=130] (-8.08, 0.95);

				\draw[color=purple] (-9.78, -1.45) to [out=335, in=205] (-8.22, -1.45);
				\draw[color=purple,dashed] (-9.77, -1.45) to [out=15, in=165] (-8.22, -1.45);

				\draw [-{Latex[length=2mm]}] (-7.5,0) to (-6.5,0);

				\draw (-5.78, -1.45) to [out=105, in=270] (-6, 0) to [out=90, in=260] (-5.92, 0.95);
				\draw (-4.22, -1.45) to [out=75, in=270] (-4, 0) to [out=90, in=280] (-4.08, 0.95);
				\draw (-5.13, 0.95) to [out=275, in=130] (-5, 0.30);
				\draw (-4.88, 0.95) to [out=265, in=50] (-5.05, 0.25);
				\draw (-5.05, -0.25) to [out=310, in=50] (-5.05, -1.3);
				\draw (-5, -0.30) to [out=240, in=130] (-5, -1.25);

				\draw [color=red] (-5.92, 0.95) to [out=310, in=230] (-5.13, 0.95);
				\draw [color=red,dashed] (-5.92, 0.95) to [out=35, in=145] (-5.13, 0.95);
				\draw [color=red] (-5.92, 0.95) to [out=85, in=95] (-5.13, 0.95);

				\draw [color=blue] (-4.88, 0.95) to [out=310, in=230] (-4.08, 0.95);
				\draw [color=blue,dashed] (-4.88, 0.95) to [out=35, in=145] (-4.08, 0.95);
				\draw [color=blue] (-4.88, 0.95) to [out=85, in=95] (-4.08, 0.95);
				\draw[color=blue] (-5.78, -1.45) to [out=335, in=205] (-4.22, -1.45);
				\draw[color=blue,dashed] (-5.77, -1.45) to [out=15, in=165] (-4.22, -1.45);
				\draw[color=blue] (-5.78, -1.45) to [out=305, in=235] (-4.22, -1.45);

				\draw [color=red] (-6.1, 0.8) node [label={$+$}] {};
				\draw [color=blue] (-3.9, 0.8) node [label={$-$}] {};
				\draw [color=blue] (-3.9, -2) node [label={$-$}] {};

				\draw (-5.92, 2) to [out=80, in=180] (-5, 3) to [out=0, in=100] (-4.08, 2);
				\draw (-5.18, 2) to [out=85, in=230] (-5, 2.4);
				\draw (-5.05, 2.45) to [out=310, in=95] (-4.82, 2);

				\draw [color=blue] (-5.92, 2) to [out=340, in=200] (-5.18,2);
				\draw [color=blue, dashed] (-5.92, 2) to [out=20, in=160] (-5.18,2);
				\draw [color=blue] (-5.92, 2) to [out=290, in=250] (-5.18,2);

				\draw [color=red] (-4.82, 2) to [out=340, in=200] (-4.08,2);
				\draw [color=red, dashed] (-4.82, 2) to [out=20, in=160] (-4.08,2);
				\draw [color=red] (-4.82, 2) to [out=290, in=250] (-4.08,2);

				\draw [color=blue] (-6.1, 1.7) node [label={$-$}] {};
				\draw [color=red] (-3.9, 1.7) node [label={$+$}] {};

				\draw (-5.8, -2.5) to [out=300, in=240] (-4.2, -2.5);
				\draw [color=red] (-5.8, -2.5) to [out=340, in=200] (-4.2, -2.5);
				\draw [color=red] (-5.8, -2.5) to [out=50, in=130] (-4.2, -2.5);
				\draw [color=red, dashed] (-5.8, -2.5) to [out=10, in=170] (-4.2, -2.5);

				\draw [color=red] (-3.9, -2.8) node [label={$+$}] {};

				\draw [-{Latex[length=2mm]}] (-3.5,0) to (-2.5,0);

				\draw (-2,-1) to (-2, 1);
				\draw (-1.2,-1) to (-1.2, 1);
				\draw [color=red] (-2, 1) to [out=330, in=210] (-1.2, 1);
				\draw [color=red] (-2, 1) to [out=85, in=95] (-1.2, 1);
				\draw [color=red, dashed] (-2, 1) to [out=25, in=155] (-1.2, 1);
				\draw [color=blue, dashed] (-2, -1) to [out=30, in=150] (-1.2, -1);
				\draw [color=blue] (-2, -1) to [out=265, in=275] (-1.2, -1);
				\draw [color=blue] (-2, -1) to [out=335, in=205] (-1.2, -1);

				\draw [color=red] (-1, 0.9) node [label={$+$}] {};
				\draw [color=blue] (-1, -1.7) node [label={$-$}] {};

				\draw (-0.5, -2.5) to [out=90, in=240] (0, 0) to [out=60, in=260] (0.5, 1.5);
				\draw (0.5, -2.5) to [out=80, in=180] (1.3, -1.65) to [out=0, in=100] (2.1, -2.5);
				\draw (3.1, -2.5) to [out=90, in=300] (2.6, 0) to [out=120, in=280] (2.1, 1.5);
				\draw [color=blue] (-0.5, -2.5) to [out=330, in=210] (0.5, -2.5);
				\draw [color=blue] (-0.5, -2.5) to [out=265, in=275] (0.5, -2.5);
				\draw [color=blue, dashed] (-0.5, -2.5) to [out=25, in=155] (0.5, -2.5);
				\draw [color=blue] (2.1, -2.5) to [out=330, in=210] (3.1, -2.5);
				\draw [color=blue] (2.1, -2.5) to [out=265, in=275] (3.1, -2.5);
				\draw [color=blue, dashed] (2.1, -2.5) to [out=25, in=155] (3.1, -2.5);
				\draw [color=red,dashed] (0.5, 1.5) to [out=30, in=150] (2.1, 1.5);
				\draw [color=red] (0.5, 1.5) to [out=85, in=95] (2.1, 1.5);
				\draw [color=red] (0.5, 1.5) to [out=335, in=205] (2.1, 1.5);

				\draw [color=red] (0.45, 1.55) node [label={$+$}] {};
				\draw [color=blue] (-0.7, -3.2) node [label={$-$}] {};
				\draw [color=blue] (3.2, -3.2) node [label={$-$}] {};
    			\draw (1.35, 0.25) to [out=310, in=50] (1.35, -0.8);
				\draw (1.35, 0.2) to [out=240, in=130] (1.37, -0.75);

				\draw [color=red] (3.5, 0) to [out=90, in=180] (4.5, 1) to [out=0, in=90] (5.5,0);
				\draw [color=red] (3.5, 0) to [out=315, in=225] (5.5,0);
				\draw [color=red, dashed] (3.5, 0) to [out=30, in=150] (5.5,0);
				\draw (3.5, 0) to [out=270, in=180] (4.5, -1) to [out=0, in=270] (5.5,0);

				\draw [color=red] (5.5, 0.5) node [label={$+$}] {};
			\end{tikzpicture}

			\caption{Cutting and filling a $b^{2m+1}$-symplectic surface with signs at the disks. Red disks contain maxima and blue disks contain minima}
			\label{figure:perfectmorseodd}
		\end{figure}

\section{A lemma in graph theory}

This section of the appendix includes a lemma in graph theory\footnote{The authors could not find a reference in the classic graph theory literature. We are immensely grateful to Juanjo Rué for helpful discussions and comments concerning this part of the appendix.}, that we use in the proof of Lemma \ref{prop:optimalbound}. We follow the notation of \cite{diestel}.

\begin{definition}
     Let $G=(V,E)$ be an undirected graph. An orientation of the graph $G$  is good if for every vertex $v$ with $\mathrm{deg}(v) > 1$, there exists at least one edge incident to $v$, whose initial vertex is $v$, and at least one edge incident to $v$, whose terminal vertex is $v$.
\end{definition}

    \begin{minipage}{.5\textwidth}
        \centering
        \begin{tikzpicture}[node distance={15mm}, thick, main/.style = {draw, circle}]
            \node[main] (1) {$x_1$};
            \node[main] (2) [right of=1] {$x_2$};
            \node[main] (3) [right of=2] {$x_3$};
            \draw[->] (1) -- (2);
            \draw[->] (2) -- (3);
        \end{tikzpicture}
        
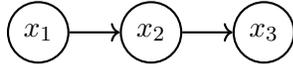
\captionof{figure}{A good orientation of the graph.}
    \end{minipage}
    \begin{minipage}{.4\textwidth}
        \centering
        \begin{tikzpicture}[node distance={15mm}, thick, main/.style = {draw, circle}]
            \node[main] (1) {$x_1$};
            \node[main] (2) [right of=1] {$x_2$};
            \node[main] (3) [right of=2] {$x_3$};
            \draw[->] (1) -- (2);
            \draw[<-] (2) -- (3);
        \end{tikzpicture}
        
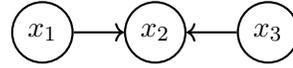
\captionof{figure}{Not a good orientation.}
    \end{minipage}

\begin{lemma}\label{lem:goodorientation}
    Let $G=(V,E)$ be a graph. Then there exists a good orientation of the graph.
\end{lemma}

\begin{proof}
    An orientation of a graph is good if it is good for each connected component. Hence, we may assume that $G$ is a connected graph.

    Given the connected graph $G=(V,E)$, consider an ordering of the edges, given by $(e_i)_{i=1,\dots,n}$. We inductively define a subgraph as follows: we initially define $G_0=G$. For $i>0$, $G_i$ is then defined to be the graph given by $G_{i}=(V,E_{i})$, where 
    
    $E_i:=E_{i-1}\setminus\{e_i\}$ \quad if $e_i$ is an edge incident to a vertex of degree strictly greater than $2$ in the graph $G_{i-1}$, and $E_i:=E_{i-1}$ otherwise.
    
    We thus end up with a graph that does not have any vertex of degree strictly greater than $2$.

    Assume we can give $G_{i+1}$ a good orientation. We claim that then $G_i$ can be given a good orientation.
    Consider the edge $e_i$. If $e_i$ is contained in $E_{i+1}$, $G_i=G_{i+1}$ and thus $G_i$ admits a good orientation. If $e_i$ is contained in $E_i\setminus E_{i+1}$, by definition, $e_i$ is incident with at least one vertex that has degree greater than $2$. We adopt the orientation of $E_{i+1}$ for all the edges of $E_i\setminus \{e_i\}$ and will choose a direction for $e_i$.
    There are $3$ possibilities for the other incident vertex to $e_i$:
    \begin{itemize}
        \item the other incident vertex is of degree greater than $2$ in $G_{i+1}$. In this case, both incident vertices of $e_i$ have thus degree greater than $2$. In that case, any direction of the edge $e_i$ is good, because both vertices have already a good direction in $E_{i+1}$.
        \item the other incident vertex $v$ is of degree $1$ in $G_{i+1}$, then we choose the ``opposite" direction, that is: if the other edge that $v$ has, has initial vertex given by $v$ (respectively has terminal vertex given by $v$), we choose the direction of $e_i$ such that its initial vertex (respectively terminal vertex) is $v$.
        \item the other incident vertex is of degree $0$ in $G_{i+1}$. In this case, any direction can be chosen for $e_i$.
    \end{itemize}

    It is therefore sufficient to prove the lemma for the subgraph $G_n$. By definition of the subgraph $G_n$, there are no edges of degree strictly greater than $2$, therefore there are only a limited number of possibilities. A connected component of the graph $G_n$ is given by
    \begin{itemize}
        \item a path,
        \item a cycle,
        \item or an isolated vertex.
    \end{itemize}
    All of these can be given a trivial good orientation. This thus proves that $G_n$ can be given a good orientation. Thus, proceeding by induction, $G=G_0$ can be given a good orientation.
\end{proof}

\newpage


\end{document}